\documentclass[11pt]{article}  
\usepackage{amsmath}
\usepackage{amssymb}
\usepackage{theorem}
\usepackage{euscript}
\usepackage{pstricks}
	
\usepackage{xcolor}
\usepackage[linesnumbered,ruled,vlined]{algorithm2e}

\usepackage{authblk}
\title{Another algorithm template}
\author{names}

\usepackage{algpseudocode}

\algnewcommand{\algorithmicand}{\rm{\bf \ and\ }}
\algnewcommand{\algorithmicor}{\rm{\bf \ or    \ }}
\algnewcommand{\OR}{\algorithmicor}
\algnewcommand{\AND}{\algorithmicand}
\algnewcommand{\var}{\texttt}

\SetCommentSty{mycommfont}

\SetKwInput{KwInput}{Input}                
\SetKwInput{KwOutput}{Output}              

\newtheorem{theorem}{Theorem}[section]
\newtheorem{lemma}{Lemma}[section]
\newtheorem{corollary}{Corollary}[section]
\newtheorem{remark}{Remark}[section]
\numberwithin{equation}{section}

\newcommand{\be}{{\boldsymbol{e}}}

\newcommand{\bm}{{\boldsymbol{m}}}

\newcommand{\bs}{{\boldsymbol{s}}}

\newcommand{\bx}{{\boldsymbol{x}}}

\newcommand{\by}{{\boldsymbol{y}}}

\newcommand{\brho}{{\boldsymbol{\rho}}}
\newcommand{\bsigma}{{\boldsymbol{\sigma}}}
\newcommand{\rd}{{\rm d}} 
\newcommand{\rev}{{\rm ev}} 
 
\newcommand{\bzero}{\boldsymbol{0}}
\def\II{\mathbb I}
\def\EE{\mathbb E}

\def\RR{{\mathbb R}}

\def\NN{{\mathbb N}}

\def\II{{\mathbb I}}

\def\NN{{\mathbb N}}
\def\RR{{\mathbb R}}

\def\Vv{{\mathbb P}}
\def\UU{{\mathbb U}}
\def\Vv{{\mathcal P}}

\def\IIi{{\mathbb I}^\infty}

\def\RRi{{\mathbb R}^\infty}

\def\RRi{{\mathbb R}^\infty}

\def\UUi{{\mathbb U}^{\infty}}

\def\Ee{{\mathcal E}}

\def\Hh{{\mathcal H}}
\def\Ii{{\mathcal I}}

\def\Ll{{\mathcal L}}
\def\Oo{{\mathcal O}}

\def\Qq{{\mathcal Q}}
\def\Ss{{\mathcal S}}
\def\Tt{{\mathcal T}}

\def\Vv{{\mathcal V}}

\def\II{{\mathbb I}}

\def\NN{{\mathbb N}}
\def\RR{{\mathbb R}}

\def\FF{{\mathbb F}}

\def\supp{\operatorname{supp}}

\def\Wi{W^1_\infty(D)}

\title{\sffamily Sparse-grid polynomial interpolation approximation and integration for parametric and stochastic elliptic PDEs with lognormal inputs} 
\author{ 
	Dinh D\~ung \\
	Information Technology Institute, Vietnam National University \\
	144 Xuan Thuy, Cau Giay, Hanoi, Vietnam\\
	{\ttfamily dinhzung@gmail.com} \\
}

\date{\today}
\begin{document}
\maketitle

\begin{abstract}
	By combining a certain  approximation property in the spatial domain, and weighted $\ell_2$-summability  of the Hermite polynomial expansion coefficients  in the parametric domain  obtained in [M. Bachmayr, A. Cohen, R. DeVore and G. Migliorati, ESAIM Math. Model. Numer. Anal. {\bf 51}(2017), 341-363] and [M. Bachmayr, A. Cohen, D. D\~ung and C. Schwab,  SIAM J. Numer. Anal. {\bf 55}(2017), 2151-2186], we investigate  linear non-adaptive methods of fully discrete polynomial interpolation approximation as well as fully discrete weighted quadrature methods of integration for parametric and stochastic elliptic PDEs with lognormal inputs.  We  construct  such methods and prove convergence rates of the approximations by them.  The linear non-adaptive methods of fully discrete polynomial interpolation approximation  are sparse-grid collocation methods which are  certain sums taken over finite nested Smolyak-type indices sets of mixed tensor products of dyadic scale successive differences of spatial approximations of particular solvers, and of  successive differences of  their parametric Lagrange interpolating polynomials. The  Smolyak-type sparse interpolation grids in the parametric domain are constructed from the roots of Hermite polynomials or their improved modifications. Moreover, they generate in a natural way fully discrete weighted quadrature formulas for integration of the solution to parametric and stochastic elliptic PDEs and its linear functionals, and the error of the  corresponding integration can be estimated via the error in the Bochner space $L_1(\RRi,V,\gamma)$  norm of the generating methods where $\gamma$ is  the Gaussian probability measure on $\RRi$ and $V$ is the energy space. 
We also briefly  consider similar problems for parametric and stochastic elliptic PDEs with affine inputs, and  problems of non-fully discrete polynomial interpolation approximation and integration. In particular, the convergence rates of non-fully discrete polynomial interpolation approximation and integration obtained in this paper significantly improve the known ones.


  
\end{abstract}

\section{Introduction}

One of basic problems in Uncertainty Quantification are approximation and numerical integration for parametric and stochastic PDEs.  Since the number of parametric variables may be very large or even infinite,  they are  treated as high-dimensional or infinite-dimensional approximation problems.  
Let $D \subset {\mathbb R}^d$ be a bounded  Lipschitz domain.  Consider the diffusion elliptic equation 
\begin{equation} \label{ellip}
- {\rm div} (a\nabla u)
\ = \
f \quad \text{in} \quad D,
\quad u|_{\partial D} \ = \ 0, 
\end{equation}
for  a given fixed right-hand side $f$ and spatially variable scalar diffusion coefficient $a$.
Denote by $V:= H^1_0(D)$ the energy space and let $V' = H^{-1}(D)$ be the conjugate space of $V$. If $a$ satisfies the ellipticity assumption
\begin{equation} \nonumber
0<a_{\min} \leq a \leq a_{\max}<\infty,
\end{equation}
by the well-known Lax-Milgram lemma, for any $f \in V'$, there exists a unique solution $u \in V$ in weak form which satisfies the variational equation
\begin{equation} \nonumber
\int_{D} a\nabla u \cdot \nabla v \, \rd \bx
\ = \
\langle f , v \rangle,  \quad \forall v \in V.
\end{equation}
We consider diffusion coefficients having a parameterized form $a=a(\by)$, where $\by=(y_j)_{j \in \NN}$
is a sequence of real-valued parameters ranging in the set $\UUi$ which is either  $\RRi$ or $\IIi :=[-1,1]^\infty$. 
In this case, the  solution $u(\by)$ to the parameterized  diffusion elliptic equation 
\begin{equation} \label{p-ellip}
- {\rm div} (a(\by)\nabla u(\by))
\ = \
f \quad \text{in} \quad D,
\quad u(\by)|_{\partial D} \ = \ 0, 
\end{equation}	
	can be considered as a  map
$\by \mapsto u(\by)$ from $\UUi$ to the space $V$. The objective is to achieve numerical 
approximation of this complex map by a small number of parameters with some
guaranteed error in a given norm. Depending on the nature of the modeled object, the parameter $\by$ may be either deterministic or random.
In the present paper, we consider the  so-called lognormal case when $\UUi= \RRi$ and the diffusion coefficient $a$  is of the form
\begin{equation} \label{lognormal}
a(\by)=\exp(b(\by)), \quad b(\by)=\sum_{j = 1}^\infty y_j\psi_j,
\end{equation}
where the $y_j$ are i.i.d. standard Gaussian random variables  and $\psi_j \in L_\infty(D)$. We also briefly consider the affine case 
when $\UUi= \IIi$ and the diffusion coefficient $a$  is of the form
\begin{equation} \label{affine}
a(\by)= \bar a + \sum_{j = 1}^\infty y_j\psi_j.
\end{equation}

In order to study fully discrete approximations of the solution $u(\by)$ to the parameterized elliptic PDEs \eqref{ellip}, we assume that $f \in L_2(D)$ and $a(\by) \in W^1_\infty(D)$, and hence we obtain
 that  $u(\by)$  has the second higher regularity, i.e.,  $u(\by) \in W$ where $W$ is the space
\begin{equation} \nonumber
W:=\{v\in V\; : \; \Delta v\in L^2(D)\}
\end{equation}
equipped with the norm
\begin{equation} \nonumber
\|v\|_{W}:=\|\Delta v\|_{L^2(D)},
\end{equation}
which coincides with the Sobolev space $V\cap H^2(D)$ with equivalent norms if the domain $D$ has $C^{1,1}$ smoothness, see \cite[Theorem 2.5.1.1]{Gr85}.
Moreover, we assume that there holds the following approximation property for the spaces $V$ and $W$.

\noindent
{\bf Assumption I}  \,  
There are a sequence $(V_n)_{n \in \NN_0}$ of subspaces $V_n \subset V $ of dimension $\le m$, and a sequence $(P_n)_{n \in \NN_0}$ of linear operators from $V$ into $V_n$, 
and a number $\alpha>0$ such that
\begin{equation} \label{spatialappn}
\|P_n(w)\|_V \leq 
C\|w\|_V , \quad
\|w-P_n(w)\|_V \leq 
Cn^{-\alpha} \|w\|_W, \quad \forall n \in \NN_0, \quad \forall w \in W.
\end{equation}

 A basic role in the approximation and numerical integration for parametric and stochastic PDEs are generalized polynomial chaos (GPC) expansions for the dependence on the parametric variables.  We refer the reader to \cite{CD15,DKS13,GWZ14,SG11,NTW08} and references there for different aspects in approximation for parametric and stochastic PDEs. In \cite{CCS15}--\cite{CDS11}, based on the conditions
$\big(\|\psi_j\|_{\Wi}\big)_{j \in \NN} \in \ell_p(\NN)$ for some $0 < p <1$ on the affine expansion \eqref{affine}, the authors have proven the $\ell_p$-summability of the coefficients  in a Taylor or Legendre polynomials expansion and hence proposed best adaptive  $n$-term methods of Galerkin and collocation approximations in energy norm  by choosing  the set of the $n$ most useful terms in these expansions. To derive a fully discrete approximation  the best  $n$-term approximants are then  approximated by finite element methods. Similar results have been received in \cite{HoS14} for Galerkin approximation in the lognormal case based  on the conditions $\big(j\|\psi_j\|_{\Wi}\big)_{j \in \NN} \in \ell_p(\NN)$ for some $0 < p <1$. In these papers, they did not take into account support properties of the functions $\psi_j$.

A different approach to studying summability   that takes into account the 
support properties has been recently proposed in \cite{BCM17} for the affine case
and \cite{BCDM17} for the lognormal case. This approach leads to significant improvements
on the results on $\ell_p$-summability when the functions $\psi_j$ have limited overlap, such as splines, finite elements or wavelet bases. These results by themselves do not imply practical applications, because they do not cover the approximation of the expansion coefficients which are functions of the spatial variable. 

In the recent paper \cite{BCDS17},  
the rates of fully discrete adaptive best $n$-term Taylor, Jacobi and Hermite polynomial approximations for elliptic PDEs with  affine or lognornal  parametrizations of the diffusion coefficients have been obtained  based on  combining a certain approximation property on the spatial domain, and extensions of the results on $\ell_p$-summability of \cite{BCM17,BCDM17} to higher-order Sobolev norms of corresponding Taylor, Jacobi and Hermite expansion coefficients. These results  providing a benchmark for convergence rates, are not constructive. In the case when $\ell_p$-summable sequences of Sobolev norms of expansion coefficients have  an $\ell_p$-summable majorant sequence, these convergence rates can be achieved by  linear methods of GPC expansion and collocation approximations  in the affine case \cite{CD15,Di15,Di18a,Di18b,Ze18,ZDS18}.  However,  this non-adaptive approach is not applicable for the improvement  of  $\ell_p$-summability in \cite{BCM17,BCDM17,BCDS17} since the weakened  $\ell_p$-assumption leads only to the $\ell_p$-summability of  expansion coefficients, but not to  an $\ell_p$-summable majorant sequence. Non-adaptive non-fully discrete methods have been considered in \cite{EST18} for polynomial collocation approximation, and in \cite{Ch18} for weighted integration (see also \cite[Remark 3.2]{BCM17} and \cite[Remark 5.1]{BCDM17} for  briefly considering non-adaptive non-fully discrete approximations by truncated GPC expansions).

Let us briefly describe the main contribution of  the present paper.  By combining spatial and parametric aproximability, namely, the   approximation property in Assumptions I  in the spatial domain and weighted $\ell_2$-summability of the $V$ and $W$ norms  of Hermite polynomial expansion coefficients  obtained in \cite{BCDM17,BCDS17}, we investigate  linear non-adaptive methods of fully discrete approximation by truncated Hermite GPC expansion and polynomial interpolation approximation as well as fully discrete weighted quadrature methods of integration for parametric and stochastic elliptic PDEs with lognormal inputs \eqref{lognormal}.  We  construct  such methods and prove convergence rates of the approximations by them.  We show that the convergence  rate in terms of the dimension of the approximation space of adaptive  fully discrete  approximation by truncated Hermite GPC expansion  obtained in \cite{BCDS17}, is achieved by linear non-adaptive methods of fully discrete approximation by truncated Hermite GPC expansion approximation. The linear non-adaptive methods of fully discrete polynomial interpolation approximation  are sparse-grid collocation methods which are  certain sums taken over  finite nested Smolyak-type indices sets  of tensor products of  dyadic scale successive differences of spatial approximations of particular solvers, and of   successive differences of their parametric Lagrange interpolating polynomials. The Smolyak-type sparse interpolation grids in the parametric domain are constructed from
 the roots of Hermite polynomials or their improved modifications. 
 Moreover, these methods generate in a natural way fully discrete weighted quadrature formulas  for integration of the solution $u(\by)$ and its linear functionals, and the error of the  corresponding integration can be estimated via the error in the space $L_1(\RRi,V,\gamma)$ norm of the generating methods where $\gamma$ is  the Gaussian probability measure on $\RRi$. 
  The convergence rate of  fully discrete integration is better than the convergence rate of  the generating fully discrete polynomial interpolation approximation due to  the simple but useful observation  that the integral $\int_{\RR} v(y) \, \rd \gamma (y)$ is zero if $v(y)$ is an odd function and $\gamma$ is the Gaussian probability measure on $\RR$. (This property has been used in \cite{Ze18,ZDS18,ZS17} for improving convergence rate of integration in the affine case.)
 We also briefly  consider similar problems for parametric and stochastic elliptic PDEs with affine inputs \eqref{affine} by using  counterparts-results in \cite{BCM17,BCDS17}, and  problems of non-fully discrete polynomial interpolation approximation and integration similar to those treated in \cite{EST18} and \cite{Ch18}. In particular, the convergence rates of non-fully discrete interpolation approximation and integration in terms of number of the evaluation points obtained in this paper, are significantly better than those which have been proven in \cite{EST18} and in \cite{Ch18}. 

 Finally, let us notice that the aim of this paper is to establish approximation results which should  show posibilities of non-adaptive approximation methods and  convergence rates of approximation by such methods for the parameterized diffusion elliptic equation \eqref{p-ellip} with lognormal inputs. The two most popular numerical methods are Galerkin projection and collocation. Since  in the lognormal case, the diffusion coefficient $a(\by)(\bx)$ is not uniformly bounded in $\by \in \RRi$, there is no a well-posed linear variational problem on the space $L_2(\RRi,V,\gamma)$ for Galerkin approximation. Some best $n$-term Galerkin approximations with respect to a "stronger" Gaussian measure $\gamma_{\varrho}$  were considered in \cite{HoS14}. Collocation methods will be discussed in a forthcoming paper which extends the results in the affine case \cite{ZDS18,ZS17}.

The paper is organized as follows. 
In Sections~\ref{Linear Galerkin approximation}--\ref{Integration},  we  construct  general linear fully discrete  and non-fully discrete methods of Hermite GPC expansion  and polynomial interpolation approximations in the Bochner space $L_p(\RRi,X^1, \gamma)$,  and quadrature  of functions taking values in $X^2$ and having a weighted $\ell_2$-summability of Hermite expansion coefficients  for Hilbert spaces $X^1$ and $X^2$ satisfying a certain ``spatial" approximation property (see \eqref{spatialappnX}). In particular,
in Section~\ref{Linear Galerkin approximation}, we prove convergence rates of general linear fully discrete  methods of approximation approximation by truncated Hermite GPC expansion; in Section~\ref{interpolation}, we prove convergence rates of general linear fully discrete and non-fully discrete polynomial interpolation methods of approximation; in Section~\ref{Integration}, we prove convergence rates of general linear fully discrete and non-fully discrete quadrature for integration. In Section~\ref{lognormal inputs}, we apply the results of Sections~\ref{Linear Galerkin approximation}--\ref{Integration} to obtain the main results of this paper on convergence rates of  linear non-adaptive methods of fully discrete approximation by truncated Hermite GPC expansion, and fully discrete  and non-fully discrete polynomial interpolation approximation and weighted  quadrature methods of integration for parametric and stochastic elliptic PDEs with lognormal inputs. In Section~\ref{affine inputs}, by extending the theory in Sections~\ref{Linear Galerkin approximation}--\ref{Integration}, we briefly consider similar problems for  parametric and stochastic elliptic PDEs with affine inputs. 

\section{Linear approximation by truncated Hermite GPC expansion }
\label{Linear Galerkin approximation}

 In this section, we treat a general linear fully discrete approximation by truncated Hermite GPC series of functions from the Bochner space $L_2(\RRi,X^2, \gamma)$. The approximation error is measured in the  Bochner space $L_p(\RRi,X^1, \gamma)$  with $0 <p \le 2$. Here, $X^1$ and $X^2$  are Hilbert spaces, and $\gamma$ is the infinite tensor product Gaussian probability measure. 
We  construct linear (non-adaptive) methods of this approximation and  prove convergence rates for the approximation error.

We first recall a concept of infinite tensor product of probability measures. (For details see, e.g., \cite[pp. 429--435]{HS65}.)
Let $\mu(y)$ be a probability measure on $\UU$, where $\UU$ is either $\RR$ or $\II:= [-1,1]$.
We introduce the probability measure $\mu(\by)$ on $\UUi$ as the infinite tensor product of probability measures $\mu(y_i)$:
\begin{equation} \nonumber
\mu(\by) 
:= \ 
\bigotimes_{j \in \NN} \mu(y_j) , \quad \by = (y_j)_{j \in \NN} \in \UUi.
\end{equation}
The sigma algebra for $\mu(\by)$ is generated by the set of cylinders $A:= \prod_{j \in \NN} A_j$, where $A_j \subset \UU$ are univariate $\mu$-measurable sets and only a finite number of $A_i$ are different from $\UU$. For such a set $A$, we have $\mu(A) = \prod_{j \in \NN} \mu(A_j)$.
If $\varrho (y)$ is the density of $\mu(y)$,  i.e., $\rd \mu(y) = \varrho(y) \rd y $, then we  write
\begin{equation} \nonumber
\mbox{d} \mu(\by) 
:= \ 
\bigotimes_{j \in \NN} \varrho(y_j) \rd (y_j), \quad \by = (y_j)_{j \in \NN} \in \UUi.
\end{equation}

Let $X$ be a Hilbert space and  $0 < p <\infty$. 
The probability measure $\mu(\by)$ induces  the Bochner space $L_p(\UUi,X,\mu)$ of strongly $\mu$-measurable mappings $v$ from $\UUi$ to $X$ which are  $p$-summable. The (quasi-)norm in $L_p( \UUi,X,\mu)$ is defined by
\begin{equation} \nonumber
\|v\|_{L_p( \UUi,X,\mu)}
:= \
\left(\int_{ \UUi} \|v(\cdot,\by)\|_X^p \, \mbox{d} \mu(\by) \right)^{1/p}.
\end{equation} 
  Notice that if $X$ is separable, $L_2( \UUi,X,\mu)$ is the tensor product of the Hilbert spaces  $L_2( \UUi,\RR,\mu)$ and $X$.
 
In the present paper, we focus our attention mainly to the lognormal case with $\UUi = \RRi$ and $\mu(\by) = \gamma(\by)$, the infinite tensor product  Gaussian probability measure.  
Let $\gamma(y)$ be the probability measure on $\RR$ with the standard Gaussian density: 
\begin{equation} \label{g}
\rd\gamma(y):=g(y)\,\rd y, \quad g(y):=\frac 1 {\sqrt{2\pi}} e^{-y^2/2}.
\end{equation}
Then the infinite tensor product Gaussian probability measure $\gamma(\by)$ on $\RRi$ can be defined by
\begin{equation} \nonumber
\rd \gamma(\by) 
:= \ 
\bigotimes_{j \in \NN} g(y_j) \rd (y_j), \quad \by = (y_j)_{j \in \NN} \in \RRi.
\end{equation}

A powerful strategy for the approximation of functions $v$  in $L_2(\RRi,X,\gamma)$ is based on the truncation of the Hermite GPC expansion 
\begin{equation} \label{series}
v(\by)=\sum_{\bs\in\FF} v_\bs \,H_\bs(\by), \quad v_\bs \in X.
\end{equation}
Here $\FF$ is the set of all sequences of non-negative integers $\bs=(s_j)_{j \in \NN}$ such that their support $\supp (\bs):= \{j \in \NN: s_j >0\}$ is a finite set,
and 
\begin{equation} \nonumber
H_\bs(\by)=\bigotimes_{j \in \NN}H_{s_j}(y_j),\quad v_\bs:=\int_{\RRi} v(\by)\,H_\bs(\by)\, \rd\gamma (\by), \quad \bs \in \FF,
\label{hermite}
\end{equation}
with $(H_k)_{k\geq 0}$ being the  Hermite polynomials normalized according to
$\int_{\RR} | H_k(y)|^2\, g(y)\, \rd y= 1.$
It is well-known that $(H_\bs)_{\bs \in \FF}$ is an orthonormal basis of $L_2(\RRi,\RR, \gamma)$. Moreover, for every $v \in L_2( \RRi,X,\gamma)$  represented by the series \eqref{series}  it holds Parseval's identity
\begin{equation} \nonumber
\|v\|_{L_2( \RRi,X,\gamma)}^2
\ = \ \sum_{\bs\in\FF} \|v_\bs\|_X^2.
\end{equation}

We make use  of the abbreviations: $L_p(\RRi,\mu):=  L_p(\RRi,\RR,\mu)$;   $\Ll_p(X):=  L_p(\RRi,X,\gamma)$ for $0< p < \infty$.  We use letter $C$ to denote a general positive  constant which may take different values, and $C_{p,q,\alpha,D,...}$  a constant depending on $p,q,\alpha,D,...$.

 To construct general linear fully discrete  methods of approximation   in the Bochner space $\Ll_p(X^1)$  and of integration of functions taking values in $X^2$, we need the following  assumption on approximation property for $X^1$ and $X^2$, which is a generalization of Assumption I.

\noindent
{\bf Assumption II}  \,  The Hilbert space $X^2$ is  a linear subspace of the Hilbert space $X^1$ and that $\|\cdot\|_{X^1} \le C\, \|\cdot\|_{X^2}$.
There are a sequence $(V_n)_{n \in \NN_0}$ of subspaces $V_n \subset X^1 $ of dimension $\le n$, and a sequence $(P_n)_{n \in \NN_0}$ of linear operators from $X^1$ into $V_n$, 
and a number $\alpha>0$ such that
\begin{equation} \label{spatialappnX}
\|P_n(w)\|_{X^1} \leq 
C\|w\|_{X^1} , \quad
\|w-P_n(w)\|_{X^1} \leq 
Cn^{-\alpha} \|w\|_{X^2}, \quad \forall n \in \NN_0, \quad \forall w \in X^2.
\end{equation}

 For $k \in \NN_0$,  we define
\begin{equation} \nonumber
\delta_k (w)
:= \
P_{2^k} (w)  - P_{2^{k-1}} (w), \ k \in \NN, \quad \delta_0 (w) = P_0 (w).
\end{equation}
Under Assumption II, we can represent every $w \in X^2$  by the series
\begin{equation} \nonumber
w
\ = \
\sum_{k =0}^\infty \delta_k (w)
\end{equation}
converging in $X^1$ and satisfying the estimates
\begin{equation} \label{delta-approx-propertyX}
\|\delta_k (w)\|_{X^1}
\ \le \
C\,\, \|w\|_{X^1}, \quad	
\|\delta_k (w)\|_{X^1}
\ \le \
 C\, 2^{-\alpha k} \, \|w\|_{X^2}, \quad k \in \NN_0.
\end{equation}

For a  finite subset $G$ in $ \NN_0 \times \FF$, 
denote by $\Vv(G)$ the subspace in $\Ll_2(X^1)$ of  all functions $v$
of the form
\begin{equation} \nonumber
v
\ = \
\sum_{(k,\bs) \in G} v_k \, H_\bs, \quad v_k \in V_{2^k}.
\end{equation}
Let Assumption II hold for Hilbert spaces $X^1$ and $X^2$. We define the linear operator $\Ss_G: \, \Ll_2(X^2) \to \Vv(G)$ by
\[
\Ss_G v
:= \
\sum_{(k,\bs) \in G} \delta_k (v_\bs) \, H_\bs
\]
for $v \in \Ll_2(X^2)$ represented by the series
\begin{equation} \label{HermiteSeriesV}
v = \sum_{\bs\in\FF} v_\bs H_\bs, \quad v_\bs \in X^2.
\end{equation}

\begin{lemma}\label{lemma[G_K-Conv]}
	Let Assumption II hold for Hilbert spaces $X^1$ and $X^2$.
	Then for every $v \in \Ll_2(X^2)$, 
	\begin{equation} \label{lim-S_Gk}
	\lim_{K \to \infty}
	\|v-  \Ss_{G_K}v \|_{\Ll_2(X^1)}
	\ = \
	0,
	\end{equation}
	where $G_K := \{(k,\bs) \in  \NN_0 \times \FF: \, 0 \le k \le K\}$.
\end{lemma}

\begin{proof}
	Obviously, by the definition,  
	\begin{equation} \nonumber
	\Ss_{G_K}v
	\ = \  \sum_{\bs \in \FF} \sum_{k=0}^K \delta_k (v_\bs) \, H_\bs
	\ = \
	\sum_{\bs \in \FF} P_{2^K}(v_\bs) \, H_\bs.
	\end{equation} 
	From Parseval's identity and \eqref{spatialappnX}   it follows that
	\begin{equation} \nonumber
	\begin{split}
	\|\Ss_{G_K}v\|_{\Ll_2(X^1)}^2
	\ &= \
	\sum_{\bs \in \FF} \|P_{2^K}(v_\bs)\|_{X^1}^2
	\ \le \
	2 \sum_{\bs \in \FF} \|v_\bs\|_{X^1}^2 + 2 \sum_{\bs \in \FF}\|v_\bs -P_{2^K}(v_\bs)\|_{X^1}^2 \\
	\ &\le \
	2 \sum_{\bs \in \FF} \|v_\bs\|_{X^1}^2 + 2C^2 2^{-\alpha K} \sum_{\bs \in \FF}\|v_\bs \|_{X^2}^2 < \infty.
	\end{split}
	\end{equation} 
	This means that $\Ss_{G_K}v \in \Ll_2(X^1)$.  Hence, by  
	Parseval's identity and  \eqref{spatialappnX} we deduce that
	\begin{equation} \nonumber
	\|v- \Ss_{G_K}v\|_{\Ll_2(X^1)}^2
	\ = \
	\sum_{\bs \in \FF}\|v_\bs -P_{2^K}(v_\bs)\|_{X^1}^2
	\ \le \
	C^2 \, 2^{-2\alpha K} \sum_{\bs \in \FF} \|v_\bs\|_{X^2}^2
	\ = \
	C^2 \, 2^{-2\alpha K} \|v\|_{\Ll_2(X^2)}^2
	\end{equation} 
	which  proves the lemma.
	\hfill
\end{proof}	

Define for $\xi>1$
\begin{equation} \label{G(xi)G}
	G(\xi)
	:= \ 
	\begin{cases}
		\big\{(k,\bs) \in \NN_0 \times\FF: \, 2^k \sigma_{2;\bs}^{q_2} \leq \xi\big\} \quad &{\rm if }  \ \alpha \le 1/q_2;\\
		\big\{(k,\bs) \in \NN_0 \times\FF: \, \sigma_{1;\bs}^{q_1} \le \xi, \  
		2^{\alpha q_1 k} \sigma_{2;\bs}^{q_1} \leq \xi\big\} \quad  & {\rm if }  \ \alpha > 1/q_2.
	\end{cases}
\end{equation}

\begin{theorem}\label{thm[L_2-approx]}
	Let $0 < p \le 2$.
	 Let Assumption II hold for Hilbert spaces $X^1$ and $X^2$.
	Let $v \in \Ll_2(X^2)$ be represented by the series \eqref{HermiteSeriesV}.
	Assume that for $r=1,2$ there exist sequences $(\sigma_{r;\bs})_{\bs \in \FF}$  of numbers strictly larger than $1$ such that 
	\begin{equation} \label{weighted-summ}
	\sum_{\bs\in\FF} (\sigma_{r;\bs} \|v_\bs\|_{X^r})^2 <\infty
	\end{equation}		
	and 
	$(\sigma_{r;\bs}^{-1})_{\bs \in \FF} \in \ell_{q_r}(\FF)$  for some $0<q_1\leq q_2 <\infty$.	
		Then there exists a constant $C$ such that for each $n \in \NN$, there exists a number $\xi_n$  such that   $\dim(\Vv(G(\xi_n)) \le n$ and
	\begin{equation} 	\label{L_2-rate}
	\|v-\Ss_{G(\xi_n)}v\|_{\Ll_p(X^1)} \leq C \times
	\begin{cases}
		n^{-\alpha} &\text{if } \alpha \leq  1/q_2,
		\\
		n^{-\beta}  &\text{if } \alpha >  1/q_2.
	\end{cases} 
	\end{equation}
	The rate $\alpha$ corresponds to the approximation of a single function in $X^2$ as given by \eqref{spatialappnX}, and
	the rate $\beta$ is given by
\begin{equation} 	\label{[beta]2}
\beta := \frac 1 {q_1}\frac{\alpha}{\alpha + \delta}, \quad 
\delta := \frac 1 {q_1} - \frac 1 {q_2}.
\end{equation}
\end{theorem}

\begin{proof}
Due to the inequality
	$\|\cdot\|_{\Ll_p(X^1)} \le  \|\cdot\|_{\Ll_2(X^1)}$,  it is sufficient to prove the theorem for $p=2$.

{\it We first consider the case $\alpha \le 1/q_2$.} Let $\xi >0 $ be given and  take  arbitrary positive number  $\varepsilon$. Since $G(\xi)$ is finite, from the definition of 
$G_K$ and Lemma~\ref{lemma[G_K-Conv]} it follows that there exists $K = K(\xi,\varepsilon)$ such that 
$G(\xi) \subset G_K$ and 
\begin{equation}\label{ineq1}
\|v- \Ss_{G_K}v\|_{\Ll_2(X^1)} 
\ \le \ 
\varepsilon.
\end{equation}
By the triangle inequality,
\begin{equation} \label{ineq2}
\|v- \Ss_{G(\xi)} v\|_{\Ll_2(X^1)}
\ \le \ 
\|v- \Ss_{G_K}v\|_{\Ll_2(X^1)}
\ + \
\|\Ss_{G_K}v - \Ss_{G(\xi)} v\|_{\Ll_2(X^1)}.
\end{equation} 
We have by Parseval's identity and \eqref{delta-approx-propertyX} that
\begin{equation} \nonumber
\begin{split}
\|\Ss_{G_K}v - \Ss_{G(\xi)} v\|_{\Ll_2(X^1)}^2
\ &= \
\Big\|\sum_{\bs \in \FF}  \sum_{k=0}^K \delta_k (v_\bs) \, H_\bs - 
\sum_{\bs \in \FF}  \sum_{2^k \sigma_{2;\bs}^{q_2} \le \xi} \delta_k (v_\bs) \, H_\bs\Big\|_{\Ll_2(X^1)}^2 
\\[1.5ex]
\ &= \
\Big\|\sum_{\bs \in \FF}  \ \sum_{\xi \sigma_{2;\bs}^{-q_2}<2^k \le 2^K} \delta_k (v_\bs) \, H_\bs\Big\|_{\Ll_2(X^1)}^2 
\ = \
\sum_{\bs \in \FF}  \ \Big\|\sum_{\xi \sigma_{2;\bs}^{-q_2}<2^k \le 2^K} \delta_k (v_\bs)\Big\|_{X^1}^2 
\\[1.5ex]
\ &\le \
\sum_{\bs \in \FF}  \ \Big(\sum_{\xi \sigma_{2;\bs}^{-q_2}<2^k \le 2^K} \|\delta_k (v_\bs)\|_{X^1}\Big)^2
\ \le \
\sum_{\bs \in \FF}  \ \Big(\sum_{\xi \sigma_{2;\bs}^{-q_2}<2^k \le 2^K} C \, 2^{-\alpha k}\|v_\bs\|_{X^2}\Big)^2 
\\[1.5ex]
\ &\le \
C \, \sum_{\bs \in \FF} \|v_\bs\|_{X^2}^2 \ \Big(\sum_{2^k >\xi \sigma_{2;\bs}^{-q_2}}  2^{-\alpha k}\Big)^2
\ \le \
C \, \sum_{\bs \in \FF} \|v_\bs\|_{X^2}^2 \ (\xi \sigma_{2;\bs}^{-q_2})^{-2\alpha}. 
\end{split}
\end{equation}
Hence, by  the inequalities  $q_2\alpha \le 1$ and $\sigma_{2;\bs} >1$, and \eqref{weighted-summ} we derive that
\begin{equation} \nonumber
\begin{split}
\|\Ss_{G_K}v - \Ss_{G(\xi)} v\|_{\Ll_2(X^1)}^2
\ &\le \
C \, \, \sum_{\bs \in \FF} 
(\sigma_{2;\bs} \|v_\bs\|_{X^2})^2 
\ = \
\ C \, \xi^{- 2\alpha}.
\end{split}
\end{equation}
Since $\varepsilon > 0$ is arbitrary, from the last estimates and \eqref{ineq1} and \eqref{ineq2} we derive that
\begin{equation} \label{ineq3}
\|v- \Ss_{G(\xi)} v\|_{\Ll_2(X^1)}
\ \le \
C\, \xi^{- \alpha}.
\end{equation}
For the dimension of the space $\Vv(G(\xi))$ we have by \eqref{weighted-summ} that
 \begin{equation} \label{dimVGxi}
 \begin{split}
 \dim \Vv(G(\xi))
 \ &\le \ \sum_{(k,\bs) \in G(\xi)} \dim V_{2^k} 
 \ \le \ \sum_{(k,\bs) \in G(\xi)} 2^k
 \\[1.5ex]
 \ &\le  \
 \sum_{\bs \in \FF} \quad 
 \sum_{2^k \  \le \xi\sigma_{2;\bs}^{-q_2}} 2^k 
 \  \le \
 2 \sum_{\bs \in \FF} \xi\sigma_{2;\bs}^{-q_2} 
 \ = \
  M\,\xi,
 \end{split}
 \end{equation} 
 where $M:= 2 \big\|\big(\sigma_{2;\bs}^{-1}\big)\big\|_{\ell_{q_2}(\FF)}^{q_2}$.
 For any $n \in \NN$, letting $\xi_n$ be a number satisfying the inequalities 
 \begin{equation} \label{[xi_n]1}
M\, \xi_n
 \ \le \
 n
 \ < \ 2M\, \xi_n,
 \end{equation}
 we derive that
 $ \dim \Vv(G(\xi_n)) \le  n$.
 On the other hand, from \eqref{[xi_n]1} it follows that
 $
 \xi_n^{-\alpha} 
 \ \le \ (2M)^\alpha \, n^{-\alpha}.
$
 This together with \eqref{ineq3} proves that 
\begin{equation} \label{case<1/q_2}
\|v- \Ss_{G(\xi_n)} v\|_{\Ll_2(X^1)}
\ \le \
C\, n^{- \alpha}, \quad \alpha \le 1/q_2 .
\end{equation} 
 
 {\it We now consider the case $\alpha > 1/q_2$.} Putting
 \[
 v_\xi
 :=
 \sum_{\sigma_{1;\bs}^{q_1} \le \xi} v_\bs H_\bs,
 \]
 we get
 \begin{equation} \nonumber
 \|v- \Ss_{G(\xi)} v\|_{\Ll_2(X^1)}
 \ \le \
 \|v- v_\xi\|_{\Ll_2(X^1)} + \|v_\xi - \Ss_{G(\xi)} v\|_{\Ll_2(X^1)}.
 \end{equation}
 The norms in the right hand side can be estimated using Parseval's identity and the hypothesis of the theorem.
 Thus, for the norm $\|v- v_\xi\|_{\Ll_2(X^1)}$ we have by \eqref{weighted-summ} that
 \begin{equation} \label{ineq:|v-v_xi|}
 \begin{split}
 \|v- v_\xi\|_{\Ll_2(X^1)} ^2
 \ &= \
 \sum_{\sigma_{1;\bs}^{q_1} > \xi} \|v_\bs\|_{X^1}^2 
 \ = \
 \sum_{\sigma_ {1;\bs} > \xi^{1/q_1} } (\sigma_{1;\bs}\|v_\bs\|_{X^1})^2 \sigma_{1;\bs}^{-2}
 \\[1.5ex]
 \ &\le \
 \xi^{-2/q_1}
 \sum_{\sigma_{1;\bs}> \xi^{1/q_1} } (\sigma_{1;\bs}\|v_\bs\|_{X^1})^2 
 \le C \xi^{-2/q_1}.
 \end{split}
 \end{equation}
 For the norm $\|v_\xi - \Ss_{G(\xi)} v\|_{\Ll_2(X^1)}$, with 
 $N= N(\xi,\bs):= 2^{\lfloor\log_2 (\xi^{1/q_1\alpha} \sigma_{2;\bs}^{-1/\alpha})\rfloor}$ we obtain 
 \begin{equation} \nonumber
 \begin{split}
 \|v_\xi - \Ss_{G(\xi)} v\|_{\Ll_2(X^1)}^2
 \ &= \
 \sum_{\sigma_{1;\bs}> \xi^{1/q_1} } \Big\| v_\bs - 
 \sum_{2^{\alpha q_1 k} \sigma_{2;\bs}^{q_1} \le  \xi}  \delta_k (v_\bs) \Big\|_{X^1}^2
 \ \le \
 \sum_{\bs \in \FF} \Big\| v_\bs - 
 P_N (v_\bs) \Big\|_{X^1}^2
 \\[1.5ex]
 \ &\le \
 \sum_{\bs \in \FF} C N^{-2\alpha} \| v_\bs \|_{X^2}^2
 \le C \xi^{-2/q_1}
 \sum_{\bs \in \FF} (\sigma_{2;\bs}\| v_\bs \|_{X^2})^2
 \le C \xi^{-2/q_1}.
 \end{split}
 \end{equation}
 These estimates yield that
 \begin{equation} \label{case>1/q_2(1)}
 \begin{split}
 \|v- \Ss_{G(\xi)} v\|_{\Ll_2(X^1)}
 \  \le \ C \xi^{-1/q_1}.
 \end{split}
 \end{equation}
 
For the dimension of the space $\Vv(G(\xi))$, 
with $q:= q_2 \alpha > 1$ and $1/q' + 1/q = 1$ we have that
\begin{equation} \nonumber
\begin{split}
\dim \Vv(G(\xi))
\ &\le \ \sum_{(k,\bs) \in G(\xi)} \dim V_{2^k} 
\ \le \ 
\sum_{\sigma_{1;\bs}^{q_1} \le \xi} \quad 
\sum_{2^{\alpha q_1 k} \sigma_{2;\bs}^{q_1} \leq \xi} 2^k 
\\ \  &\le \
2 \sum_{\sigma_{1;\bs}^{q_1} \le \xi} \xi^{1/(q_1 \alpha)}\sigma_{2;\bs}^{-1/\alpha} 
\ \le \
2 \xi^{1/(q_1 \alpha)}\left(\sum_{\sigma_{1;\bs}^{q_1} \le \xi}\sigma_{2;\bs}^{-q_2} \right)^{1/q}
\left(\sum_{\sigma_{1;\bs}^{q_1} \le \xi} 1 \right)^{1/q'}
\\ \  &\le 
2 \xi^{1/(q_1 \alpha)}\left(\sum_{\bs \in \FF} \sigma_{2;\bs}^{-q_2} \right)^{1/q}
\left(\sum_{\bs \in \FF} \xi \sigma_{1;\bs}^{-q_1}  \right)^{1/q'}
\ =  \
M \xi^{1 + \delta/\alpha},
\end{split}
\end{equation}
where
$
M:=  2 \big\|\big(\sigma_{2;\bs}^{-1}\big)\big\|_{\ell_{q_2}(\FF)}^{q_2/q} \,
 \big\|\big(\sigma_{1;\bs}^{-1}\big)\big\|_{\ell_{q_1}(\FF)}^{q_1/q'}.
$
For any $n \in \NN$, letting $\xi_n$ be a number satisfying the inequalities 
\begin{equation} \label{[xi_n]2}
M\, \xi_n^{1 + \delta/\alpha}
\ \le \
n
\ < \ 2M\, \xi_n^{1 + \delta/\alpha},
\end{equation}
we derive that  $ \dim \Vv(G(\xi_n)) \le  n$.
On the other hand, by \eqref{[xi_n]2},
\begin{equation} \nonumber
\xi_n^{-1/q_1} 
\ \le \  (2M)^{\frac{\alpha}{\alpha + \delta}}\, n^{-\frac{1}{q_1} \frac{\alpha}{\alpha + \delta}}.
\end{equation}
This together with \eqref{case>1/q_2(1)} proves that
\begin{equation} \nonumber
\|v- \Ss_{G(\xi_n)} v\|_{\Ll_2(X^1)}
\  \le \ C n^{-\beta}, \quad \alpha > 1/q_2.
\end{equation}
By combining the last estimate and \eqref{case<1/q_2} we obtain  \eqref{L_2-rate}.
\hfill
\end{proof}

\begin{remark}
{\rm		
	 Let us compare the non-adaptive fully discrete method constructed in Theorem \ref{thm[L_2-approx]}  with adaptive one considered in \cite[Theorem 3.1]{BCDS17}. Both the methods give the same convergence rate $\min(\alpha,\beta)$. However, the ways to form them are different. Let us explain this.
		
	In  \cite[Theorem 3.1]{BCDS17} for the lognormal case, for a given $v \in X^2$,  a preliminary polynomial approximation $v_m := \sum_{\bs \in \Lambda_m}v_\bs H_\bs$ is taken by  truncation of the Hermite GPC expansion \eqref{HermiteSeriesV}, where $\Lambda_m \subset \FF$ is a set corresponding to $m$ largest $\|v_\bs\|_{X^1}$.  The coefficients $v_\bs \in X^2$ then is approximated by 
	$v_{\bs, m_\bs}:= P_{m_\bs} (v_\bs)$, and $v$ is approximated by the resulting approximant $v_\bm := \sum_{\bs \in \Lambda_m} v_{\bs, m_\bs} H_\bs$. An optimal choice of $\big(m_\bs\big)_{\bs \in \Lambda_m}$ give the rate $\min(\alpha,\beta)$ in terms of $n$ where $n = \sum_{\bs \in \Lambda_m} m_\bs$ is the dimension of the space $\Vv_\bm \subset X^2$ of all functions of the form $\sum_{\bs \in \Lambda_m}v_\bs H_\bs$, $v_\bs \in V_{m_\bs}$. This is an adaptive approximation method, since the choice of $m$ largest $\|v_\bs\|_{X^1}$ essentially depends of $v$.
	
	  In Theorem \ref{thm[L_2-approx]}, the approximant  $\Ss_{G(\xi_n)} v$ belongs to the space  $\Vv(G(\xi_n)) \subset \Vv(X^2)$.  The convergence rate $\min(\alpha,\beta)$ of approximation by $\Ss_{G(\xi_n)} v$  is given in terms of  $n$  where  the shareholding parameter $\xi_n$ is chosen such that  $\dim(\Vv(G(\xi_n)) \le n$.  
	 Notice that  $\Ss_{G(\xi_n)} v = \sum_{\bs \in \Lambda} v_{\bs, m_\bs} H_\bs$ and the space $\Vv(G(\xi_n))$ consists of all functions of the form $\sum_{\bs \in \Lambda}v_\bs H_\bs$, $v_\bs \in V_{m_\bs}$, for a certain set $\Lambda$ depending on $n$, i.e., formally they are similar to those in \cite[Theorem 3.1]{BCDS17}. The  difference here is that the set $\Lambda$ is  defined  independently of $v$. Hence, our approximation methods are non-adaptive provided that 
	 there is a sequence $(P_n)_{n \in \NN_0}$ of linear operators from $X^1$ into $n$-dimensional subspaces $V_n \subset X^1 $ satisfying  \eqref{spatialappnX} for all $v \in X^2$ (Assumption II). See also \cite[Remark 3.2]{BCDS17}.
} 
\end{remark}

\section{Polynomial interpolation  approximation}
\label{interpolation}

In this section, we construct general linear fully discrete polynomial interpolation methods of approximation  in the Bochner space $\Ll_p(X^1)$ of functions taking values in $X^2$ and having a weighted $\ell_2$-summability of Hermite expansion coefficients  for Hilbert spaces $X^1$ and $X^2$ satisfying a certain ``spatial" approximation property. In particular,
we prove convergence rates for these methods of approximation. We also briefly consider linear non-fully discrete polynomial interpolation methods of approximation. 

\subsection{Auxiliary results}

Let $w = \exp(-Q)$, where $Q$ is an even function on $\RR$ and $yQ'(y)$ is positive and increasing in $(0,\infty)$, with limits $0$ and $\infty$ at $0$ and $\infty$.  Notice that for the standard Gaussian density $g$ defined in \eqref{g}, $\sqrt{g}$ is such a function.
For $m \in \NN$, the $n$th Mhaskar-Rakhmanov-Saff number $a_m= a_m(w)$ is defined as the positive root of the equation
\begin{equation} \nonumber
m
\ = \
\frac{2}{\pi} \int_0^1 \frac{a_myQ'(a_my)}{\sqrt{1 - y^2}} \, \rd y.
\end{equation}
From \cite[Page 11]{Lu07}  we have for  $w=\sqrt{g}$,
\begin{equation} \label{a_m(g)}
a_m(\sqrt{g})
\ = \
\sqrt{m}.
\end{equation}
For $0 < p,q \le \infty$, we introduce the quantity
\begin{equation} \nonumber
\delta (p,q)
:= \ \frac 1 2 \left|\frac{1}{p} - \frac{1}{q}\right|.
\end{equation}

\begin{lemma} \label{lemma[N-ineq]}
	Let $0 < p,q \le \infty$. Then  there exists a positive constant $C_{p,q}$ such that for every polynomial $\varphi$ of degree $\le m$,  the Nikol'skii-type inequality holds
	\begin{equation} \nonumber
	\|\varphi \sqrt{g} \|_{L_p( \RR)}
	\ \le \
	C_{p,q} m^{\delta(p,q)}\, \|\varphi \sqrt{g} \|_{L_q( \RR)}.
	\end{equation}
\end{lemma}

\begin{proof} This lemma is an immediate consequence of \eqref{a_m(g)}  and the inequality 
	\begin{equation} \nonumber
	\|\varphi \sqrt{g} \|_{L_p( \RR)}
	\ \le \
	C_{p,q}N_m(p,q)\, \|\varphi \sqrt{g} \|_{L_q( \RR)}
	\end{equation}
	which follows from  \cite[Theorem 9.1, p. 61]{Lu07},
	where
	\begin{equation} \nonumber
	N_m(p,q)
	:= \
	\begin{cases}
	a_m^{1/p - 1/q}, \quad & p < q, \\
	\left(\frac{n}{a_m}\right)^{1/q - 1/p}, & p > q. 
	\end{cases}
	\end{equation}		
	\hfill
\end{proof}

\begin{lemma} \label{lemma[H_bs|_infty]}
	We have 
	\begin{equation} \label{eq[H_bs|_infty]}
	\|H_m\sqrt{g}\|_{L_\infty(\RR)}
	\ \le \
	1, \quad m \in \NN_0.	
	\end{equation}
\end{lemma}

\begin{proof}
	From Cram\'er's  bound (see, e.g., \cite[Page 208, (19)]{EMOT55}) we have  for every  $m \in \NN_0$ and every $x \in \RR$, $|H_m(x) \sqrt{g(x)}| \le K (2\pi)^{-1/4}$, where $K:= 1.086435$. This implies \eqref{eq[H_bs|_infty]}.
	\hfill
\end{proof}

For our application  the estimate \eqref{eq[H_bs|_infty]} is sufficient, see \cite{DILP18} for an improvement.

For $\theta, \lambda \ge 0$, we define the sequence
\begin{equation} \label{[p_s]}
p_\bs(\theta,\lambda) := \prod_{j \in \NN} (1 + \lambda s_j)^\theta, \quad \bs \in \FF.
\end{equation}

\begin{lemma}\label{lemma[AbsConv]}
	Let $0 < p \le 2$ and $X$ be a Hilbert space. 
	Let $v \in L_2(\RRi,X,\gamma)$ be represented by the series \eqref{series}.
	Assume that there exists a sequence  $\bsigma =(\sigma_\bs)_{\bs \in \FF}$ of positive numbers such that
	\begin{equation} \nonumber
	\sum_{\bs\in\FF} (\sigma_\bs \|v_\bs\|_{X})^2 <\infty.
	\end{equation}
	We have the following.
	\begin{itemize}
		\item[{\rm (i)}] If $	\left(p_\bs(\theta,\lambda) \sigma_\bs^{-1}\right)_{\bs \in \FF} \in \ell_{q}(\FF)$
		for some $0< q \le 2$ and $\theta, \lambda \ge 0$,
		then	$\left(p_\bs(\theta,\lambda) \|v_\bs\|_{X}\right)_{\bs \in \FF} \in \ell_{{\bar q}}(\FF)$ for ${\bar q}$ such that  $1/{\bar q} = 1/2 + 1/q$. 
		
		\item[{\rm (ii)}] 	
		If $\left(\sigma_\bs^{-1}\right)_{\bs \in \FF} \in \ell_{q}(\FF)$ 	for some $0< q \le 2$, then the series \eqref{series} converges absolutely in  $\Ll_p(X)$  to $v$.
	\end{itemize}
\end{lemma}
\begin{proof}Since $\tau:= 2/{\bar q} \ge  1$, with $1/\tau + 1/\tau' = 1$ and $p_\bs = p_\bs(\theta,\lambda)$ by the H\"older inequality we have that
	\begin{equation} \nonumber
	\begin{split}	 
	\sum_{\bs\in\FF} \left(p_\bs \|v_\bs\|_{X}\right)^{{\bar q}} 
	&\leq 
	\bigg( \sum_{\bs\in \FF} (\sigma_\bs^{{\bar q}} \|v_\bs\|_{X}^{{\bar q}})^\tau\bigg)^{1/\tau} 
	\bigg(\sum_{\bs\in \FF} \left(p_\bs^{{\bar q}} \sigma_\bs^{-{\bar q}}\right)^{\tau'} \bigg)^{1/\tau'} 
	\\[1.5ex]
	&=
	\bigg( \sum_{\bs\in \FF} (\sigma_\bs \|v_\bs\|_{X})^2\bigg)^{{\bar q}/2} 
	\bigg(\sum_{\bs\in \FF} \left(p_\bs\sigma_\bs^{-1}\right)^{q} \bigg)^{1-{\bar q}/2} <\infty\,.
	\end{split}
	\end{equation}	
	This proves the assertion (i).
	
	We have by the inequality ${\bar q} \le 1$ and (i) for $\theta = \lambda = 0$,
	\begin{equation} \nonumber
	\begin{split}
	\sum_{\bs\in\FF} \|v_\bs H_\bs\|_{\Ll_2(X)}
	\ &= \
	\sum_{\bs\in\FF} \|v_\bs\|_{X} \|H_\bs\|_{L_2(\RRi,\gamma)}
	\\[1ex]
	\ &\le \
	\sum_{\bs\in\FF} \|v_\bs\|_{X} 
	\ \le \
	\left(\sum_{\bs\in\FF} (\|v_\bs\|_{X} )^{{\bar q}}\right)^{1/{\bar q}}
	\ < \ \infty.
	\end{split}
	\end{equation}
	This yields that the series \eqref{series} absolutely converges in $\Ll_2(X)$ to $v$, since by the assumption this series converges in $\Ll_2(X)$ to $v$. The assertion (ii) is proven for the case $p=2$.  The case $0 < p < 2$ is derived from the case $p =2$ and the inequality $\|\cdot\|_{\Ll_p(X)} \le \|\cdot\|_{\Ll_2(X)}$.
	\hfill
\end{proof}

\begin{lemma}\label{lemma[AbsConv]2}
	Let $0 < p \le 2$.
	Let Assumption II hold for Hilbert spaces $X^1$ and $X^2$, and let the assumptions of Lemma \ref{lemma[AbsConv]}(ii) hold for the space  $X^1$.
	Then every $v \in \Ll_2(X^2)$ can be represented as the series 
	\begin{equation} \label{[delta-series]}
	v
	\ = \
	\sum_{(k,\bs) \in \NN_0 \times \FF} \delta_k (v_\bs)\, H_\bs
	\end{equation}
	converging absolutely  in $\Ll_p(X^1)$ to $v$. 
\end{lemma}

\begin{proof} 
	This lemma can be proven in a way similar to the proof of \cite[Lemma 2.1]{Di18b}. For completeness, let us give a detailed proof.
	As in the proof of Lemma~\ref{lemma[AbsConv]}, it is sufficient to prove the lemma for the case $p=2$.
	Put $v_{k,\bs} (\by)(\bx):= \delta_k (v_\bs)(\bx) \, H_\bs(\by)$. It is well known  that  the unconditional convergence in a Banach space follows from the absolute convergence. Using this fact,  from  Lemma \ref{lemma[AbsConv]}(ii) and Assumption II we derive that on one hand, the series 
	$
	\sum_{\bs \in \FF} v_{k,\bs}(\by)(\bx)
	$
	converges unconditionally in  $\Ll_2(X^1)$, and uniformly for $k \in \NN_0$, and  on the other hand, the series 
	$
	\sum_{k \in \NN_0} v_{k,\bs}(\by)(\bx)
	$
	converges absolutely in  $\Ll_2(X^1)$,  and uniformly for $s \in \FF$ to $v_\bs(\bx) \, H_\bs(\by)$. Hence, since the series \eqref{HermiteSeriesV} converges unconditionally in   $\Ll_2(X^1)$, we have that
	\begin{equation} \nonumber
	v(\by)(\bx)
	\ = \
	\sum_{\bs \in \FF} v_\bs(\bx) H_\bs(\by)
	\ = \
	\sum_{\bs \in \FF} \sum_{k \in \NN_0} v_{k,\bs} (\by)(\bx)
	\ = \
	\sum_{(k,\bs) \in \NN_0 \times \FF} v_{k,\bs} (\by)(\bx), \quad \bx \in D, \ \by \in \RRi,
	\end{equation}
	where the last series converges unconditionally in  $\Ll_2(X^1)$.  	This means that the series in \eqref{[delta-series]} converges absolutely to $v$, since  by Lemma~\ref{lim-S_Gk} the sum $S_{G_K}$ converges in $\Ll_2(X^1)$ to $v$ when $K \to \infty$. 
	\hfill 
\end{proof}

 Define for $\xi>1$,
\begin{equation} \label{G_(xi)}
	G(\xi)
	:= \ 
	\begin{cases}
		\big\{(k,\bs) \in \NN_0 \times\FF: \, 2^k \sigma_{2;\bs}^{q_2} \leq \xi\big\} \quad &{\rm if }  \ \alpha \le 1/q_2 - 1/2;\\
		\big\{(k,\bs) \in \NN_0 \times\FF: \, \sigma_{1;\bs}^{q_1} \le \xi , \  
		2^{\tau k} \sigma_{2;\bs}\leq \xi^\vartheta \big\} \quad  & {\rm if }  \ \alpha > 1/q_2 - 1/2,
	\end{cases}
\end{equation}
where
\begin{equation} \label{kappa}
	\tau:= \frac{2\alpha}{2 - q_2}, \qquad		
	\vartheta:= \frac{2}{2 - q_2} \left(\frac{1}{q_1} - \frac{1}{2}\right).
\end{equation}	

We will need the following two lemmata for application in estimating the convergence rate of the fully discrete polynomial interpolation approximation in this section and of integration in Section~\ref{Integration}.

\begin{lemma}\label{lemma[L_p-approx]}
	Under the hypothesis of Theorem~\ref{thm[L_2-approx]}, assume in addition that $q_1 < 2$.
		Then there exists a constant $C$ such that for each  $\xi >1$,
	\begin{equation} 	\label{L_p-rate}
	\|v-\Ss_{G(\xi)}v\|_{\Ll_p(X^1)} 
	\ \leq \
	C\times  
	\begin{cases}
	\xi^{-\alpha}
 \quad &{\rm if }  \ \alpha \le 1/q_2 - 1/2;\\
\xi^{-(1/q_1 - 1/2)}
\quad  & {\rm if }  \ \alpha > 1/q_2 - 1/2.
\end{cases}	
	\end{equation}
	The rate $\alpha$ is  given by \eqref{spatialappnX}.
\end{lemma}

\begin{proof}Similarly to the proof of Lemma~\ref{lemma[AbsConv]}, it is sufficient to prove the lemma for $p=2$.
	Since in the case $\alpha \le 1/q_2  - 1/2$, the formulas \eqref{G(xi)G} and \eqref{G_(xi)} define the same set $G(\xi)$ for $\xi >1$, from \eqref{ineq3} follows the lemma for this case.	
Let us consider the case $\alpha > 1/q_2 - 1/2$. Putting
	\[
	v_\xi
	:=
	\sum_{\sigma_{1;\bs}^{q_1} \le \xi} v_\bs H_\bs,
	\]
	we get
	\begin{equation} \label{triangle-ineq}
	\|v- \Ss_{G(\xi)} v\|_{\Ll_2(X^1)}
	\ \le \
	\|v- v_\xi\|_{\Ll_2(X^1)} + \|v_\xi - \Ss_{G(\xi)} v\|_{\Ll_2(X^1)}.
	\end{equation}
  As in the proof of Lemma \ref{lemma[AbsConv]2},  by Lemma~\ref{lemma[AbsConv]}(ii) the series \eqref{series} converges unconditionally in $\Ll_2(X^1)$ to $v$.
	Hence  the norm $\|v- v_\xi\|_{\Ll_2(X^1)}$  can be estimated by
	\begin{equation} \label{|u-u_xi|}
	\begin{split}
	\|v- v_\xi\|_{\Ll_2(X^1)} 
	\ &\le \
	\sum_{\sigma_{1;\bs}> (\xi )^{1/q_1}}
	\|v_\bs\|_{X^1} \, \|H_\bs\|_{L_2(\RR^\infty,\gamma)}
	\ = \
	\sum_{\sigma_{1;\bs}> (\xi )^{1/q_1}}  
	\|v_\bs\|_{X^1} 
	\\[1.5ex]
	\ &\le \
	\left(\sum_{\sigma_{1;\bs}> (\xi )^{1/q_1}}
	 (\sigma_{1;\bs}\|v_\bs\|_{X^1})^2\right)^{1/2}
	\left(\sum_{\sigma_{1;\bs}> (\xi )^{1/q_1}} \sigma_{1;\bs}^{-2}\right)^{1/2}
	\\[1.5ex]
	\ &\le \
	C\,
	\left(\sum_{\sigma_{1;\bs}> (\xi )^{1/q_1}}
	\sigma_{1;\bs}^{-q_1} 
	\sigma_{1;\bs}^{-(2- q_1)}\right)^{1/2}
	\\[1.5ex]
	\ &\le \
	C \xi^{-(1/q_1 - 1/2)}
	\left(\sum_{\bs \in \FF} \sigma_{1;\bs}^{-q_1} \right)^{1/2}
	\le C (\xi )^{-(1/q_1 - 1/2)}.
	\end{split}
	\end{equation}

For the norm $\|v_\xi - \Ss_{G(\xi)} v\|_{\Ll_2(X^1)}$, with 
$N= N(\xi,\bs):= 2^{\big\lfloor\log_2\big(  \sigma_{2;\bs}^{-1/\tau} 
	\xi^{\vartheta/ \tau}\big)\big\rfloor}$ 
we have 
\begin{equation} \nonumber
	\begin{split}
		\|v_\xi - \Ss_{G(\xi)} v\|_{\Ll_2(X^1)}
		\ &\le \
		\sum_{\sigma_{1;\bs}^{q_1}  \le  \xi }
		\Big\| v_\bs - \sum_{2^k \le \sigma_{2;\bs}^{-1/\tau}  \xi^{\vartheta/\tau}}  \delta_k (v_\bs) \Big\|_{X^1} \, \|H_\bs\|_{L_2(\RR^\infty,\gamma)}
		\\[1.5ex]
		\ &= \
		C\, \sum_{\sigma_{1;\bs}^{q_1}  \le  \xi }  \Big\| v_\bs - P_N (v_\bs) \Big\|_{X^1} 
		\ \le \
		C\, \sum_{\sigma_{1;\bs}^{q_1}  \le  \xi } N^{-\alpha} \| v_\bs \|_{X^2} 
		\\[1.5ex]
		\ &\le \
		C\, \sum_{\sigma_{1;\bs}^{q_1}  \le  \xi } 
		(\sigma_{2;\bs}^{-1/\tau}  \xi^{\vartheta/\tau})^{-\alpha} \| v_\bs \|_{X^2} 
		\ \le \ 
		C\, \xi^{- \vartheta \alpha/\tau}
		\sum_{\sigma_{1;\bs}^{q_1}  \le  \xi }\sigma_{2;\bs}^{\alpha/\tau}\| v_\bs \|_{X^2} 
		\\[1.5ex]
		\ &\le \ 
		C\, \xi^{- \vartheta \alpha/\tau}
		\left(\sum_{\sigma_{1;\bs}^{q_1}  \le  \xi } (\sigma_{2;\bs}\|v_\bs\|_{X^2})^2\right)^{1/2}
		\left(\sum_{\sigma_{1;\bs}^{q_1}  \le  \xi } \sigma_{2;\bs}^{2(\alpha/\tau -1)} \right)^{1/2}
		\\[1.5ex]
		\ &= \ 
		C\, \xi^{-(1/q_1 - 1/2)}
		\left(\sum_{\sigma_{1;\bs}^{q_1}  \le  \xi } \sigma_{2;\bs}^{- q_2} \right)^{1/2}
		\ \le \ 
			C\, \xi^{-(1/q_1 - 1/2)}.
	\end{split}
\end{equation} 
Here we used the equalities $\vartheta \alpha/\tau = 1/q_1 - 1/2$ and $2(1 - \alpha/\tau) = q_2$.
Summing up,  we find 
\begin{equation} \nonumber
	\begin{split}
		\|v_\xi - \Ss_{G(\xi)} v\|_{\Ll_2(X^1)}
		\le
		C\, \xi^{- (1/q_1 - 1/2)}.
	\end{split}
\end{equation}
	This, (3.8) and (3.9) prove the lemma for  the case $\alpha > 1/q_2  - 1/2$. 
\hfill
\end{proof}

We make use the notation: $\FF_\rev := \{\bs \in \FF: s_j \ {\rm  even}, \ j \in \NN \}$.  Define for $\xi>1$, 
\begin{equation} \label{G_rev(xi)}
	G_\rev(\xi)
	:= \ 
	\begin{cases}
		\big\{(k,\bs) \in \NN_0 \times\FF_\rev: \, 2^k \sigma_{2;\bs}^{q_2} \leq \xi\big\} \ &{\rm if }  \ \alpha \le 1/q_2 - 1/2;\\
		\big\{(k,\bs) \in \NN_0 \times\FF_\rev: \, \sigma_{1;\bs}^{q_1} \le \xi , \  
		2^{\tau k} \sigma_{2;\bs}\leq \xi^\vartheta \big\} \  & {\rm if }  \ \alpha > 1/q_2 - 1/2,
	\end{cases}
\end{equation}
where $\tau$, $\vartheta$  are as in \eqref{kappa}.	

The following lemma can be proven in a similar way.

\begin{lemma}\label{lemma[V_p-approx]ev}
	Let $0 < p \le 2$. Let Assumption II hold for Hilbert spaces $X^1$ and $X^2$.
	Let $v \in \Ll_2(X^2)$ be represented by the series
	\begin{equation} \label{HermiteSeriesVev}
	v = \sum_{\bs\in\FF_\rev} v_\bs H_\bs, \quad v_\bs \in X^2.
	\end{equation}		
	Assume that for $r=1,2$ there exist sequences $(\sigma_{r;\bs})_{\bs \in \FF_\rev}$ of numbers strictly larger than $1$ such that
	\begin{equation} \nonumber
	\sum_{\bs\in\FF_\rev} (\sigma_{r;\bs} \|v_\bs\|_{X^r})^2 <\infty
	\end{equation}		
	and 
	$(\sigma_{r;\bs}^{-1})_{\bs \in \FF_\rev} \in \ell_{q_r}(\FF_\rev)$ 
	for some $0<q_1\leq q_2 <\infty$ with $q_1 < 2$.	
		Then there exists a constant $C$ such that for each  $\xi >1$,
	\begin{equation} 	\label{L_p-rate-ev}
	\|v-\Ss_{G_\rev(\xi_n)}v\|_{\Ll_p(X^1)} 
	\ \leq \
	C\times  
	\begin{cases}
	\xi^{-\alpha}
	\quad &{\rm if }  \ \alpha \le 1/q_2 - 1/2;\\
	\xi^{-(1/q_1 - 1/2)}
	\quad  & {\rm if }  \ \alpha > 1/q_2 - 1/2.
	\end{cases}	
	\end{equation}
	The rate $\alpha$ is given by \eqref{spatialappnX}.	
\end{lemma}

\subsection{Interpolation  approximation}

For every $m \in \NN_0$, let $Y_m = (y_{m;k})_{k=0}^m$ be a sequence  of  points in $\RR$ such that 
\begin{equation} \label{interpolation points}
- \infty < y_{m;0} < \cdots < y_{m;m-1} < y_{m;m} < + \infty; \quad y_{0;0} = 0.
\end{equation}	
 If $v$ is a function on $\RR$ taking value in a Hilbert space $X$ and $m \in \NN_0$, we define the function  $I_m(v)$  on $\RR$ taking value in $X$ by
 \begin{equation} \label{I_(v)}
I_m(v)(y):= \ \sum_{k=0}^mv(y_{m;k}) \ell_{m;k}, \quad 
\ell_{m;k}(y) := \prod_{j=0,\ j\not=k}^m\frac{y - y_{m;j}}{y_{m;k} - y_{m;j}}, \quad y \in \RR,
 \end{equation}	
 interpolating $v$ at $y_{m;k}$, i.e., $I_m(v)(y_{m;k}) = v(y_{m;k})$.   Notice that for a function $v: \, \RR \to \RR$, the function $I_m(v)$ is the Lagrange polynomial having degree 
 $\le m + 1$, and that $I_m(\varphi) = \varphi$ for every polynomial $\varphi$ of degree $\le m + 1$.
 
 Let
  \begin{equation} \nonumber
 \lambda_m(Y_m):= \ \sup_{\|v\sqrt{g}\|_{L_\infty(\RR)} \le 1} \|I_m(v)\sqrt{g}\|_{L_\infty(\RR)} 
 \end{equation}	
 be the Lebesgue constant, see, e.g.,  \cite[Page 78]{Lu07}. 
 We want to choose a sequence $(Y_m)_{m=0}^\infty$ so that for some positive numbers $\varkappa$ and $C$, there holds the inequality
 \begin{equation} \label{ineq[lambda_m]}
\lambda_m(Y_m) 
\ \le \
(Cm+1)^\varkappa, \quad \forall m \in \NN_0.
 \end{equation}
 
We present two examples of $(Y_m)_{m=0}^\infty$ satisfying \eqref{ineq[lambda_m]}.
The first example is $(Y_m^*)_{n=0}^\infty$ where   $Y_m^* = (y_{m;k}^*)_{k=0}^n$  are the strictly increasing sequences of the roots of $H_{m+1}$. Indeed, it was proven by Matjila and Szabados \cite{Mat1994a,Mat1994b,Sza1997} that
\begin{equation} \nonumber
\lambda_m(Y_m^*) 
\ \le \
C(m+1)^{1/6}, \quad m \in \NN,
\end{equation}
for some positive constant $C$ independent of $n$ (with the obvious inequality $\lambda_0(Y_0^*) \le 1$). Hence, for every $\varepsilon > 0$, there exists a positive constant $C_\varepsilon$ independent of $n$ such that
\begin{equation} \label{ineq[lambda_m]2}
\lambda_m(Y_m^*) 
\ \le \
(C_\varepsilon m+1)^{1/6 + \varepsilon}, \quad \forall m \in \NN_0.
\end{equation}
 The minimum distance between consecutive roots $d_{m+1}$ satisfies the inequalities
$
\frac{\pi\sqrt{2}}{\sqrt{2m + 3}}  <  d_{m+1}  <  \frac{\sqrt{21}}{\sqrt{2m + 3}},
$
see \cite[pp. 130--131]{Sz39}. 
The sequences $Y_m^*$ have been used in \cite{Ch18} for sparse quadrature for non-fully discrete integration with the measure $\gamma$, and in \cite{EST18} non-fully discrete polynomial interpolation approximation with the measure $\gamma$.

The inequality \eqref{ineq[lambda_m]2}  can be improved if  $Y_{m-2}^*$ is slightly modified by the ``method of adding points" suggested by Szabados \cite{Sza1997} (for details, see also \cite[Section 11]{Lu07}). More precisely, for $n > 2$, he added to $Y_{m-2}^*$ two points $\pm \zeta_{m-1}$, near $\pm a_{m-1}(g)$,  which are defined by the condition $|H_{m-1} \sqrt{g}|(\zeta_{m-1})= \|H_{m-1} \sqrt{g}\|_{L_\infty(\RR)}$. By this way, he obtained the strictly increasing sequences 
\[
\bar{Y}_m^*:= \{-\zeta_m,y_{m-2;0}^*,\cdots,y_{m-2,m-2}^*,\zeta_m\}
\]
for which the sequence $(\bar{(Y}_m^*)_{n=0}^\infty$ satisfies the inequality
\begin{equation} \nonumber
\lambda_m(\bar{Y}_m^*) 
\ \le \
C\log (n-1) \quad (n > 2)
\end{equation} 
which yields that for every $\varepsilon > 0$, there exists a positive constant $C_\varepsilon$ independent of $n$ such that
\begin{equation} \nonumber
\lambda_m(\bar{Y}_m^*) 
\ \le \
(C_\varepsilon m+1)^\varepsilon, \quad  \forall m \in \NN_0.
\end{equation} 

For a given sequence $(Y_m)_{m=0}^\infty$, we define the univariate operator $\Delta^{{\rm I}}_m$ for $m \in \NN_0$ by
\begin{equation} \nonumber
\Delta^{{\rm I}}_m
:= \
I_m - I_{m-1},
\end{equation} 
with the convention $I_{-1} = 0$, and the univariate operator $\Delta^{{\rm I}*}_m$ for even $m \in \NN_0$ by
\begin{equation} \nonumber
	\Delta^{{\rm I}*}_m
	:= \
	I_m - I_{m-2},
\end{equation} 
with the convention $I_{-2} = 0$. 

 \begin{lemma} \label{lemma[Delta_{bs}]}
 	Assume that $(Y_m)_{m=0}^\infty$ is a sequence satisfying the condition \eqref{ineq[lambda_m]} for some positive numbers $\varkappa$ and $C$. Then for every $\varepsilon > 0$, there exists a positive constant $C_\varepsilon$ independent of $m$ such that for every function  $v$ on $\RR$,
 	\begin{equation} \label{ineq[Delta]1}
 \|\Delta^{{\rm I}}_m(v){\sqrt{g}}\|_{L_\infty(\RR)}	
 	\ \le \
 	(C_\varepsilon m+1)^{\varkappa + \varepsilon} \|v{\sqrt{g}}\|_{L_\infty(\RR)}	, \quad \forall m \in \NN_0,
 	\end{equation}
 whenever the norm in the right-hand side is finite.
\end{lemma}

\begin{proof} 
From the assumptions we have that
	\begin{equation} \nonumber
\|\Delta^{{\rm I}}_m(v){\sqrt{g}}\|_{L_\infty(\RR)}	
\ \le \
2(Cm+1)^\varkappa \|v{\sqrt{g}}\|_{L_\infty(\RR)}, \quad \forall m \in \NN_0,
\end{equation}
 which similarly to \eqref{ineq[lambda_m]2} gives \eqref{ineq[Delta]1}. 
\hfill
\end{proof}

In order to have a correct definition of interpolation operator let us
impose some necessary restrictions on $v \in L_2(\RRi,X,\gamma)$.
Let $\Ee$ be a $\gamma$-measurable subset in $\RRi$ such that $\gamma(\Ee) =1$ and $\Ee$ contains  all $\by \in \RRi$ with $|\by|_0 < \infty$, where $|\by|_0$ denotes the number of nonzero components $y_j$ of $\by$. If $v$ is a function defined on $\Ee$ and takes values in $X$, and if $v$ is a representative of a strongly $\gamma$-measurable function on $\Ee$, we can extend $v$ as a strongly $\gamma$-measurable function defined on the whole $\RRi$ in a certain way. For convenience, this extension is again denoted by $v$ which can be considered as an element in $v \in L_2(\RRi,X;\gamma)$. Denote by 
$L_2^\Ee(\RRi,X, \gamma)$ the set of all such functions.
We will treat  all elements 
$v \in L_2^\Ee(\RRi,X, \gamma)$ as representatives of elements in $L_2(\RRi,X, \gamma)$.  Throughout this  and next sections, we fix  a set $\Ee$. 

For  $v \in \Ll_2^\Ee(X)$, we introduce the tensor product operator $\Delta^{{\rm I}}_\bs$ for $\bs \in \FF$ by
\begin{equation} \nonumber
\Delta^{{\rm I}}_\bs(v)
:= \
\bigotimes_{j \in \NN} \Delta^{{\rm I}}_{s_j}(v),
\end{equation}
where the univariate operator
$\Delta^{{\rm I}}_{s_j}$ is applied to the univariate function $v$ by considering $v$ as a 
function of  variable $y_i$ with the other variables held fixed. From the definition of $\Ll_2^\Ee(X)$ one can see that the operators  $\Delta^{{\rm I}}_\bs$ are well-defined for all $\bs \in \FF$.

Next, we introduce the interpolation operator $I_\Lambda$ for a given finite set $\Lambda \subset \FF$ by
\begin{equation} \nonumber
I_\Lambda
:= \
\sum_{\bs \in \Lambda} \Delta^{{\rm I}}_\bs.
\end{equation}

 Let Assumption II hold for Hilbert spaces $X^1$ and $X^2$. We introduce the interpolation operator $\Ii_G: \Ll_2^\Ee(X^2) \to \Vv(G)$ for a given finite set $G \subset \NN_0 \times \FF$ by
\begin{equation} \nonumber
\Ii_G v
:= \
\sum_{(k,\bs) \in G} \delta_k \Delta^{{\rm I}}_\bs (v).
\end{equation}

The interpolation operators $\Delta^{{\rm I}*}_\bs$ for $\bs \in \FF_\rev$, $I^*_\Lambda$ for a finite set $\Lambda \subset \FF_\rev$, and $\Ii^*_G: \Ll_2^\Ee(X^2) \to \Vv(G)$ for a  finite set $G \subset \NN_0 \times \FF_\rev$, are defined in a similar way  by  replacing $\Delta^{{\rm I}}_{s_j}$ with $\Delta^{{\rm I}*}_{s_j}$, $j \in \NN$.
  
Notice that $\Ii_G v$  is a linear (non-adaptive) method of fully discrete polynomial interpolation approximation  which is the sum taken over the  indices set $G$, of mixed tensor products of  dyadic scale successive differences of ``spatial" approximations to $v$, and of   successive differences of their parametric Lagrange interpolating polynomials. It has been introduced in \cite{Di15} (see also \cite{Di18b}). A similar  construction for the multi-index stochastic collocation method for computing the expected value of a functional of the solution to elliptic PDEs with random data has been introduced in \cite{HNTT2016a,HNTT2016b}  by using Clenshaw-Curtis points for quadrature.

A set $\Lambda \subset \FF$ $(\Lambda \subset \FF_\rev)$  is called downward closed in $\FF$ 
(in $\FF_\rev$) if the inclusion $\bs \in \Lambda$ yields the inclusion $\bs' \in \Lambda$ for every $\bs' \in \FF$ ($\bs' \in \FF_\rev$) such that $\bs' \le \bs$. The inequality $\bs' \le \bs$ means that $s_j' \le s_j$, $j \in \NN$.   A sequence 
$(\sigma_\bs)_{\bs \in \FF}$ ($(\sigma_\bs)_{\bs \in \FF_\rev}$) is called increasing if $\sigma_{\bs'} \le \sigma_\bs$ for $\bs' \le \bs$. Put $R_\bs:= \{\bs' \in \FF: \bs' \le \bs\}$ and 
$R_{\rev;\bs}:= \{\bs' \in \FF_\rev: \bs' \le \bs\}$ .	

\begin{theorem}\label{thm[coll-approx]} 
	Let $0 < p \le 2$.  Let Assumption II hold for Hilbert spaces $X^1$ and $X^2$. Assume that
	$(Y_m)_{m \in \NN_0}$ is a sequence satisfying the condition \eqref{ineq[lambda_m]} for some positive numbers $\varkappa$ and $C$.  	Let $v \in \Ll_2^\Ee(X^2)$ be represented by the series \eqref{HermiteSeriesV}. Assume that for $r=1,2$ there exist increasing sequences $(\sigma_{r;\bs})_{\bs \in \FF}$  of numbers strictly larger than $1$ such that 
	\begin{equation} \label{summability}
	\sum_{\bs\in\FF} (\sigma_{r;\bs} \|v_\bs\|_{X^r})^2 <\infty \ \ \ \text{and} \ \ \
	\left(p_{\bs}(2\theta,\lambda)\sigma_{r;\bs}^{-1}\right)_{\bs \in \FF} \in \ell_{q_r}(\FF)
	\end{equation}		
	  for some $0<q_1\leq q_2 <\infty$ with $q_1 <2$, where
	\begin{equation} \label{theta,lambda2}
	\theta := \varkappa + \varepsilon + 5/4, \quad  
	\lambda :=  
	\max\left(C_{\infty,2}, C_{2,\infty},C_\varepsilon, 1 \right),
	\end{equation}
	$C_{\infty,2}$, $C_{2,\infty}$ are as in Lemma~\ref{lemma[N-ineq]}, $\varepsilon$ is arbitrary positive number and  $C_\varepsilon$ is as in Lemma~\ref{lemma[Delta_{bs}]}. For $\xi >1$, let $G(\xi)$ be the set defined as in \eqref{G_(xi)}.
	 	
	  Then there exists a constant $C$ such that for each $n \in \NN$, there exists a number $\xi_n$  such that  
	  for the interpolation operator $\Ii_{G(\xi_n)}: \Ll_2^\Ee(X^2) \to \Vv(G(\xi_n))$, we have that $\dim \Vv(G(\xi_n)) \le  n$ and 
	\begin{equation} \label{u-I_Gu, p le 2}
	\|v -\Ii_{G(\xi_n)}v\|_{\Ll_p(X^1)} \leq 
	C \times
	\begin{cases}
		n^{-\alpha} &\text{if } \alpha \leq  1/q_2 - 1/2,
		\\
		n^{-\beta} \log n &\text{if } \alpha >  1/q_2 - 1/2.
	\end{cases} 
	\end{equation}
	The rate $\alpha$ corresponds to the approximation of a single function in $X^2$ as given by \eqref{spatialappnX}.
	The rate $\beta$ is given by 
\begin{equation} 	\label{[beta]1}
\beta := \left(\frac 1 {q_1} - \frac 1 2 \right)\frac{\alpha}{\alpha + \delta}, \ \ \ 
\delta := \frac 1 {q_1} - \frac 1 {q_2}.
\end{equation}	
\end{theorem}

\begin{proof} 
Clearly, by  the inequality $\|\cdot\|_{\Ll_p(X^1)} \le \|\cdot\|_{\Ll_2(X^1)}$ it is sufficient to prove the theorem for $p = 2$. 		
 By Lemmata~\ref{lemma[AbsConv]} and \ref{lemma[AbsConv]2} the series \eqref{HermiteSeriesV} and \eqref{[delta-series]} 
converge absolutely, and therefore, unconditionally  in  the Hilbert space $\Ll_2(X^1)$ to $v$. We have that $\Delta^{{\rm I}}_\bs H_{\bs'} = 0$ for every $\bs \not\le \bs'$. Moreover, if  $\Lambda \subset \FF$ is  downward closed set in $\FF$, then $I_{\Lambda} H_\bs = H_\bs$ for every $\bs \in \Lambda$, and 
hence we can write
\begin{equation}\label{I_Lambda}
I_{\Lambda}v
\ = \
I_{\Lambda}\Big(\sum_{ \bs \in \FF} v_\bs \,H_\bs \Big)
\ =  \
\sum_{ \bs \in \FF} v_\bs \,I_{\Lambda} H_\bs
\ =  \
\sum_{ \bs \in \Lambda}  v_\bs \, H_\bs
\ + \ \sum_{\bs \not\in \Lambda} v_\bs \, I_{\Lambda \cap R_\bs}\, H_\bs.
\end{equation}

Let $\xi>1$ be given. For $k \in  \NN_0$, put
\begin{equation} \nonumber
	\Lambda_k(\xi)
	:= \ 
	\begin{cases}
		\big\{\bs \in \FF: \,\sigma_{2;\bs}^{q_2} \leq 2^{-k}\xi\big\} \quad &{\rm if }  \ 
		\alpha \le 1/q_2 - 1/2;\\
		\big\{\bs \in \FF: \, \sigma_{1;\bs}^{q_1} \le \xi , \  
		\sigma_{2;\bs} \leq 2^{- \tau k}\xi^\vartheta \big\} \quad  & {\rm if }  \ 
		\alpha > 1/q_2 - 1/2.
	\end{cases}
\end{equation}
For $\xi > 1$, we put $k(\xi):= \lfloor \log_2 \xi \rfloor$ if $\alpha \le 1/q_2 - 1/2$, and 
	$k(\xi):= \lfloor \vartheta \tau^{-1}\log_2 \xi \rfloor$ if $\alpha > 1/q_2 - 1/2$.
	Observe that $\Lambda_k(\xi) = \emptyset$ for all $k > k(\xi)$, and consequently, 
	we have that
	\begin{equation}  \label{Eq[Ii]}
	\Ii_{G (\xi)} v
	\ = \ 
	\sum_{k=0}^{k(\xi)} \delta_k \Big(\sum_{ \bs \in \Lambda_k(\xi)} \Delta^{{\rm I}}_\bs \Big)v
	\ = \ 
	\sum_{k=0}^{k(\xi)} \delta_k I_{\Lambda_k(\xi)}v.
	\end{equation}
	Since the sequence $(\sigma_{2;\bs})_{\bs \in \FF}$ is increasing,  $\Lambda_k(\xi)$ are downward closed sets in $\FF$, and consequently, the sequence $\big\{\Lambda_k(\xi)\big\}_{k=0}^{k(\xi)}$ is nested in the inverse order, i.e., $\Lambda_{k'} \subset \Lambda_k(\xi)$ if $k' > k$, and 
	$\Lambda_0$ is the largest and $\Lambda_{k(\xi)} = \{0_\FF\}$.
	Therefore, from  the unconditional convergence of the series \eqref{[delta-series]} to $v$,  \eqref{Eq[Ii]} and \eqref{I_Lambda} we derive that
	\begin{equation}\nonumber
	\begin{split}
	\Ii_{G (\xi)} v
	\ &= \
	\sum_{k=0}^{k(\xi)}  \sum_{ \bs \in \Lambda_k(\xi)} \delta_k(v_\bs) \, H_\bs
	\ + \ 
	\sum_{k=0}^{k(\xi)}  \sum_{\bs \not\in \Lambda_k(\xi)} \delta_k(v_\bs) 
	I_{\Lambda_k(\xi) \cap R_\bs}\, H_\bs
	\\[1.5ex]
	\ &= \
	\Ss_{G (\xi)} v
	\ + \ 
	\sum_{k=0}^{k(\xi)}  \sum_{\bs \not\in \Lambda_k(\xi)} \delta_k(v_\bs) 
	I_{\Lambda_k(\xi) \cap R_\bs}\, H_\bs.
	\end{split}
	\end{equation}
	This implies that
	\begin{equation} \label{v-Ii_G}
	v \ - \ \Ii_{G (\xi)} v
	\ = \
	v \ - \ \Ss_{G (\xi)} v
	\ - \ 
\sum_{k=0}^{k(\xi)}  \sum_{\bs \not\in \Lambda_k(\xi)} \delta_k(v_\bs) 
I_{\Lambda_k(\xi) \cap R_\bs}\, H_\bs.
	\end{equation}
	Observe that for $k \le k(\xi)$, if $\bs \not\in\Lambda_k(\xi)$, then $(k,\bs) \not\in G(\xi)$.  Hence, by \eqref{v-Ii_G} it follows that 
	\begin{equation} \label{[|u-Iu|<]1}
		\big\|v- \Ii_{G (\xi)} v\big\|_{\Ll_2(X^1)}
		\ \le \
		\big\|v- \Ss_{G (\xi)} v\big\|_{\Ll_2(X^1)} 
		+  
		\ \sum_{(k,\bs) \not\in G (\xi)} \|\delta_k(v_\bs)\|_{X^1} \, 
		\big\|I_{\Lambda_k(\xi) \cap R_\bs}\, H_\bs\big\|_{L_2(\RRi,\gamma)}.
	\end{equation}

We need the following auxiliary bound for 
$\big\|I_{\Lambda_k(\xi) \cap R_\bs}\, H_\bs\big\|_{L_2(\RRi,\gamma)}$,
\begin{equation} \label{I_Lambda_k}
	\big\|I_{\Lambda_k(\xi) \cap R_\bs}\, H_\bs\big\|_{L_2(\RRi,\gamma)}
	\ \le \
	p_\bs(\theta, \lambda), 
\end{equation}
where
$\theta$ and $\lambda$ are as in \eqref{theta,lambda2}.	
We have that
\begin{equation} \label{I_Lambda_k(xi)}
	\big\|I_{\Lambda_k(\xi) \cap R_\bs}\, H_\bs\big\|_{L_2(\RRi,\gamma)}
	\ \le \
	\sum_{\bs' \in \Lambda_k(\xi) \cap R_\bs}\, \|\Delta^{{\rm I}}_{\bs'} (H_\bs)\big\|_{L_2(\RRi,\gamma)}.
\end{equation}
We estimate the norms in the sum in the right-hand side.  Assuming $\bs \in \FF$ to be such that $\supp (\bs) \subset \{1,...,J\}$,	we have 
$\Delta^{{\rm I}}_{\bs'} (H_\bs) = \prod_{j=1}^J\Delta^{{\rm I}}_{s'_j} (H_{s_j})$.
Since $\Delta^{{\rm I}}_{s'_j} (H_{s_j})$ is a polynomial of degree $\le s'_j$ in variable $y_j$, from Lemma~\ref{lemma[N-ineq]}	we obtain that
\begin{equation} \nonumber
	\begin{split}
		\|\Delta^{{\rm I}}_{\bs'} (H_\bs)\big\|_{L_2(\RRi,\gamma)}
		\ &= \
		\prod_{j=1}^J\|\Delta^{{\rm I}}_{s'_j} (H_{s_j})\|_{L_2(\RR,\gamma)}
		\ = \
		\prod_{j=1}^J\|\Delta^{{\rm I}}_{s'_j} (H_{s_j})\sqrt{g}\|_{L_2(\RR)}
		\\[1.5ex]
		\ &\le \
		p_{\bs'}(\theta,\lambda)  \prod_{j=1}^J\|\Delta^{{\rm I}}_{s'_j} (H_{s_j})\sqrt{g}\|_{L_\infty(\RR)},
	\end{split}
\end{equation}
where $\theta = 1/4$, $\lambda:=  C_{2,\infty}$ and $C_{2,\infty}$ is the constant in Lemma \ref{lemma[N-ineq]}.
Due to the assumption \eqref{ineq[lambda_m]}, we have by Lemmata  \ref{lemma[Delta_{bs}]} and \ref{lemma[H_bs|_infty]} that 
\begin{equation} \nonumber
	\begin{split}
		\|\Delta^{{\rm I}}_{s'_j} (H_{s_j})\sqrt{g}\|_{L_\infty(\RR)}
		\ &\le \
		(1 + C_\varepsilon s'_j)^{\varkappa + \varepsilon} \, \|H_{s_j}\sqrt{g}\|_{L_\infty(\RR)}\\
		\ &\le \
		(1 + C_\varepsilon s'_j)^{\varkappa + \varepsilon} (1 + C_{\infty,2} s_j)^{1/4} \, \|H_{s_j}\sqrt{g}\|_{L_2(\RR)}\\
		\ &= \
		(1 + C_\varepsilon s'_j)^{\varkappa + \varepsilon} (1 + C_{\infty,2} s_j)^{1/4}
	\end{split}
\end{equation}
and consequently,
\begin{equation}  \label{[|Delta_{bs'}|<]4}
	\|\Delta^{{\rm I}}_{\bs'} (H_\bs)\big\|_{L_2(\RRi,\gamma)}
	\ \le \
	p_{\bs'}(\theta,\lambda)
	\ \le \ 
	p_\bs(\theta,\lambda),
\end{equation}
where 
\begin{equation} \label{theta,lambda}
	\theta :=  \varkappa + \varepsilon + 1/4, \quad  
	\lambda :=  
	\max\left(C_{\infty,2}, C_{2,\infty}, C_\varepsilon \right). 
\end{equation}
Substituting $\|\Delta^{{\rm I}}_{\bs'} (H_\bs)\big\|_{L_2(\RRi,\gamma)}$ in 
\eqref{I_Lambda_k(xi)} with the right-hand side of \eqref{[|Delta_{bs'}|<]4}  gives that
\begin{equation} \label{I_Lambda_k-1}
	\big\|I_{\Lambda_k(\xi) \cap R_\bs}\, H_\bs\big\|_{L_2(\RRi,\gamma)}
	\ \le \
	\sum_{\bs' \in \Lambda_k(\xi) \cap R_\bs}\, p_\bs(\theta, \lambda) 
	\ \le \
	\min \big(|\Lambda_k(\xi) |,|R_\bs|\big)\, p_\bs(\theta, \lambda),
\end{equation}	
where $\theta$ and $\lambda$ are as in \eqref{theta,lambda}. Hence, by the inequality 
$|R_\bs|\le p_\bs(1,1)$ we get	
\eqref{I_Lambda_k}.

 From  \eqref{[|u-Iu|<]1} and \eqref{I_Lambda_k} it follows that 
\begin{equation} \label{[|v-Iv|}
	\big\|v- \Ii_{G (\xi)} v\big\|_{\Ll_2(X^1)}
	\ \le \
	\big\|v- \Ss_{G (\xi)} v\big\|_{\Ll_2(X^1)} 
	+  
	A(\xi).
\end{equation}
where
\begin{equation} \label{A(xi)}
	A(\xi):= \ \sum_{(k,\bs) \not\in G (\xi)} \|\delta_k(v_\bs)\|_{X^1} \, 
		p_\bs(\theta, \lambda), 
	\end{equation}
	where $\theta$ and $\lambda$ are as in \eqref{theta,lambda2}.	In the next step, we employ the inequality \eqref{[|v-Iv|} to establish bounds for $\big\|v- \Ii_{G (\xi)} v\big\|_{\Ll_2(X^1)}$.

	{\it Let us first consider the case $\alpha \le 1/q_2  - 1/2$.}	
Lemma  \ref{lemma[L_p-approx]}  gives
	\begin{equation}  \label{v-Ss_Gv}
	\Big\|v- \Ss_{G (\xi)} u\Big\|_{\Ll_2(X^1)}
	\ \le \ 
	C\,  \xi^{-\alpha}.
	\end{equation}
Let us estimate  the  term $A(\xi)$  in the right-hand side of \eqref{[|v-Iv|}. 
	By using  \eqref{delta-approx-propertyX}  we  derive that
	\begin{equation} \nonumber
	\begin{split}
A(\xi)
	\ &\le \
	C	\sum_{(k,\bs) \not\in G (\xi)} 2^{-\alpha k}p_\bs(\theta, \lambda) \|v_\bs\|_{X^2}
	\ = \
	C \sum_{\bs \in \FF} p_\bs(\theta, \lambda)  \|v_\bs\|_{X^2} \ \sum_{2^k >\xi \sigma_{2;\bs}^{-q_2}}  2^{-\alpha k}
	\\[1.5ex]
	\ &\le \
	C \sum_{\bs \in \FF} p_\bs(\theta, \lambda)  \|v_\bs\|_{X^2} \ (\xi \sigma_{2;\bs}^{-q_2})^{-\alpha}
	\ \le \
	C \xi^{- \alpha}\, \sum_{\bs \in \FF} p_\bs(\theta, \lambda)  \sigma_{2;\bs}^{q_2\alpha}\|v_\bs\|_{X^2}.
	\end{split}
	\end{equation}
	By  the inequalities  $2(1 - q_2\alpha) \ge q_2$ and $\sigma_{2;\bs} >1$ and the assumptions we have that
	\begin{equation} \nonumber
	\begin{split}
	\sum_{\bs \in \FF} p_\bs(\theta, \lambda)  \sigma_{2;\bs}^{q_2\alpha}\|v_\bs\|_{X^2}
	\ &\le \
	\left(\sum_{\bs \in \FF}  (\sigma_{2;\bs}\|v_\bs\|_{X^2})^2\right)^{1/2}
	\left(\sum_{\bs \in \FF} p_\bs^2(\theta, \lambda)  \sigma_{2;\bs}^{-2(1 - q_2\alpha)}\right)^{1/2}
	\\[1.5ex]
	\ &\le \
	\left(\sum_{\bs \in \FF}  (\sigma_{2;\bs}\|v_\bs\|_{X^2})^2\right)^{1/2}
	\left(\sum_{\bs \in \FF} p_\bs(2\theta,\lambda) \sigma_{2;\bs}^{- q_2}\right)^{1/2}
	\ = \
	\ C \ < \infty.
	\end{split}
	\end{equation}	
Thus, we obtain the estimates
	\begin{equation} \nonumber
A(\xi)
\ \le \
C \xi^{-\alpha}.
\end{equation}
This together with \eqref{[|u-Iu|<]1} and \eqref{[|v-Iv|} implies that
\begin{equation} \nonumber
\|v- \Ii_{G(\xi)} u\|_{\Ll_2(X^1)}
\  \le \ C \xi^{-\alpha}.
\end{equation}
  Hence, similarly to \eqref{dimVGxi}--\eqref{case<1/q_2}, for each $n \in \NN$ we can find  a number $\xi_n$  such that  $\dim \Vv(G(\xi_n)) \le  n$ and 
\begin{equation} \label{u-I_Gu, alpha<}
\|v -\Ii_{G(\xi_n)}v\|_{\Ll_p(X^1)} \leq Cn^{-\alpha}, \quad \alpha \le 1/q_2 - 1/2.
\end{equation}

{\it We now consider the case $\alpha > 1/q_2  - 1/2$.}
Lemma \ref{lemma[L_p-approx]} gives 
\begin{equation} \label{v - Ss_G}
	\big\|v- \Ss_{G (\xi)} v\big\|_{\Ll_2(X^1)} 
	\  \le \  C \xi^{-(1/q_1 - 1/2)}.
\end{equation}
 
We split $A(\xi)$ into two sums as
\begin{equation*} 
	A(\xi)
	= 
	 A_1(\xi) + A_2(\xi),
\end{equation*}
where
\begin{equation*} 
	A_1(\xi)
	:= 
	\sum_{\sigma_{1;\bs}^{q_1} > \xi, \  
		\sigma_{2;\bs} \leq 2^{- \tau k}\xi^\vartheta}   \delta_k(v_\bs) \, 
p_\bs(\theta, \lambda),	
\end{equation*}
and
\begin{equation*} 
A_2(\xi)
:= 
\sum_{ 
	\sigma_{2;\bs} > 2^{- \tau k}\xi^\vartheta}   \delta_k(v_\bs) \, 
p_\bs(\theta, \lambda).
\end{equation*}
We get by Assumption II,
\begin{equation} \label{A_1<}
	\begin{split}
		A_1(\xi)
		\ &\le \ 	
			\sum_{\sigma_{1;\bs}^{q_1} > \xi}  \|\delta_k(v_\bs)\|_{X^1}	 
	p_\bs(\theta, \lambda)
		\\
		\  &\le  \
		C \sum_{\sigma_{1;\bs}^{q_1} > \xi} \|v_\bs\|_{X^1} 
		\sum_{k=0}^{k(\xi)}\, p_\bs(\theta,\lambda)
		\\
		\ &\le \ 	
		C	\log \xi \sum_{\sigma_{1;\bs}^{q_1} > \xi}
		\|v_\bs\|_{X^1}	 \, p_\bs(\theta,\lambda).
	\end{split}
\end{equation}
We obtain by the H\"older inequality and the  hypothesis of the theorem,
\begin{equation*} 
	\begin{split}
		\sum_{\sigma_{1;\bs}^{q_1} > \xi}   \|v_\bs\|_{X^1}  \, p_\bs(\theta,\lambda)
		\ &\le \
		\left(\sum_{\sigma_{1;\bs}^{q_1} > \xi} (\sigma_{1;\bs}\|v_\bs\|_{X^1})^2\right)^{1/2}
		\left(\sum_{\sigma_{1;\bs}^{q_1} > \xi} p_\bs^2(\theta,\lambda) \sigma_{1;\bs}^{-2}\right)^{1/2}
		\\[1.5ex]
		\ &\le \
		C\,
		\left(\sum_{\sigma_{1;\bs}^{q_1} > \xi} p_\bs^2(\theta,\lambda) \sigma_{1;\bs}^{-q_1} 
		\sigma_{1;\bs}^{-(2- q_1)}\right)^{1/2}
		\\[1.5ex]
		\ &\le \
		C (\xi)^{-(1/q_1 - 1/2)}
		\left(\sum_{\bs \in \FF} p_\bs(2\theta,\lambda) \sigma_{1;\bs}^{-q_1} \right)^{1/2}
		\\[1.5ex]
		\ &\le \
		C \xi^{-(1/q_1 - 1/2)}.
	\end{split}
\end{equation*} 
This and \eqref{A_1<} yield that
\begin{equation} \label{A_1}
	A_1(\xi)
	\  \le \  C \xi^{-(1/q_1 - 1/2)} \log \xi.
\end{equation}

We now give a bound for $A_2(\xi)$. Observe that $\vartheta \alpha/\tau = 1/q_1 - 1/2$ and 
$\alpha/\tau =1 -  q_2/2$. 
Hence,  employing \eqref{delta-approx-propertyX}, the assumption~\eqref{summability} and  the H\"older inequality we get 
\begin{equation*} 
	\begin{split}
	A_2(\xi)
		\ &\le \
		\sum_{\bs \in \FF} \
		\sum_{ 2^k > \left(\xi^\vartheta\sigma_{2;\bs}^{-1}\right)^{1/\tau} }
		\|	\delta_k(v_\bs)\|_{X^1} \, 
		\big\|I_{\Lambda_k(\xi) \cap R_\bs}\, H_\bs\big\|_{L_2(\RRi,\gamma)}
		\\[1.5ex]
		\ &\le \
		C \, \sum_{\bs \in \FF} \
		\sum_{ 2^k > \left(\xi^\vartheta\sigma_{2;\bs}^{-1}\right)^{1/\tau} }
		2^{-\alpha k} \|v_\bs\|_{X^2}\, p_\bs(\theta,\lambda)							
		\\[1.5ex]
		\ &\le \
		C\,
		\sum_{\bs \in \FF} \
		\big(  \sigma_{2;\bs}^{-1/\tau} \xi^{\vartheta /\tau}\big)^{- \alpha }\|v_\bs\|_{X^2}\, p_\bs(\theta,\lambda)
		\\[1.5ex]
		\ &\le \
		C\,
		\xi^{- \vartheta \alpha/\tau }\sum_{\bs \in \FF} \
		\sigma_{2;\bs}^{\alpha/\tau} 
		\|v_\bs\|_{X^2}\, p_\bs(\theta,\lambda)
		\\[1.5ex]
		\ &\le \
		C\,
		\xi^{- (1/q_1 - 1/2)}\	\sum_{\bs \in \FF} \
		\sigma_{2;\bs}^{1 - q_2/2} 
		\|v_\bs\|_{X^2}\, p_\bs(\theta,\lambda)
		\\
		\ &\le \
		C\,
		\xi^{- (1/q_1 - 1/2)}\	
		\left(\sum_{\bs \in \FF}  (\sigma_{2;\bs}\|v_\bs\|_{X^2})^2\right)^{1/2}
		\left(\sum_{\bs \in \FF}  p_\bs^2(\theta,\lambda) \sigma_{2;\bs}^{-2}\right)^{1/2}
		\\
		\ &\le \
		C\,
		\xi^{- (1/q_1 - 1/2)}.
	\end{split}
\end{equation*}
This proves that 
\begin{equation} \label{A_2}
	A_2(\xi)
	\  \le \  C \xi^{-(1/q_1 - 1/2)}.
\end{equation} 

Combining \eqref{A_1}, \eqref{A_2}  and \eqref{v - Ss_G} leads to the estimate
\begin{equation} \label{xi-rate4}
\|v- \Ii_{G(\xi)} v\|_{\Ll_2(X^1)}
\  \le \  C \xi^{-(1/q_1 - 1/2)} \log \xi.
\end{equation}

We   estimate   the dimension of the space $\Vv(G(\xi))$.  Put $q:= \tau q_2 $ and define $q'$ by $1/q' + 1/q = 1$. 
Since $\alpha > 1/q_2 - 1/2$, we have that $q >1$, consequently, by using  the H\"older inequality and \eqref{summability}, we derive that
\begin{equation} \nonumber
	\begin{split}
		\dim \Vv(G(\xi))
		\ &\le \
		\ \sum_{(k,\bs) \in G(\xi)} \dim V_{2^k} 
		\ \le \ 
		\sum_{\sigma_{1;\bs}^{q_1} \le \xi} \quad 
		\sum_{2^{\tau k} \sigma_{2;\bs} \leq \xi^\vartheta} 2^k 
		\\ \  &\le \
		2 \sum_{\sigma_{1;\bs}^{q_1} \le \xi} \xi^{\vartheta/\tau}\sigma_{2;\bs}^{-1/\tau} 
		\ \le \
		2 \xi^{\vartheta/\tau} \left(\sum_{\sigma_{1;\bs}^{q_1} \le \xi}\sigma_{2;\bs}^{-q_2} \right)^{1/q}
		\left(\sum_{\sigma_{1;\bs}^{q_1} \le \xi} 1 \right)^{1/q'}
		\\ \  &\le 
		2 \xi^{\vartheta/\tau}\left(\sum_{\bs \in \FF} \sigma_{2;\bs}^{-q_2} \right)^{1/q}
		\left(\sum_{\bs \in \FF} \xi \sigma_{1;\bs}^{-q_1}  \right)^{1/q'}
		\\ \  &=
		M \xi^{\vartheta/\tau + 1/q'} 
		\ =  \
		M \xi^{1 + \delta/\alpha},
	\end{split}
\end{equation}
where
$
M:=  2 \big\|\big(\sigma_{2;\bs}^{-1}\big)\big\|_{\ell_{q_2}(\FF)}^{q_2/q} \,
\big\|\big(\sigma_{1;\bs}^{-1}\big)\big\|_{\ell_{q_1}(\FF)}^{q_1/q'}.
$

For any $n \in \NN$, letting $\xi_n$ be a number satisfying the inequalities 
\begin{equation} \label{[xi_n]}
	M\, \xi_n^{1 + \delta/\alpha}
	\ \le \
	n
	\ < \ 2M\, \xi_n^{1 + \delta/\alpha},
\end{equation}
we derive that  $\dim \Vv(G(\xi_n)) \le  n$.
On the other hand, by \eqref{[xi_n]},
	\begin{equation} \nonumber
		\xi_n^{-(1/q_1 - 1/2)} 
		\asymp  n^{-(1/q_1 - 1/2)\frac{\alpha}{\alpha + \delta}}
		\ = \ n^{-\beta}.
	\end{equation}
	This together with \eqref{xi-rate4} proves that
	\begin{equation} \nonumber
		\|v -\Ii_{G(\xi_n)}v\|_{\Ll_p(X^1)} \leq Cn^{-\beta}\log n, \quad \alpha > 1/q_2 - 1/2.
	\end{equation}

By combining the last estimate and \eqref{u-I_Gu, alpha<} we derive  \eqref{v-Ii_G}.
\hfill
\end{proof}

	Denote by $\Gamma_\bs$ and $\Gamma(\Lambda)$ the set of interpolation points in the operators $\Delta^{{\rm I}}_\bs$ and $I_\Lambda$, respectively. We have that 
	$\Gamma_\bs = \{\by_{\bs - \be;\bm}:  e \in \EE_\bs; \ m_j = 0,...,s_j - e_j, \ j \in \NN \}$ and $\Gamma(\Lambda) = \cup_{\bs \in \Lambda} \Gamma_\bs$, where $\EE_\bs$ is the subset in $\FF$ of all $\be$ such that $e_j$ is $1$ or $0$ if $s_j > 0$, and $e_j$ is $0$ if $s_j = 0$, and $\by_{\bs;\bm}:= (y_{s_j;m_j})_{j \in \NN}$. 
	 
	\begin{remark}	\label{rm 3.1a}
{\rm		
(i) Observe that  the operator $\Ii_{G(\xi_n)}$ in Theorem \ref{thm[coll-approx]} can be represented in the form of a multilevel approximation method with $k_n$ levels:
		\begin{equation}  \nonumber
		\Ii_{G (\xi_n)} 
		\ = \ 
		\sum_{k=0}^{k_n} \delta_k I_{\Lambda_k(\xi_n)},
		\end{equation}
		where  $k_n:= \lfloor \log_2 \xi_n \rfloor$ if $\alpha \le 1/q_2 - 1/2$, and  
		$k_n:= \lfloor \vartheta \tau^{-1}\log_2 \xi_n \rfloor$ if $\alpha > 1/q_2 - 1/2$, and for $k \in \NN_0$ and $\xi>1$,
		\begin{equation} \nonumber
		\Lambda_k(\xi)
		:= \ 
		\begin{cases}
		\big\{\bs \in \FF: \,\sigma_{2;\bs}^{q_2} \leq 2^{-k}\xi\big\} \quad &{\rm if }  \ 
		\alpha \le 1/q_2 - 1/2;\\
		\big\{\bs \in \FF: \, \sigma_{1;\bs}^{q_1} \le \xi , \  
		\sigma_{2;\bs} \leq 2^{- \tau k}\xi^\vartheta \big\} \quad  & {\rm if }  \ 
		\alpha > 1/q_2 - 1/2.
		\end{cases}
		\end{equation}
		Moreover,  $\Lambda_k(\xi_n)$ are downward closed sets, and consequently, the sequence $\big\{\Lambda_k(\xi_n)\big\}_{k=0}^{k_n}$ is nested in the inverse order, i.e., $\Lambda_{k'}(\xi_n) \subset \Lambda_k(\xi_n)$ if $k' > k$, and 
		$\Lambda_0(\xi_n)$ is the largest and $\Lambda_{k_n}(\xi_n) = \{0_\FF\}$.
	
		\noindent		
(ii) Theorem \ref{thm[coll-approx]} is a non-adaptive  "collocation" extension of  \cite[Theorem 3.1]{BCDS17} for the lognormal case.  The approximant  $\Ii_{G(\xi_n)} v$ belongs to the space  $\Vv(G(\xi_n)) \subset \Vv(X^2)$. The convergence rate $\min(\alpha,\beta)$ of the approximation by $\Ii_{G(\xi_n)} v$  is given in terms of  $n$  where  the thresholding parameter $\xi_n$ is chosen such that  $\dim(\Vv(G(\xi_n)) \le n$.   This rate is the same as the rate of the approximation by the truncated Hermite GPC expansion $\Ss_{G(\xi_n)} v$.  The fully discrete polynomial interpolation approximation of $v \in \Ll_2^\Ee(X^2)$ by operators $\Ii_{G(\xi_n)}$ is  based on the finite point-wise information in $\by$, more precisely, on   
 $|\Gamma(\Lambda_0(\xi_n))|$ of particular values of   $v$ at the interpolation points $\by \in \Gamma(\Lambda_0(\xi_n))$ and the approximations of   $v(\by)$, $\by \in \Gamma(\Lambda_0(\xi_n))$, by  $P_{2^k}v(\by)$ for $k=0,..., k_n$.    Moreover, we have that 
\[
|\Gamma(\Lambda_0(\xi_n))| \le \sum_{\bs \in \Lambda_0(\xi_n)}p_\bs(1,2) = \Oo(n).
\]

	\noindent		
	(iii) Under the assumptions of Theorem 	\ref{thm[coll-approx]},  we have that for every $v \in \Ll_2(X^2)$ and every $G \subset \NN_0 \times \FF$,
	\begin{equation} \nonumber
	\Ii_G v
	\ = \
	\sum_{(k,\bs) \in G}  \Delta^{{\rm I}}_\bs (\delta_k v).
	\end{equation} 
}		
	\end{remark} 

\begin{theorem}\label{thm[coll-approx]ev}
	Let $0 < p \le 2$.	Let Assumption II hold for Hilbert spaces $X^1$ and $X^2$. 
	Let $v \in \Ll_2^\Ee(X^2)$ be represented by the series \eqref{HermiteSeriesVev}.  Assume that
	$(Y_m)_{m \in \NN_0}$ is a sequence satisfying the condition \eqref{ineq[lambda_m]} for some positive numbers $\tau$ and $C$. 
	Assume that for $r=1,2$ there exist increasing sequences $(\sigma_{r;\bs})_{\bs \in \FF_\rev}$ of numbers strictly larger than $1$ such that
	\begin{equation} \nonumber
	\sum_{\bs\in\FF_\rev} (\sigma_{r;\bs} \|v_\bs\|_{X^r})^2 <\infty
	\end{equation}		
	and 
	$(p_{\bs}(2\theta,\lambda)\sigma_{r;\bs}^{-1})_{\bs \in \FF_\rev} \in \ell_{q_r}(\FF_\rev)$  for some $0<q_1\leq q_2 <\infty$ with $q_1 <2$, where $\theta$	and $\lambda$  are as in \eqref{theta,lambda2}. For $\xi>1$, let $G_\rev(\xi)$ be the set defined as in \eqref{G_rev(xi)}.
	
	 	Then there exists a constant $C$ such that for each  $n \in \NN$, there exists a number $\xi_n$  such that  
for the operator $\Ii^*_{G_\rev(\xi_n)}: \Ll_2^\Ee(X^2) \to \Vv(G_\rev(\xi_n))$, we have that 	$\dim \Vv(G_\rev(\xi_n)) \le  n$ and 
	\begin{equation} \label{u-I_Gu, p le 2ev}
	\|v -\Ii^*_{G_\rev(\xi_n)}v\|_{\Ll_p(X^1)} \leq C \times
	\begin{cases}
		n^{-\alpha} &\text{if } \alpha \leq  1/q_2 - 1/2,
		\\
		n^{-\beta} \log n &\text{if } \alpha >  1/q_2 - 1/2.
	\end{cases} 
	\end{equation}
	The rate $\alpha$ corresponds to the approximation of a single function in $X^2$ as given by \eqref{spatialappnX}.
	The rate $\beta$ is given by \eqref{[beta]1}.
\end{theorem}

\begin{proof} The proof of this theorem is similar to the proof of Theorem~\ref{thm[coll-approx]} with some modification. For example, all the indices sets are taken from the sets $\FF_\rev$ and $\NN_0 \times \FF_\rev$  instead $\FF$ and $\NN_0 \times \FF$; estimates similar to \eqref{v-Ss_Gv} and \eqref{v - Ss_G} are given by Lemma~\ref{lemma[V_p-approx]ev} instead Lemma~\ref{lemma[L_p-approx]}. 
	\hfill
\end{proof}

\begin{corollary}\label{corollary[coll-approx]} 
	Let $0 < p \le 2$.
Let $v \in \Ll_2^\Ee(X)$ be represented by the series \eqref{series} for a Hilbert space $X$. 
Assume that
$(Y_m)_{m \in \NN_0}$ is a sequence satisfying the condition \eqref{ineq[lambda_m]} for some positive numbers $\tau$ and $C$. 
Assume that there exists an increasing sequence $(\sigma_\bs)_{\bs \in \FF}$ of numbers strictly larger than $1$ such that
\begin{equation} \nonumber
\sum_{\bs\in\FF} (\sigma_\bs \|v_\bs\|_X)^2 <\infty
\end{equation}		
and $(p_{\bs}(2\theta, \max(2,\lambda))\sigma_\bs^{-1})_{\bs \in \FF} \in \ell_q(\FF)$  for some $0<q<2$, where $\theta$	and $\lambda$  are as in \eqref{theta,lambda2}. 
For $\xi>1$, define
\begin{equation} \label{Lambda(xi)}
\Lambda(\xi):= \{\bs \in \FF: \, \sigma_\bs^q \le \xi\}.
\end{equation}	

	Then there exists a constant $C$ such that for each $n \in \NN$, there exists a number $\xi_n$  such that  $|\Gamma(\Lambda(\xi_n))| \le n$ and 
\begin{equation} \label{u-I_Lambdau, p le 2}
\|v -I_{\Lambda(\xi_n)}v\|_{\Ll_p(X)} \leq Cn^{-(1/q - 1/2)}.
\end{equation}
\end{corollary}
	
\begin{proof}
Similarly to the proof of Theorem~\ref{thm[coll-approx]} it is sufficient to prove \eqref{u-I_Lambdau, p le 2} for $p=2$. In the same way as in proving \eqref{[|u-Iu|<]1}, we can show that
\begin{equation} \nonumber
\big\|v- I_{\Lambda(\xi)} v\big\|_{\Ll_2(X^1)}
\ \le \
\big\|v- S_{\Lambda(\xi)} v\big\|_{\Ll_2(X^1)} 
+  
\sum_{\bs  \not\in \Lambda (\xi)}  
\|v_\bs\|_{X^1} \big\|I_{\Lambda(\xi) \cap R_\bs}\, H_\bs\big\|_{L_2(\RRi,\gamma)},
\end{equation}
where
\begin{equation} \nonumber
S_{\Lambda(\xi)} v 
:= \ 
\sum_{\bs\in {\Lambda(\xi)}} v_\bs H_\bs.
\end{equation}	
By estimating  $\big\|v- S_{\Lambda(\xi)} v\big\|_{\Ll_2(X^1)} $ and $\sum_{\bs  \not\in \Lambda (\xi)}  
\|v_\bs\|_{X^1} \big\|I_{\Lambda(\xi) \cap R_\bs}\, H_\bs\big\|_{L_2(\RRi,\gamma)}$ similarly to \eqref{ineq:|v-v_xi|} and  \eqref{xi-rate4}, respectively, we derive
\begin{equation} \nonumber
\big\|v- I_{\Lambda(\xi)} v\big\|_{\Ll_2(X^1)} 
\  \le \  C \xi^{-(1/q - 1/2)}.
\end{equation}
Since $|\Gamma_\bs| \le \prod_{j \in \NN} (2s_j + 1) = p_\bs(1,2)$, we have from the definition
\begin{equation} \label{Gamma(Lambda(xi))}
|\Gamma(\Lambda(\xi))|
\ \le  \ 
\sum_{\bs \in \Lambda(\xi)} |\Gamma_\bs|  
\ \le  \ 
\sum_{\xi \sigma_\bs^{-q} \ge 1} p_\bs(1,2) 
\ \le \ M \xi,
\end{equation}
where $M:= \sum_{\bs \in \FF} p_\bs(1,2)  \sigma_\bs^{-q}  < \infty$ by the assumption. For any $n \in \NN$, by choosing a number $\xi_n$ satisfying the inequalities 
$M \xi_n \le n <  2M \xi_n$,
we derive \eqref{u-I_Lambdau, p le 2}.
\hfill
\end{proof}

Similarly to Corollary~\ref{corollary[coll-approx]} we have the following

\begin{corollary}\label{corollary[coll-approx]ev}
	Let $v \in \Ll_2^\Ee(X)$ be represented by the series \eqref{HermiteSeriesVev} for a Hilbert space $X$.  
	Assume that
	$(Y_m)_{m \in \NN_0}$ is a sequence satisfying the condition \eqref{ineq[lambda_m]} for some positive numbers $\tau$ and $C$. 
	Assume that there exists an increasing sequence $(\sigma_\bs)_{\bs \in \FF_\rev}$ of numbers strictly larger than $1$ such that
	\begin{equation} \nonumber
	\sum_{\bs\in\FF_\rev} (\sigma_\bs \|v_\bs\|_X)^2 <\infty
	\end{equation}				
and $(p_{\bs}(2\theta, \max(2,\lambda))\sigma_\bs^{-1})_{\bs \in \FF_\rev} \in \ell_q(\FF_\rev)$  for some $0<q<2$, where $\theta$	and $\lambda$  are as in \eqref{theta,lambda2}. 
For $\xi>1$, define
\begin{equation} \label{Lambda_rev(xi)}
\Lambda_\rev(\xi):= \ \Lambda(\xi) \cap \FF_\rev
\ = \
\{\bs \in \FF_\rev: \, \sigma_\bs^q \le \xi\}.
\end{equation}
	
		Then there exists a constant $C$ such that for each $n \in \NN$, there exists a number $\xi_n$  such that  $|\Gamma(\Lambda_\rev(\xi_n))| \le n$ and 
	\begin{equation} \label{u-I_Lambdau, p le 2ev}
	\|v -I^*_{\Lambda_\rev(\xi_n)}v\|_{\Ll_p(X)} \leq Cn^{-(1/q - 1/2)}.
	\end{equation}
\end{corollary}

\begin{remark} \label{remark3.1}
{\rm	
(i) Theorem~\ref{thm[coll-approx]ev} and Corollary~\ref{corollary[coll-approx]ev} 
will be applied in proving the convergence rates of fully and non-fully discrete integration in the next section.

\noindent
(ii)
The bound $\|v -I_{\Lambda_n}v\|_{L_2(\RRi, \Hh,\gamma)} \leq Cn^{-(1/q - 1/2)}$ has been obtained in \cite[Theorem 3.14]{EST18} for a Hilbert space $\Hh$, where $\Lambda_n$ is the set of $\bs$ corresponding to the $n$ largest elements of an $\ell_q$-summable majorant  of  the sequence $(\sigma_\bs^{-1} \, p_\bs(\theta, \lambda))_{\bs \in \FF}$.

\noindent
(iii)
The operators $I_{\Lambda(\xi)}$ and $I^*_{\Lambda_\rev(\xi)}$ represent non-adaptive collocation methods of approximation of $v \in X^1$ based on the particular values $v(\by)$ at the points $\by$ in the grids $\Gamma(\Lambda(\xi))$ and $\Gamma(\Lambda_\rev(\xi))$, respectively. Moreover, the sparsity of $\Gamma(\Lambda_\rev(\xi))$ is much higher than that of $\Gamma(\Lambda(\xi))$: the generating set $\Lambda_\rev(\xi)$ contains only even indices of  $\Lambda(\xi)$. This remarkable property, in particular, plays an important role  in improving the rate of quadrature of the solution to the parameterized elliptic PDEs with lognormal inputs  \eqref{ellip}, see Corollary \ref{corollary[quadrature]pde} and its proof as well as Remark \ref{remark5.1}.
}
\end{remark}

\section{Integration}
\label{Integration}

In this section, we construct general linear fully discrete  methods  for integration  of functions taking values in $X^2$ and having a weighted $\ell_2$-summability of Hermite expansion coefficients  for Hilbert spaces $X^1$ and $X^2$ satisfying a certain "spatial" approximation property, and their bounded linear functionals. In particular,
we give convergence rates for these methods of integration which are derived from results on convergence rate of polynomial interpolation approximation in $\Vv_1(X^1)$  in Theorem~\ref{thm[coll-approx]ev}. We also briefly consider linear non-fully discrete methods  for integration.

  If $v$ is a function defined on $\RR$ taking values in a Hilbert space $X$,  the function  $I_m(v)$ in  \eqref{I_(v)}  generates the quadrature formula defined as
\begin{equation} \nonumber
Q_m(v)
:= \ \int_{\RR} I_m(v)(y) \, \rd \gamma(y)
\ = \
\sum_{k=0}^m\omega_{m;k}\, v(y_{m;k}),
\end{equation}	
where
\begin{equation} \nonumber
\omega_{m;k}
:=  \int_{\RR} \ell_{m;k}(y) \, \rd \gamma(y).
\end{equation}
Notice that 
\begin{equation} \nonumber
Q_m(\varphi)
\  = \ 
\int_{\RR} \varphi(y) \, \rd \gamma(y)
\end{equation}	
for every polynomial $\varphi$ of degree $\le m$, due to the identity $I_m(\varphi) = \varphi$. 

For integration purpose, we additionally assume that the sequence $Y_m$ as in \eqref{interpolation points} is symmetric for every $m \in \NN_0$, i. e., 	$y_{m;m-k} = - y_{m;k}$ for $k=0,...,m$. The sequences $Y_m^*$ of the of the roots of the Hermite polynomials $H_{m+1}$ and their modifications ${\bar Y}_m^*$  are symmetric. Also, for the sequence $Y_m^*$, it is well-known that
\begin{equation} \nonumber
\omega_{m;k}
\ = \
\frac{1}{(m+1) H_m^2(y_{m;k}^*)}.
\end{equation}

For a given sequence $(Y_m)_{m=0}^\infty$, we define the univariate operator $\Delta^{{\rm Q}}_m$ for even $m \in \NN_0$ by
\begin{equation} \nonumber
\Delta^{{\rm Q}}_m
:= \
Q_m - Q_{m-2},
\end{equation} 
with the convention $Q_{-2} := 0$. 

 For  a function $v \in \Ll_2^\Ee(X)$, we introduce the operator $\Delta^{{\rm Q}}_\bs$ defined for $\bs \in \FF_\rev$ by
\begin{equation} \nonumber
\Delta^{{\rm Q}}_\bs(v)
:= \
\bigotimes_{j \in \NN} \Delta^{{\rm Q}}_{s_j}(v),
\end{equation} 
where the univariate operator
$\Delta^{{\rm Q}}_{s_j}$ is applied to the univariate function $v$ by considering $v$ as a 
function of  variable $y_i$ with the other variables held fixed. As $\Delta^{{\rm I}}_\bs$, the operators  $\Delta^{{\rm Q}}_\bs$ are well-defined for all $\bs \in \FF_\rev$. 
For a finite set $\Lambda \subset \FF_\rev$, we introduce the quadrature operator $Q_\Lambda$ which is generated by the interpolation operator $I^*_\Lambda$ as follows
\begin{equation} \nonumber
 Q_\Lambda v
:= \
\sum_{\bs \in \Lambda} \Delta^{{\rm Q}}_\bs (v)
\ = \
\int_{\RRi} I^*_\Lambda v (\by)\, \rd \gamma(\by).
\end{equation} 
Further, if $\phi \in X'$ is a bounded linear functional on $X$, denote by $\langle \phi, v \rangle$ the value of $\phi$ in $v$.  For a finite set $\Lambda \subset \FF_\rev$, the quadrature formula $Q_\Lambda v$ generates the quadrature formula $Q_\Lambda \langle \phi, v \rangle$ for integration of $\langle \phi, v \rangle$ by
\begin{equation} \nonumber
Q_\Lambda \langle \phi, v \rangle
:= \
\langle \phi, Q_\Lambda \rangle\,
\ = \
\int_{\RRi} \langle \phi, I^*_\Lambda v (\by) \rangle\, \rd \gamma(\by).
\end{equation} 

Let Assumption II hold for Hilbert spaces $X^1$ and $X^2$, and $v \in \Ll_2^\Ee(X^2)$. For a finite set $G \subset \NN_0 \times \FF_\rev$, we introduce the quadrature operator $\Qq_G$ which is generated by the interpolation operator $\Ii^*_G: \Ll_2^\Ee(X^2) \to \Vv(G)$, and  which is defined for $v$ by
\begin{equation}  \label{Qq=int}
\Qq_G v
:= \
\sum_{(k,\bs) \in G} \delta_k \Delta^{{\rm Q}}_\bs (v)
\ = \
\int_{\RRi} \Ii^*_G v (\by)\, \rd \gamma(\by).
\end{equation}
 Further, if $\phi \in (X^1)'$  is a bounded linear functional on $X^1$, for a finite set $G \subset \NN_0 \times \FF_\rev$, the quadrature formula $\Qq_G v$ generates the quadrature formula $\Qq_G \langle \phi, v \rangle$ for integration of $\langle \phi, v \rangle$ by
\begin{equation} \nonumber
\Qq_G \langle \phi, v \rangle
:= \
 \langle \phi, \Qq_G v  \rangle
 \ = \
\int_{\RRi} \langle \phi, \Ii^*_G v (\by) \rangle\, \rd \gamma(\by).
\end{equation}

For a function  $v \in \Ll_2^\Ee(X)$ and is represented by the series \eqref{series}, consider the function $v_\rev \in~\Ll_2^\Ee(X)$ defined by
\begin{equation} \nonumber
v_\rev := \ \sum_{\bs\in\FF_\rev} v_\bs H_\bs.
\end{equation}	
Notice that
\begin{equation} \label{int v = int v_ev}
\int_{\RRi} v(\by) \, \rd \gamma(\by)
\ = \
\int_{\RRi} v_\rev(\by) \, \rd \gamma(\by).
\end{equation}
Moreover,
if $Y_m$ is symmetric for every $m \in \NN_0$,
\begin{equation} \label{Delta^QH_s=0}
 \Delta^{{\rm Q}}_{\bs'} H_\bs(\by) 
\ = \
0, \quad \bs \notin \FF_\rev, \ \bs' \in \FF.
\end{equation}

\begin{theorem}\label{thm[quadrature]} Under the hypothesis of Theorem \ref{thm[coll-approx]}, assume additionally that the sequences $Y_m$, $m \in \NN_0$, are symmetric. For $\xi>1$, let $G_\rev(\xi)$ be the set defined as in \eqref{G_rev(xi)}. Then  we have the following.
	
	\begin{itemize}
		\item[{\rm (i)}]
	Then there exists a constant $C$ such that for  each $n \in \NN$, there exists a number $\xi_n$  such that  $\dim\Vv(G_\rev(\xi_n))\le n$ and 
\begin{equation} \label{u-Q_Gu}
\left\|\int_{\RRi}v(\by)\, \rd \gamma(\by ) - \Qq_{G_\rev(\xi_n)}v\right\|_{X^1} \leq C \times
	\begin{cases}
		n^{-\alpha} &\text{if } \alpha \leq  1/q_2 - 1/2,
		\\
		n^{-\beta} \log n &\text{if } \alpha >  1/q_2 - 1/2.
	\end{cases} 
	\end{equation}
		\item[{\rm (ii)}] 
Let  $\phi \in (X^1)'$ be a bounded linear functional on $X^1$. Then there exists a constant $C$ such that for  each $n \in \NN$, there exists a number $\xi_n$  such that  $\dim\Vv(G_\rev(\xi_n))\le n$ and 
\begin{equation} \label{u-Q_Gu_phi}
\left|\int_{\RRi} 
\langle \phi,  v (\by) \rangle\, \rd \gamma(\by ) - \Qq_{G_\rev(\xi_n)} \langle \phi,  v \rangle\right| 
\leq C \|\phi\|_{(X^1)'}\times
	\begin{cases}
		n^{-\alpha} &\text{if } \alpha \leq  1/q_2 - 1/2,
		\\
		n^{-\beta} \log n &\text{if } \alpha >  1/q_2 - 1/2.
	\end{cases} 
\end{equation}
\end{itemize}

	The rate $\alpha$ corresponds to the approximation of a single function in $X^2$ as given by \eqref{spatialappnX}. The rate $\beta$ is given by \eqref{[beta]1}.
\end{theorem}

\begin{proof}	
For a given $n \in \NN$, we approximate the integral in the right-hand side of \eqref{int v = int v_ev} by $\Qq_{G_\rev(\xi_n)}$ where $\xi_n$ is as in Theorem \ref{thm[coll-approx]ev}. By Lemmata~\ref{lemma[AbsConv]} and \ref{lemma[AbsConv]2} the series \eqref{HermiteSeriesV} and \eqref{[delta-series]} 
converge absolutely, and therefore, unconditionally  in  the Hilbert space $\Ll_2(X^1)$ to $v$. Hence, by  \eqref{Delta^QH_s=0} we derive that $\Qq_{G_\rev(\xi_n)}v = \Qq_{G_\rev(\xi_n)}v_\rev$.
Due to \eqref{Qq=int} and  \eqref{int v = int v_ev}  there holds the equality
	\begin{equation} \label{eq[int_RRi]}
\int_{\RRi} v(\by)\, \rd \gamma(\by ) - \Qq_{G_\rev(\xi_n)}v 
\ = \
\int_{\RRi} \left(v_\rev(\by) - \Ii^*_{G_\rev(\xi_n)}v_\rev (\by)\right) \, \rd \gamma(\by).
\end{equation}
Hence, applying \eqref{u-I_Gu, p le 2ev} in Theorem \ref{thm[coll-approx]ev} for $p=1$, we obtain (i):
	\begin{equation} \nonumber
		\begin{split}
\left\|\int_{\RRi} v(\by)\, \rd \gamma(\by ) - \Qq_{G_\rev(\xi_n)}v \right\|_{X^1}  
\ &\le \
\left\|v_\rev - \Ii^*_{G_\rev(\xi_n)}v_\rev\right\|_{\Ll_1(X^1)}
\\
\ &\le \
C \times
	\begin{cases}
		n^{-\alpha} &\text{if } \alpha \leq  1/q_2 - 1/2,
		\\
		n^{-\beta} \log n &\text{if } \alpha >  1/q_2 - 1/2
	\end{cases}.				
	\end{split}
\end{equation}

	For a given $n \in \NN$, we approximate the integral $\int_{\RRi} \langle \phi, v (\by) \rangle\, \, \rd \gamma(\by)$ by $\Qq_{\Lambda_\rev(\xi_n)} \langle \phi, v \rangle$ where $\xi_n$ is as in Corollary \ref{corollary[coll-approx]ev}.
Similarly to \eqref{eq[int_RRi]}, there holds the equality
\begin{equation} \nonumber
\int_{\RRi} \langle \phi, v_\rev (\by) \rangle\, \, \rd \gamma(\by ) - \Qq_{G_\rev(\xi_n)} \langle \phi, v_\rev(\by) \rangle
\ = \
\int_{\RRi} \langle \phi, v_\rev (\by)  - \Ii^*_{G_\rev(\xi_n)}  v_\rev(\by) \rangle \, \rd \gamma(\by).
\end{equation}
Hence, applying \eqref{u-I_Gu, p le 2ev} in Theorem \ref{thm[coll-approx]ev} for $p=1$, we prove (ii):
\begin{equation} \nonumber
\begin{split}
\left|\int_{\RRi} \langle \phi, v_\rev (\by) \rangle\, \, \rd \gamma(\by ) - \Qq_{G_\rev(\xi_n)}  \langle \phi, v \rangle \right|
\ &\le \
\int_{\RRi} \left|\langle \phi, v_\rev (\by)  - \Ii^*_{G_\rev(\xi_n)}   v_\rev(\by) \rangle \right| \, \rd \gamma(\by)
\\[1.5ex]
\ &\le \
\int_{\RRi} \|\phi\|_{(X^1)'} \|v_\rev (\by)  - \Ii^*_{G_\rev(\xi_n)}   v_\rev(\by)\|_{X^1} \, \rd \gamma(\by)
\\[1.5ex]
\ &\le \
C  \|\phi\|_{(X^1)'} \|v_\rev  - \Ii^*_{G_\rev(\xi_n)} v_\rev \|_{\Ll_1(X^1)} 
\\[1.5ex]
\ &\le \
C \|\phi\|_{(X^1)'} \times
	\begin{cases}
		n^{-\alpha} &\text{if } \alpha \leq  1/q_2 - 1/2,
		\\
		n^{-\beta} \log n &\text{if } \alpha >  1/q_2 - 1/2.
	\end{cases} 
\end{split}
\end{equation}
	\hfill
\end{proof}

Similarly to the proof of Theorem~\ref{thm[quadrature]}, applying \eqref{u-I_Lambdau, p le 2ev} in Corollary~\ref{corollary[coll-approx]ev} for $p=1$, we can derive the following
\begin{corollary}\label{corollary[quadrature]}
	Under the hypothesis of Corollary~\ref{corollary[coll-approx]}, assume additionally that the sequences $Y_m$, $m \in \NN_0$, are symmetric. For $\xi>1$, let $\Lambda_\rev(\xi)$ be the set defined as in \eqref{Lambda_rev(xi)}.Then we have the following.
	\begin{itemize}
	\item[{\rm (i)}]	
	Then there exists a constant $C$ such that for each $n \in \NN$, there exists a number $\xi_n$  such that  $|\Gamma(\Lambda_\rev(\xi_n))| \le n$ and 
\begin{equation} \label{u-Q_Lambdau}
\left\|\int_{\RRi}v(\by)\, \rd \gamma(\by ) - Q_{\Lambda_\rev(\xi_n)}v\right\|_X
\ \le \
Cn^{-(1/q - 1/2)}.
\end{equation}	
	\item[{\rm (ii)}] 
	Let  $\phi \in X'$ be a bounded linear functional on $X$.		Then there exists a constant $C$ such that for each $n \in \NN$, there exists a number $\xi_n$  such that  $|\Gamma(\Lambda_\rev(\xi_n))| \le n$ and 
	\begin{equation} \label{u-Q_Lambdau_phi}
	\left|\int_{\RRi} \langle \phi, v (\by) \rangle\, \, \rd \gamma(\by ) - Q_{\Lambda_\rev(\xi_n)} \langle \phi, v \rangle \right|
	\ \le \
	C\|\phi\|_{X'} n^{-(1/q - 1/2)}.
	\end{equation}
\end{itemize}	
\end{corollary}


We make use the following notation: for $\nu \in \NN$,
\begin{equation} \nonumber
	\FF_\nu := \{\bs \in \FF: s_j \in \NN_{0,\nu},  \ j \in \NN \};
	\quad
	\NN_{0,\nu} := \{n \in \NN_0: n = 0, \nu,\nu+1, ...\}.
\end{equation}
The set $\FF_\nu$ has been introduced in \cite{ZS17}.  The set   $\FF_2$ plays an important role in establishing improved  convergence rates for sparse-grid Smolyak quadrature in \cite{ZDS18,ZS17}.

There is another way to construct quadrature operators which perform the same convergence rates as in \eqref{u-Q_Gu}, \eqref{u-Q_Gu_phi} and \eqref{u-Q_Lambdau}, \eqref{u-Q_Lambdau_phi}. We briefly consider it. 
For a given sequence $(Y_m)_{m=0}^\infty$, we define the univariate operator $\Delta^{{\rm Q}}_m$ for $m \in \NN_{0,2}$ by
\begin{equation} \nonumber
	\tilde{\Delta}^{{\rm Q}}_m
	:=
	Q_m - Q_{m-1},  \ \ \text{if} \ m\not=2, \ \ \ \text{and } \ \ \Delta^{{\rm Q}}_2:= Q_2 - Q_0,
\end{equation} 
with the convention $Q_{-1} := 0$. 

Define for $\xi>1$,
\begin{equation} \label{G_2(xi)}
	G_2(\xi)
	:= \ 
	\begin{cases}
		\big\{(k,\bs) \in \NN_0 \times\FF_2: \, 2^k \sigma_{2;\bs}^{q_2} \leq \xi\big\} \quad &{\rm if }  \ \alpha \le 1/q_2 - 1/2;\\
		\big\{(k,\bs) \in \NN_0 \times\FF_2: \, \sigma_{1;\bs}^{q_1} \le \xi , \  
		2^{\tau  k} \sigma_{2;\bs}\leq \xi^\vartheta \big\} \quad  & {\rm if }  \ \alpha > 1/q_2 - 1/2,
	\end{cases}
\end{equation}
where $\tau$ and  $\vartheta$ are as in \eqref{kappa}.	

The operators $\tilde{\Delta}^{{\rm Q}}_{2\bs}$ for $\bs \in \FF_2$, $\tilde{Q}_\Lambda$ for a finite set $\Lambda \subset \FF_2$, and $\tilde{\Qq}_G$ for a  finite set $G \subset \NN_0 \times \FF_2$, are defined in similar way  as the operators $\Delta^{{\rm Q}}_{2\bs}$, $Q_\Lambda$ and $\Qq_G$, respectively,  by  replacing $\Delta^{{\rm Q}}_{s_j}$ with $\tilde{\Delta}^{{\rm Q}}_{s_j}$, $j \in \NN$.

\begin{theorem}\label{thm[quadrature2]} 
	Under the hypothesis of Theorem \ref{thm[coll-approx]}, assume additionally that the sequences $Y_m$, $m \in \NN_0$, are symmetric. 
	
	Then  we have the following.
	
	\begin{itemize}
		\item[{\rm (i)}]
		Then there exists a constant $C$ such that  each $n \in \NN$ there exists a number $\xi_n$  such that  $\dim\Vv(G_2(\xi_n))\le n$ and 
		\begin{equation} \label{u-Q_Gu2}
			\left\|\int_{\RRi}v(\by)\, \rd \gamma(\by ) - \tilde{\Qq}_{G_2(\xi_n)}v\right\|_{X^1} \leq C \times
	\begin{cases}
		n^{-\alpha} &\text{if } \alpha \leq  1/q_2 - 1/2,
		\\
		n^{-\beta} \log n &\text{if } \alpha >  1/q_2 - 1/2.
	\end{cases} 
		\end{equation}
		\item[{\rm (ii)}] 
		Let  $\phi \in (X^1)'$ be a bounded linear functional on $X^1$. 	Then there exists a constant $C$ 
		for each $n \in \NN$ there exists a number $\xi_n$  such that  $\dim\Vv(G_2(\xi_n))\le n$ and 
		\begin{equation} \label{u-Q_Gu_phi2}
			\left|\int_{\RRi} \langle \phi,  v (\by) \rangle\, \rd \gamma(\by ) - \tilde{\Qq}_{G_2(\xi_n)} \langle \phi,  v \rangle\right| \leq C \|\phi\|_{(X^1)'}  \times 
	\begin{cases}
		n^{-\alpha} &\text{if } \alpha \leq  1/q_2 - 1/2,
		\\
		n^{-\beta} \log n &\text{if } \alpha >  1/q_2 - 1/2.
	\end{cases} 
		\end{equation}
	\end{itemize}

	The rate $\alpha$ corresponds to the approximation of a single function in $X^2$ as given by \eqref{spatialappnX}. The rate $\beta$ is given by \eqref{[beta]1}.
\end{theorem}

\begin{corollary}\label{corollary[quadrature2]}
	Under the hypothesis of Corollary~\ref{corollary[coll-approx]}, assume additionally that the sequences $Y_m$, $m \in \NN_0$, are symmetric. 
	For $\xi>1$, define
	\begin{equation} \label{Lambda_2(xi)}
		\Lambda_2(\xi):= \{\bs \in \FF_2: \, \sigma_\bs^q \le \xi\}.
	\end{equation}	
	Then we have the following.
	\begin{itemize}
		\item[{\rm (i)}]	
		For each $n \in \NN$ there exists a number $\xi_n$  such that  $|\Gamma(\Lambda(\xi_n))| \le n$ and 
		\begin{equation} \label{u-Q_Lambdau2}
			\left\|\int_{\RRi}v(\by)\, \rd \gamma(\by ) - \tilde{Q}_{\Lambda_2(\xi_n)}v\right\|_X
			\ \le \
			Cn^{-(1/q - 1/2)}.
		\end{equation}	
		\item[{\rm (ii)}] 
		Let  $\phi \in X'$ be a bounded linear functional on $X$.		Then there exists a constant $C$ such that for each $n \in \NN$ there exists a number $\xi_n$  such that  $|\Gamma(\Lambda_2(\xi_n))| \le n$ and 
		\begin{equation} \label{u-Q_Lambdau_phi2}
			\left|\int_{\RRi} \langle \phi, v (\by) \rangle\, \, \rd \gamma(\by ) - \tilde{Q}_{\Lambda_2(\xi_n)} \langle \phi, v \rangle \right|
			\ \le \
			C\|\phi\|_{X'} n^{-(1/q - 1/2)}.
		\end{equation}
	\end{itemize}	
\end{corollary}

Theorem \ref{thm[quadrature2]} and Corollary \ref{corollary[quadrature2]} can be proven in a similar manner as Theorem \ref{thm[quadrature]} and Corollary \ref{corollary[quadrature]} with a modification.

%

\section{Elliptic PDEs with lognormal inputs}
\label {lognormal inputs}
In this section, we apply the results  in Sections~\ref{Linear Galerkin approximation}--\ref{Integration} to Hermite GPC expansion and polynomial interpolation  approximations as well as integration for  the parameterized diffusion elliptic equation \eqref{p-ellip} with lognormal inputs  \eqref{lognormal}.

We approximate the solution $u(\by)$ to  this equation by truncation of the Hermite series
\begin{equation} \nonumber
u(\by)=\sum_{\bs\in\FF} u_\bs \,H_\bs(\by), \quad u_\bs \in V.
\end{equation}

For convenience, we introduce the conventions: $W^1:= V$; $W^2:=W$; $H^0(D):= L_2(D)$; $W^{0,\infty}(D):= L_\infty(D)$. Constructions of fully discrete approximations and integration are based on the approximation property \eqref{spatialappn} in Assumption I and the weighted $\ell_2$-summability of  the series $(\|u_\bs\|_{W^r})_{\bs \in \FF}$, $r=1,2$  in  the following lemma which has been proven in \cite{BCDM17} for $r = 1$ and in  \cite{BCDS17} for $r=2$.

\begin{lemma}\label {lemma[ell_2summability]}
	Let $r=1,2$. 	Assume that the right side $f$ in \eqref{ellip}
	belongs to $H^{r-2}(D)$, that the domain $D$ has $C^{r-2,1}$ smoothness,
	that all functions $\psi_j$ belong to $W^{r-1,\infty}(D)$.
	Assume that  there exist a number $0<q_r<\infty$ and a sequence $\brho_r=(\rho_{r;j}) _{j \in \NN}$ of positive numbers such that $(\rho_{r;j}^{-1}) _{j \in \NN}\in \ell_{q_r}(\NN)$ and
	\begin{equation} \nonumber
\sup_{|\alpha|\leq r-1} 
	\left\| \sum _{j \in \NN} \rho_{r;j} |D^\alpha \psi_j| \right\|_{L^\infty(D)} 
	<\infty.
	\end{equation}
	
	Then we have that for any $\eta \in \NN$,
		\begin{equation} \label{sigma_r,s}
		\sum_{\bs\in\FF} (\sigma_{r;\bs} \|u_\bs\|_{W^r})^2 <\infty \quad \text{with} \quad 
		\sigma_{r;\bs}^2:=\sum_{\|\bs'\|_{\ell_\infty(\FF)}\leq \eta}{\bs\choose \bs'} \brho_r^{2\bs'}.
		\end{equation}
\end{lemma}

We need  two auxiliary lemmata.	
	\begin{lemma} \label{measurable}
		Let the assumptions of Lemma~\ref {lemma[ell_2summability]} hold  for the space $W^1$ with $0 <q_1 <2$.  Then the solution map 
		$\by \mapsto u(\by)$ is $\gamma$-measurable and $u \in \Ll_2(W^1)$.	 Moreover, $u \in \Ll_2^\Ee(W^1)$ where
		\begin{equation} \nonumber
		\Ee
		:= \
		\left\{\by \in \RRi: \  \sup_{j \in \NN} \rho_{1;j}^{-1} |y_j| < \infty \right\}
		\end{equation}
		 having $\gamma(\Ee) =1$ and  containing  all $\by \in \RRi$ with $|\by|_0 < \infty$.
	\end{lemma}
	
	\begin{proof}
		The proof of this lemma already is in \cite{BCDM17}. Indeed, under the assumptions of Lemma~\ref {lemma[ell_2summability]}  for the space $W^1$ with $0 <q_1 <2$, by \cite[Remark~2.5]{BCDM17}  Assumption {\bf A} in \cite[Page~349]{BCDM17} holds for the sequence $\brho_1=(\rho_{1;j}) _{j \in \NN}$. Hence, by  
		\cite[Corollary 2.3]{BCDM17}  the solution map $\by \mapsto u(\by)$ is $\gamma$-measurable and $u \in \Ll_2(W^1)$.  Moreover, $\gamma(\Ee) =1$ \cite[(2.23)]{BCDM17} and, obviously, $\Ee$ contains  all $\by \in \RRi$ with $|\by|_0 < \infty$.
		For a point  $\by \in \RRi$, by the Lax-Milgram lemma the solution $u(\by)$ is well-defined if $b(\by) \in L_\infty(D)$. This inclusion  holds true if $\by \in \Ee$ \cite[(2.26)]{BCDM17}. This means that $u \in \Ll_2^\Ee(W^1)$. 
		\hfill
	\end{proof}

The following lemma is a generalization of \cite[Lemma 5.1]{BCDM17}.
\begin{lemma} \label{lemma[bcdm]}
	Let $0 < q <\infty$, $\eta \in \NN$,	$\brho=(\rho_j) _{j \in \NN}$ of  positive numbers such the sequence $(\rho_j^{-1}) _{j \in \NN}$ belongs to $\ell_q(\NN)$. Let $\theta, \lambda$ be arbitrary positive numbers and  $(p_\bs(\theta,\lambda))_{\bs \in \FF}$ the sequence given in \eqref{[p_s]}. Let for numbers $\eta \in \NN$  the sequence $(\sigma_\bs)_{\bs \in \FF}$ be defined by
	\begin{equation} \nonumber
	\sigma_\bs^2  :=   \sum_{\|\bs'\|_{\ell_\infty(\FF)} \le \eta}\binom{\bs}{\bs'}\rho^{2\bs'}.
	\end{equation}	
	Then for any  $\eta > \frac{2\nu (\theta + 1)}{q}$, we have
	\begin{equation} \nonumber
	\sum_{\bs \in \FF_\nu} p_\bs(\theta,\lambda) \sigma_\bs^{-q/\nu} < \infty.
	\end{equation}
\end{lemma}

\begin{proof} 
	With $\theta' := 2\theta \nu/ q$, we have that
	\begin{equation} \nonumber
	\begin{split}
	\sum_{\bs \in \FF_\nu} p_\bs(\theta,\lambda) \sigma_\bs^{-q/\nu}  &=
	\sum_{\bs \in \FF_\nu}	\prod_{j \in \NN}\ \left( \sum_{k=0}^\eta \binom{s_j}{k} \ (1 + \lambda s_j )^{\theta'}\rho_j^{2k}\right)^{-q/2\nu} \\
	&= 
	\prod_{j \in \NN}\ \sum_{n \in \NN_{0,\nu}} \left(\sum_{k=0}^\eta \binom{n}{k} \ (1 + \lambda n )^{\theta'}\rho_j^{2k}\right)^{-q/2\nu}=: \prod_{j \in \NN} B_j,
	\end{split}
	\end{equation}
	and	
	\begin{equation} \nonumber
	\begin{split}
	B_j 
	&\le  
	\sum_{n \in \NN_{0,\nu}} \left( \binom{n}{\min(n,\eta)} \ 
	(1 + \lambda n )^{\theta'}\rho_j^{2\min(n,\eta)}\right)^{-q/2\nu} \\
	&\le
	\sum_{n \in \NN_{0,\nu}, \ n < \eta} \left( \binom{n}{n} \ 
	(1 + \lambda n )^{\theta'}\rho_j^{2n}\right)^{-q/2\nu}
	+  	\sum_{ n \ge \eta} \left(  \ 
	(1 + \lambda n )^{\theta'}\rho_j^{2\eta}\right)^{-q/2\nu}\\
	&\le
	\sum_{n \in \NN_{0,\nu}, \ n < \eta} 
	(1 + \lambda n )^{\theta}\rho_j^{-nq/\nu}
	+  	\rho_j^{-\eta q/\nu}\sum_{ n \ge \eta} \ \binom{n}{\eta}^{-q/2\nu} \ 
	(1 + \lambda n )^{\theta} =: B_{j,1} + B_{j,2}. \\
	\end{split}
	\end{equation}
	We estimate $B_{j,1}$ and   $B_{j,2}$. 
	We have 	
	\begin{equation} \nonumber
	\begin{split}
	B_{j,1}
	&\le
	1 + \sum_{n=\nu}^{\eta - 1} 
	(1 + \lambda n )^{\theta}\rho_j^{-nq/\nu} 
	\le 
	1 + (1 + (\eta - 1)\lambda  )^{\theta}\sum_{n=\nu}^{\eta - 1} \rho_j^{-nq/\nu}.
	\end{split}
	\end{equation}
	From the inequalities $\left(\frac{n}{\eta}\right)^\eta \le \binom{n}{\eta} $ and $\eta q/2\nu - \theta >1$ we derive that
	\begin{equation} \nonumber
	\begin{split}
	B_{j,2}
	& \le
	\rho_j^{-\eta q/\nu}\sum_{ n \ge \eta}  \left(\frac{n}{\eta}\right)^{- \eta q/2\nu} 
	(1 + \lambda n )^{\theta}
	\le
	C
	\rho_j^{-\eta q/\nu}\sum_{ n \ge \eta}  n^{- (\eta q/2\nu - \theta)}
	\le
	C
	\rho_j^{-\eta q/\nu}.
	\end{split}
	\end{equation}
	Summing up we obtain that 
	\begin{equation} \nonumber
	\begin{split}
	B_j \le
	B_{j,1} + B_{j,2}
	&\le
	1 + C \sum_{n=\nu}^\eta \rho_j^{-nq/\nu}.
	\end{split}
	\end{equation}
	Since the sequence $(\rho_j^{-1}) _{j \in \NN}$ belongs to $\ell_q(\NN)$, there exists $j^*$ large enough such that  $\rho_j > 1$ for all $j \ge j^*$. Hence, there exists a constant $C$ independent of $j$ such that
	$B_j \le 1 + C \rho_j^{- q}$
	 for all $j \in \NN$, and consequently,
	\begin{equation} \nonumber
	\sum_{\bs \in \FF_\nu} p_\bs(\theta,\lambda) \sigma_\bs^{-q/2\nu}    \le
	\prod_{j \in \NN} B_j  \le 
	\prod_{j \in \NN} (1 + C  \rho_j^{-q}) \le
	\exp\Big(\sum_{j \in \NN} C \rho_j^{-q}\Big) < \infty.
	\end{equation}
	\hfill
\end{proof}

 In the present paper, as noticed in Introduction we want to  show posibilities of non-adaptive approximation methods and  convergence rates of approximation by such methods for the parametrized diffusion elliptic equation \eqref{p-ellip} with lognormal inputs.  Here we do not consider Galerkin approximations. 
To treat fully discrete approximations we assume that $f \in L_2(D)$ and that it holds the  approximation property \eqref{spatialappn} in Assumption I  for all $ v \in W$, see, for instance, \cite[Theorem 3.2.1]{Cia78} for the case when $D$ is a polygonal set.
Notice  that classical error estimates yield
the convergence rate $\alpha=1/d$ by using Lagrange finite elements of order at least $1$
on quasi-uniform partitions. Also, the spaces $W$ do not always coincide with $H^2(D)$. 
For example, for $d=2$, we know that $W$ is strictly larger than $H^2(D)$
when $D$ is a polygon with re-entrant corner. In this case, it is well known that 
the optimal rate $\alpha=1/2$ is yet attained, 
when using spaces $V_n$ associated to meshes $(\Tt_n)_{n>0}$ with proper refinement near the re-entrant corners where the functions $v\in W$ might have singularities.

\begin{theorem}\label{thm[L_2-approx]pde}
	Let $0< p \le 2$. Let Assumption I hold. Let the assumptions of Lemma~\ref {lemma[ell_2summability]} hold  for the spaces $W^1=V$ and $W^2=W$
 with some $0<q_1\leq q_2 <\infty$.	For $\xi>1$, let $G(\xi)$ be the set defined by \eqref{G(xi)G} for $\sigma_{r;\bs}$ as in \eqref{sigma_r,s}, $r=1,2$. 
 
	Then there exists a constant $C$ such that for each $n \in \NN$ there exists a number $\xi_n$  such that   $\dim(\Vv(G(\xi_n)) \le n$ and
	\begin{equation} 	\label{L_2-rate-pde}
	\|u-\Ss_{G(\xi_n)}u\|_{\Ll_p(V)} \leq C \times
	\begin{cases}
		n^{-\alpha} &\text{if } \alpha \leq  1/q_2,
		\\
		n^{-\beta}  &\text{if } \alpha >  1/q_2.
	\end{cases} 
	\end{equation}
	The rate $\alpha$ corresponds to the spatial approximation of a single function in $W$ as given by \eqref{spatialappn}, and
	the rate $\beta$ is given by \eqref{[beta]2}.
\end{theorem}

\begin{proof} 
To prove the theorem it is sufficient to notice that the assumptions of Theorem~\ref{thm[L_2-approx]} are satisfied for $X^1 = V$ and $X^2 = W$. This can be done by using Lemmata~\ref {lemma[ell_2summability]} -- \ref{lemma[bcdm]}. (By multiplying the sequences $\brho_r$ in Lemma~\ref {lemma[ell_2summability]} with a positive constant we can get $\sigma_{r;\bs} > 1$ for $\bs \in \FF$.)
	\hfill
\end{proof}

\begin{remark}
{\rm	
(i) The rate $\min(\alpha,\beta)$ in \eqref{L_2-rate-pde} is the rate of best adaptive  $n$-term Hermite GPC expansion  approximation in $\Ll_2(V)$ based on $\ell_{p_1}$-summability of $(\|u_\bs\|_V)_{\bs \in \FF}$ and $\ell_{p_2}$-summability of $(\|u_\bs\|_W)_{\bs \in \FF}$ proven in \cite{BCDS17}, where  $1/p_r = 1/q_r + 1/2$ for $r =1,2$.

\noindent
(ii) Observe that  $\Ss_{G(\xi_n)}$  can be represented in the form of a multilevel approximation method with $k_n$ levels:
\begin{equation}  \nonumber
\Ss_{G (\xi_n)} 
\ = \ 
\sum_{k=0}^{k_n} \delta_k S_{\Lambda_k(\xi_n)},
\end{equation}
where  $S_\Lambda u:= \sum_{\bs \in \Lambda} u_\bs$ for $\Lambda \subset \FF$,  
$k_n:= \lfloor \log_2 \xi_n \rfloor$ if $\alpha \le 1/q_2$, and 
$k_n:= \lfloor (\alpha q_1)^{-1}\log_2 \xi_n \rfloor$ if $\alpha > 1/q_2$, and  for $k \in \NN_0$ and $\xi>1$,
\begin{equation} \label{Lambda_k(xi)}
\Lambda_k(\xi)
:= \ 
\begin{cases}
\big\{\bs \in \FF: \,\sigma_{2;\bs}^{q_2} \leq 2^{-k}\xi\big\} \quad &{\rm if }  \ 
\alpha \le 1/q_2 ;\\
\big\{\bs \in \FF: \, \sigma_{1;\bs}^{q_1} \le \xi, \  
\sigma_{2;\bs}^{q_1} \leq 2^{- \alpha q_1 k}\xi\big\} \quad  & {\rm if }  \ 
\alpha > 1/q_2 .
\end{cases}
\end{equation}
}
\end{remark}

	\begin{remark}
{\rm		
Since the index set $G(\xi)$ defined as in \eqref{G(xi)G} plays a key role in determining the operator $\Ss_{G (\xi)}$, we give an algorithm for constructing it, for instance, for the case $\alpha > 1/q_2$. The case $\alpha \le 1/q_2$ can be done similarly. We additionally assume that the sequences $\rho_{r;\bs}$, $r =1, 2$, are monotonically increasing. This assumption yields that
 if $\bs \in \FF$ and $i < j$ are such that $s_i = s_j = 0$, than $\sigma_{r;\bs+\be_i} \le \sigma_{r;\bs+\be_j}$, $r =1, 2$, where $\be_j := (\delta_{i,j})_{i \in \NN} \in \FF$. 
Observe that  
\begin{equation}  \nonumber
G (\xi) 
\ = \ 
\bigcup_{k=0}^{k_\xi} \big\{(k,\bs) \in \NN_0 \times\FF: \, \bs \in \Lambda_k(\xi)\big\},
\end{equation}
where  $k_\xi:= \lfloor \log_2 \xi \rfloor$ if $\alpha \le 1/q_2$, and 
$k_\xi:= \lfloor (\alpha q_1)^{-1}\log_2 \xi\rfloor$ if $\alpha > 1/q_2$, and $\Lambda_k(\xi)$ is defined as in \eqref{Lambda_k(xi)}, $\xi>1$.
Moreover,  $\Lambda_k(\xi)$ are downward closed sets, and consequently, the sequence $\big\{\Lambda_k(\xi)\big\}_{k=0}^{k_\xi}$ is nested in the inverse order, i.e., $\Lambda_{k'}(\xi) \subset \Lambda_k(\xi)$ if $k' > k$, and 
$\Lambda_0(\xi)$ is the largest and $\Lambda_{k_\xi}(\xi)$ is the smallest.
Hence,  the index set $G(\xi)$ can be constructed as in Algorithm~\ref{algo}.
}

\begin{algorithm}[H] \label{algo}
	\DontPrintSemicolon
	
	$\bs = \bzero$ 
	
	$G(\xi)=\emptyset$
	
	\If{$\sigma_{1,\bs}\leq \xi^{1/q_1} \AND \sigma_{2,\bs}\leq \xi^{1/q_1}$}
	{
		$G(\xi) \leftarrow \{(0,\bs)\}$
		
		$k\leftarrow 1$
		
		\While{$ \sigma_{2,\bs}\leq 2^{-\alpha k}\xi^{1/q_1}$}
		{
			$G(\xi)\leftarrow G(\xi) \cup \{(k,\bs) \}$
			
			$k\leftarrow k+1$
		}
	}
	\While{\rm True}
	{
		$d\leftarrow 1$
		
		\While{$\sigma_{1,\bs+\be_d}>\xi^{1/q_1}\OR \sigma_{2,\bs+\be_d}>\xi^{1/q_1}$}
		{
			\If{$s_d\not=0$}
			{
				$s_d\leftarrow 0$
				
				$d\leftarrow d+1$
			}
			\ElseIf{$\bs\not =\bzero$}
			{
				$d=\min\{j\in \mathbb{N}:\ s_j\not =0\}$
			}
			\Else
			{
				{\rm break}
			}
		}
		
		$\bs\leftarrow\bs+\be_d$
		
		$G(\xi)\leftarrow G(\xi) \cup \{(0,\bs) \}$
		
		$k\leftarrow 1$
		
		\While{$ \sigma_{2,\bs}\leq 2^{-\alpha k}\xi^{1/q_1}$}
		{
			$G(\xi)\leftarrow G(\xi) \cup \{(k,\bs) \}$
			
			$k\leftarrow k+1$
		}
	}
	\Return $G(\xi)$

	\caption{Constructing $G(\xi)$}
\end{algorithm}

{\rm
Let us estimate the computational complexity of Algorithm 1, by using some results from \cite[Lemmata 3.1.12 and 3.1.13]{Ze18}. Each  from 1st to 5th lines and 10th to 21st lines in this algorithm is executed at most 
$4|\Lambda_0(\xi)| + 1$ times. For every multi-index $\bs \in \Lambda_0(\xi)$ we store $\{(j,s_j)\, : s_j \not=0\}$. Each multi-index therefore occupies a memory of size $\Oo(d(\Lambda_0(\xi)))$, where  $d(\Lambda):= \sup_{\bs \in \Lambda}|\supp \bs|$. Assuming  elementary operations such as multiplications and divisions to be of complexity $\Oo(1)$, we can deduce that the computational complexity executing each  from 1st to 5th lines and 10th to 21st lines in Algorithm 1 is bounded by $\Oo(m(\Lambda_0(\xi)) + d(\Lambda_0(\xi)))$, where 
 $m(\Lambda):= \sup_{\bs \in \Lambda}|\sum_{j \in \NN} s_j|$.  Algorithm 1 terminates and gives $G(\xi)$ before $k > k_\xi$. Hence the overall computational complexity and memory consumption of Algorithm 1 behave like
\[
\Oo\big(k_\xi |\Lambda_0(\xi)| (m(\Lambda_0(\xi)) + d(\Lambda_0(\xi)))\big)
\ = \ \Oo\Bigg(\log_2 \xi \sum_{\bs \in  \Lambda_0(\xi)} p_\bs (1,1)\Bigg)
\ = \ \Oo(\xi\log_2 \xi).
\]
In the last step we used the equality $\sum_{\sigma_{1,\bs} \le \xi^{1/q_1}} = \Oo(\xi)$ which follows from the inequality 
$\sum_{\bs \in \FF}p_\bs (1,1)\sigma_{1,\bs}^{-q_1} < \infty$ (cf. \eqref{Gamma(Lambda(xi))}).
}
\end{remark}

\medskip
We now consider the problem of sparse-grid interpolation approximation and integration of the solution  $u(\by)$ to 
 the parameterized diffusion elliptic equation \eqref{p-ellip} with lognormal inputs. 
By using Lemmata \ref{lemma[ell_2summability]} -- \ref{lemma[bcdm]},   in the same way as the proof of Theorem \ref{thm[L_2-approx]pde}, from Theorems~\ref{thm[L_2-approx]} and \ref{thm[coll-approx]} and  Corollary~\ref{corollary[coll-approx]} we derive the following two theorems and two corollaries.  

\begin{theorem}\label{thm[coll-approx]pde}
Let $0 < p \le 2$.	Let Assumption I hold. Let the assumptions of Lemma~\ref {lemma[ell_2summability]} hold  for the spaces $W^1=V$ and $W^2=W$
	with some $0<q_1\leq q_2 <\infty$ with $q_1 <2$. Assume that
	$(Y_m)_{m \in \NN_0}$ is a sequence satisfying the condition \eqref{ineq[lambda_m]} for some positive numbers $\tau$ and $C$. For $\xi>1$, let $G(\xi)$ be the set defined by \eqref{G_(xi)} for  $\sigma_{r;\bs}$ as in \eqref{sigma_r,s}, $r=1,2$.  
	
		Then there exists a constant $C$ such that  for each $n \in \NN$, there exists a number $\xi_n$  such that  
for the interpolation operator $\Ii_{G(\xi_n)}: \Ll_2^\Ee(W) \to \Vv(G(\xi_n))$, we have that $\dim \Vv(G(\xi_n)) \le  n$ and 
	\begin{equation} \label{u-I_Gu, p le 2-pde}
	\|u -\Ii_{G(\xi_n)}u\|_{\Ll_p(V)} \leq C \times
	\begin{cases}
		n^{-\alpha} &\text{if } \alpha \leq  1/q_2 - 1/2,
		\\
		n^{-\beta} \log n &\text{if } \alpha >  1/q_2 - 1/2.
	\end{cases} 
	\end{equation}
	The rate $\alpha$ corresponds to the spatial approximation of a single function in $W$ as given by \eqref{spatialappn}.
	The rate $\beta$ is given by \eqref{[beta]1}.
\end{theorem}

\begin{remark}
{\rm	
(i) Observe that  $\Ii_{G(\xi_n)}$  can be represented in the form of a multilevel approximation method, see  Remarks \ref{rm 3.1a}(i) for details.

\noindent
(ii)
The fully discrete polynomial interpolation approximation by operators $\Ii_{G(\xi_n)}$ is a collocation approximation based on the finite number $|\Gamma(\Lambda_0(\xi_n))| \le \sum_{\bs \in \Lambda_0(\xi_n)}p_\bs(1,2)$ of the particular solvers  $u(\by)$, $\by \in \Gamma(\Lambda_0(\xi_n))$, where, we recall, $\Gamma(\Lambda_0(\xi_n)) = \cup_{\bs \in \Lambda_0(\xi_n)} \Gamma_\bs$ and 
$\Gamma_\bs = \{\by_{\bs - \be;\bm}: \, e \in \EE_\bs; \ m_j = 0,...,s_j - e_j, \ j \in \NN \}$ ($\EE_\bs$ denotes  the subset in $\FF$ of all $\be$ such that $e_j$ is $1$ or $0$ if $s_j > 0$, and $e_j$ is $0$ if $s_j = 0$, and $\by_{\bs;\bm}:= (y_{s_j;m_j})_{j \in \NN}$.)
}
\end{remark}

\begin{corollary}\label{corollary[collocation]pde} Let $0 < p \le 2$.
	Under the hypothesis of Lemma~\ref {lemma[ell_2summability]}  for the spaces $W^1=V$ with some $0 < q_1 = q  < 2$. For $\xi>1$, let $\Lambda(\xi)$ be the set defined by \eqref{Lambda(xi)} for  $\sigma_\bs=\sigma_{1;\bs}$ as in \eqref{sigma_r,s}. 	Then there exists a constant $C$ such that
	for each $n \in \NN$, there exists a number $\xi_n$  such that  $|\Gamma(\Lambda(\xi_n))| \le n$ and 
	\begin{equation} \label{u-I_Lambdau, p le 2pde}
	\|u -I_{\Lambda(\xi_n)}u\|_{\Ll_p(V)} \leq Cn^{-(1/q - 1/2)}.
	\end{equation}
\end{corollary}

The rate $1/q - 1/2$ in Corollary~\ref{corollary[collocation]pde} is much better than the rate $\frac{1}{2}(1/q- 1/2)$ which has been obtained in \cite[Theorem 3.18]{EST18} for a similar approximation in $\Ll_2(V)$.

 Similarly to  $\Ii_{G(\xi_n)}$, the approximation to $u$ by the operator $I_{\Lambda(\xi_n)}$,  is a collocation approximation based on the finite number $|\Gamma(\Lambda(\xi_n))| \le \sum_{\bs \in \Lambda(\xi_n)}p_\bs(1,2)$ of the particular solvers  $u(\by)$, $\by \in \Gamma(\Lambda(\xi_n))$.

For $\xi>1$, let the set $G_\rev^*(\xi)$ be  defined by 
\begin{equation} \label{G_rev(xi)2}
	G_\rev^*(\xi)
	:= \ 
	\begin{cases}
		\big\{(k,\bs) \in \NN_0 \times\FF_\rev: \, 2^k \sigma_{2;\bs}^{q_2/2} \leq \xi\big\} \ &{\rm if }  \ \alpha \le 2/q_2 - 1/2;\\
		\big\{(k,\bs) \in \NN_0 \times\FF_\rev: \, \sigma_{1;\bs}^{q_1/2} \le \xi , \  
		2^{\tau k} \sigma_{2;\bs}\leq \xi^\vartheta \big\} \  & {\rm if }  \ \alpha > 2/q_2 - 1/2,
	\end{cases}
\end{equation}
for  $\sigma_{r;\bs}$ as in \eqref{sigma_r,s}, $r=1,2$, where 
\begin{equation} \label{tau,vartheta}
	\tau:= \frac{4\alpha}{4 - q_2}, \qquad		
	\vartheta:= \frac{4}{4 - q_2} \left(\frac{2}{q_1} - \frac{1}{2}\right).
\end{equation}

\begin{theorem}\label{thm[quadrature]pde}
	Let Assumption I hold. Let the assumptions of Lemma~\ref {lemma[ell_2summability]} hold  for the spaces $W^1=V$ and $W^2=W$
	with some $0<q_1\leq q_2 <\infty$ with $q_1 < 4$. Assume that
	$(Y_m)_{m \in \NN_0}$ is a sequence satisfying the condition \eqref{ineq[lambda_m]} for some positive numbers $\tau$ and $C$, and such that $Y_m$ is symmetric for every 
	$m \in \NN_0$. Let the set $G_\rev^*(\xi)$ be  defined as in \eqref{G_rev(xi)2}.
	Then we have the following.
	\begin{itemize}
		\item[{\rm (i)}]
There exists a constant $C$ such that
for each $n \in \NN$, there exists a number $\xi_n$  such that  $\dim\Vv(G_\rev^*(\xi_n))\le n$ and 
		\begin{equation} \label{u-Q_Gu-quadrature}
		\left\|\int_{\RRi}v(\by)\, \rd \gamma(\by ) - \Qq_{G_\rev^*(\xi_n)}v\right\|_V \leq C \times
	\begin{cases}
		n^{-\alpha} &\text{if } \alpha \leq  2/q_2 - 1/2,
		\\
		n^{-\beta} \log n &\text{if } \alpha >  2/q_2 - 1/2.
	\end{cases} 
		\end{equation}
		\item[{\rm (ii)}] 
		Let  $\phi \in V'$ be a bounded linear functional on $V$. Then there exists a constant $C$ such that
		for each $n \in \NN$ there exists a number $\xi_n$  such that  $\dim\Vv(G_\rev^*(\xi_n))\le n$ and 
		\begin{equation} \label{u-Q_Gu_phi-pde}
		\left|\int_{\RRi} \langle \phi,  v (\by) \rangle\, \rd \gamma(\by ) - \Qq_{G_\rev^*(\xi_n)} \langle \phi,  v \rangle\right| \leq C \|\phi\|_{V'} \times
	\begin{cases}
		n^{-\alpha} &\text{if } \alpha \leq  2/q_2 - 1/2,
		\\
		n^{-\beta} \log n &\text{if } \alpha >  2/q_2 - 1/2.
	\end{cases} 
		\end{equation}
	\end{itemize}
	The rate $\alpha$ corresponds to the spatial approximation of a single function in $W$ as given by
	 \eqref{spatialappn}. The rate $\beta$ is given by 
	\begin{equation} 	\nonumber
\beta := \left(\frac 2 {q_1} - \frac{1}{2}\right)\frac{\alpha}{\alpha + \delta}, \ \ \
\delta := \frac 2 {q_1} - \frac 2 {q_2}.
\end{equation}	 
\end{theorem}

\begin{proof}
Observe that $\FF_\rev \subset \FF_2$.
From Lemma~\ref {lemma[ell_2summability]} and Lemma~\ref{lemma[bcdm]} we can see that 	the assumptions of Theorem~\ref{thm[coll-approx]} hold for $X^1=V$	and $X^2 = W$ with $0 < q_1/2 \le q_2/2 < \infty$ and $q_1/2 < 2$. Hence, by applying Theorem~\ref{thm[quadrature]} we prove the theorem.
	\hfill
\end{proof}

Observe that the rate in \eqref{u-Q_Gu-quadrature} and \eqref{u-Q_Gu_phi-pde} can be improved as 
$\min(\alpha,  \frac 2 {q_1} \frac{\alpha}{\alpha + \delta})$ if the sequences  $(\|u_\bs\|_V)_{\bs \in \FF}$ and $(\|u_\bs\|_{W^r})_{\bs \in \FF}$ have $\ell_{p_1}$- and  $\ell_{p_r}$-summable majorant sequences, respectively, where  $1/p_1 = 1/q_1 + 1/2$ and $1/p_r = 1/q_r + 1/2$. 
Similarly to $\Ii_{G(\xi_n)}$, the quadrature operator  $\Qq_{G_\rev^*(\xi_n)}$ 
can be represented in the form of a multilevel integration method with $k_n$ levels:
\begin{equation}  \nonumber
\Qq_{G_\rev ^*(\xi_n)} 
\ = \ 
\sum_{k=0}^{k_n} \delta_k Q_{\Lambda_{\rev;k}^*(\xi_n)}, 
\end{equation}
where  $k_n:= \lfloor \log_2 \xi_n \rfloor$ if $\alpha \le 1/q_2 - 1/2$, and 
$k_n:= \lfloor \vartheta \tau^{-1}\log_2 \xi_n \rfloor$ if $\alpha > 1/q_2 - 1/2$, and  for $k \in \NN_0$ and $\xi>1$,
\begin{equation} \nonumber
\Lambda_{\rev;k}^*(\xi)
:= \ 
\begin{cases}
\big\{\bs \in \FF_\rev: \,\sigma_{2;\bs}^{q_2/2} \leq 2^{-k}\xi\big\} \quad &{\rm if }  \ 
\alpha \le 1/q_2 - 1/2;\\
\big\{\bs \in \FF_\rev: \, \sigma_{1;\bs}^{q_1/2} \le \xi , \  
\sigma_{2;\bs} \leq 2^{- \tau k}\xi^\vartheta\big\} \quad  & {\rm if }  \ 
\alpha > 1/q_2 - 1/2.
\end{cases}
\end{equation}

For $\xi>1$, let the set $\Lambda_\rev^*(\xi)$ be defined by 
\begin{equation} \label{Lambda_rev(xi)2}
	\Lambda_\rev^*(\xi):= 
	\{\bs \in \FF_\rev: \, \sigma_\bs^{q/2} \le \xi\}.
\end{equation}

In the same way, from  Corollary~\ref{corollary[quadrature]} we derive the following  

\begin{corollary}\label{corollary[quadrature]pde}
	Let the assumptions of Lemma~\ref {lemma[ell_2summability]} hold  for the spaces $W^1=V$ 
	with some $0<q_1 = q < 4$. Assume that
	$(Y_m)_{m \in \NN_0}$ is a sequence satisfying the condition \eqref{ineq[lambda_m]} for some positive numbers $\tau$ and $C$, and such that $Y_m$ is symmetric for every $m \in \NN_0$. For $\xi>1$, let the set $\Lambda_\rev^*(\xi)$ be defined by 
	\eqref{Lambda_rev(xi)2}
	 for $\sigma_\bs=\sigma_{1;\bs}$ as in \eqref{sigma_r,s}. Then we have the following.
	\begin{itemize}
	\item[{\rm (i)}]	
	There exists a constant $C$ such that for each  $n \in \NN$, there exists a number $\xi_n$  such that  $|\Gamma(\Lambda_\rev^*(\xi_n))| \le n$ and 
\begin{equation} \label{u-Q_Lambdau_pde}
\left\|\int_{\RRi}u(\by)\, \rd \gamma(\by ) - Q_{\Lambda_\rev^*(\xi_n)}u\right\|_V
\ \le \
Cn^{-(2/q - 1/2)}.
\end{equation}			
	\item[{\rm (ii)}] 
	Let  $\phi \in V'$ a bounded linear functional on $V$. 
	There exists a constant $C$ such that for each $n \in \NN$, there exists a number $\xi_n$  such that  $|\Gamma(\Lambda_\rev^*(\xi_n))| \le n$ and 
	\begin{equation} \label{u-Q_Lambdau_phi-pde}
	\left|\int_{\RRi} 
	\langle \phi, u (\by) \rangle\, \, \rd \gamma(\by ) - Q_{\Lambda_\rev^*(\xi_n)} \langle \phi, u \rangle \right|
	\ \le \
	C\|\phi\|_{V'} n^{-(2/q - 1/2)}.
	\end{equation}
\end{itemize}		  
\end{corollary}

In a similar way, from Theorem \ref{thm[quadrature2]} and Corollary \ref{corollary[quadrature2]} we obtain

\begin{theorem}\label{thm[quadrature]pde2}
	Let Assumption I hold. Let the assumptions of Lemma~\ref {lemma[ell_2summability]} hold  for the spaces $W^1=V$ and $W^2=W$
	with some $0<q_1\leq q_2 <\infty$ with $q_1 < 4$. Assume that
	$(Y_m)_{m \in \NN_0}$ is a sequence satisfying the condition \eqref{ineq[lambda_m]} for some positive numbers $\tau$ and $C$, and such that $Y_m$ is symmetric for every 
	$m \in \NN_0$. 	For $\xi>1$, let $G_2^*(\xi)$ be the set defined by
	\begin{equation} \label{G_2(xi)2}
		G_2^*(\xi)
		:= \ 
		\begin{cases}
			\big\{(k,\bs) \in \NN_0 \times\FF_2: \, 2^k \sigma_{2;\bs}^{q_2/2} \leq \xi\big\} \quad &{\rm if }  \ \alpha \le 2/q_2 - 1/2;\\
			\big\{(k,\bs) \in \NN_0 \times\FF_2: \, \sigma_{1;\bs}^{q_1/2} \le \xi , \  
			2^{\tau  k} \sigma_{2;\bs}\leq \xi^\vartheta \big\} \quad  & {\rm if }  \ \alpha > 2/q_2 - 1/2,
		\end{cases}
	\end{equation}
	 for
	$\sigma_{r;\bs}$ as in \eqref{sigma_r,s}, $r=1,2$,
	where $\tau$ and  $\vartheta$ are as in \eqref{tau,vartheta}.

	Then  we have the following.
	\begin{itemize}
		\item[{\rm (i)}]
			There exists a constant $C$ such that for each $n \in \NN$, there exists a number $\xi_n$  such that  $\dim\Vv(G_2^*(\xi_n))\le n$ and 
		\begin{equation} \label{u-Q_Gu-quadrature2}
			\left\|\int_{\RRi}v(\by)\, \rd \gamma(\by ) - \tilde{\Qq}_{G_2^*(\xi_n)}v\right\|_V \leq C \times
	\begin{cases}
		n^{-\alpha} &\text{if } \alpha \leq  2/q_2 - 1/2,
		\\
		n^{-\beta} \log n  &\text{if } \alpha >  2/q_2 - 1/2.
	\end{cases} 
		\end{equation}
		\item[{\rm (ii)}] 
		Let  $\phi \in V'$ be a bounded linear functional on $V$. 	Then there exists a constant $C$ such that for each $n \in \NN$, there exists a number $\xi_n$  such that  $\dim\Vv(G_2^*(\xi_n))\le n$ and 
		\begin{equation} \label{u-Q_Gu_phi-pde2}
			\left|\int_{\RRi} \langle \phi,  v (\by) \rangle\, \rd \gamma(\by ) - \tilde{\Qq}_{G_2^*(\xi_n)} \langle \phi,  v \rangle\right| \leq C \|\phi\|_{V'}  \times 
	\begin{cases}
		n^{-\alpha} &\text{if } \alpha \leq  2/q_2 - 1/2,
		\\
		n^{-\beta} \log n  &\text{if } \alpha >  2/q_2 - 1/2.
	\end{cases} 
		\end{equation}
	\end{itemize}
	The rate $\alpha$ corresponds to the spatial approximation of a single function in $W$ as given by
	\eqref{spatialappn}. The rate $\beta$  are  given by 
	\begin{equation} 	\nonumber
		\beta := \left(\frac 2 {q_1} - \frac{1}{2}\right)\frac{\alpha}{\alpha + \delta}, \ \ \ 
		\delta := \frac 2 {q_1} - \frac 2 {q_2}.
	\end{equation}	 
\end{theorem}

\begin{corollary}\label{corollary[quadrature]pde2}
	Let the assumptions of Lemma~\ref {lemma[ell_2summability]} hold  for the spaces $W^1=V$ 
	with some $0<q_1 = q < 4$. Assume that
	$(Y_m)_{m \in \NN_0}$ is a sequence satisfying the condition \eqref{ineq[lambda_m]} for some positive numbers $\tau$ and $C$, and such that $Y_m$ is symmetric for every $m \in \NN_0$. For $\xi>1$, let $\Lambda_2^*(\xi)$ be the set defined by
	\begin{equation*} 
		\Lambda_2^*(\xi):= 
		\{\bs \in \FF_2: \, \sigma_\bs^{q/2} \le \xi\}
	\end{equation*}
	for $\sigma_\bs=\sigma_{1;\bs}$ as in \eqref{sigma_r,s}. Then we have the following.
	\begin{itemize}
		\item[{\rm (i)}]	
		There exists a constant $C$ such that  for each $n \in \NN$, there exists a number $\xi_n$  such that  $|\Gamma(\Lambda_2^*(\xi_n))| \le n$ and 
		\begin{equation} \label{u-Q_Lambdau_pde2}
			\left\|\int_{\RRi}u(\by)\, \rd \gamma(\by ) - \tilde{Q}_{\Lambda_2^*(\xi_n)}u\right\|_V
			\ \le \
			Cn^{-(2/q - 1/2)}.
		\end{equation}			
		\item[{\rm (ii)}] 
		Let  $\phi \in V'$ a bounded linear functional on $V$. 
		Then there exists a constant $C$ such that for each $n \in \NN$, there exists a number $\xi_n$  such that  $|\Gamma(\Lambda_2^*(\xi_n))| \le n$ and 
		\begin{equation} \label{u-Q_Lambdau_phi-pde2}
			\left|\int_{\RRi} \langle \phi, u (\by) \rangle\, \, \rd \gamma(\by ) - \tilde{Q}_{\Lambda_2^*(\xi_n)} \langle \phi, u \rangle \right|
			\ \le \
			C\|\phi\|_{V'}  n^{-(2/q - 1/2)}.
		\end{equation}
	\end{itemize}		  
\end{corollary}

\begin{remark} \label{remark5.1}
{\rm	
(i)  As noticed in Section \ref{Integration}, the sparsity of the grids $\Gamma(\Lambda_{\rev;0}^*(\xi))$ and $\Gamma(\Lambda_\rev^*(\xi))$ of the evaluation points  in the quadrature operators $\Qq_{G_\rev^*(\xi)}$ and $Q_{\Lambda_\rev^*(\xi)}$ are much higher than the sparsity of the  grids $\Gamma(\Lambda_0(\xi))$ and  $\Gamma(\Lambda(\xi))$ of the evaluation points  in the generating  interpolation operators $\Ii_{G(\xi)}$ and $I_{\Lambda(\xi)}$. 

\noindent 
(ii) The rate $2/q - 1/2$ in Corollary~\ref{corollary[quadrature]pde}  is a significant improvement of the rate $\frac{1}{2}(1/q - 1/2)$ which has been recently obtained in  \cite[Corllary 3.12]{Ch18}. 

\noindent 
(iii) Since the use and analysis of non-adaptive construction methods for sparse-grid interpolation are important, let us compare in details our methods in Corollary \ref{corollary[quadrature]pde} with those which has been also discussed in \cite{Ch18}.  To construct a quadrature of the form $Q_\Lambda$, the author of the last work used the set $\Lambda_m \subset \FF$ of all indices $\bs$ (including non-even) corresponding to the $m$ smallest values of $\sigma_\bs$. The number $n= n(m)$ of quadrature points in $\Lambda_m \subset \FF$ is estimated as $n \le C m^2$ \cite[Proposition 3.16]{EST18}. This lead to the rate  $\frac{1}{2}(1/q - 1/2)$. In the present paper, we used the set $\Lambda_\rev^*(\xi)\subset \FF_\rev$ of all only even indices $\bs$ by thresholding $\sigma_\bs \le \xi^{1/q}$. Formally, this is similar to choosing all even indices $\bs$  corresponding to the  smallest values of $\sigma_\bs$ satisfying $\sigma_\bs \le \xi^{1/q}$. Then for a given $n \in \NN$, we selected a number $\xi_n$ such that the number  of quadrature points in the grid $\Gamma(\Lambda_\rev^*(\xi_n))$ does not exceed $n$. Hence, due to the evenness of the indices in  the set $\Lambda_\rev^*(\xi_n)$ we obtained 
 the improved rate $2/q - 1/2$ and that the sparsity of $\Lambda_\rev^*(\xi_n)$ is much higher then that of $\Lambda_{m(n)}$.
}
\end{remark}

\section{Elliptic PDEs with  affine inputs}
\label {affine inputs}

The theory of non-adaptive  approximation and integration of functions  in Bochner spaces with infinite tensor product Gaussian measure in Sections \ref{Linear Galerkin approximation}--\ref{Integration} can be generalized and extended to other situations.  In this section, we present some results on similar problems for  the parameterized diffusion elliptic equation \eqref{p-ellip} with the affine inputs \eqref{affine}. 

In the affine case, for given $a,b > -1$, we consider  the orthogonal Jacobi expansion of the solution $u(\by)$  of the form
\begin{equation} \nonumber
\sum_{\bs\in\FF} u_\bs J_\bs(\by), \quad J_\bs(\by)=\bigotimes_{j \in \NN}J_{s_j}(y_j),\quad u_\bs:=\int_\II u(\by)J_\bs(\by) \rd\nu_{a,b}(\by),
\end{equation}
where
\begin{equation} \nonumber
\rd \nu_{a,b}(\by):=\bigotimes_{j \in \NN} \delta_{a,b}(y_j)\,\rd y_j,
\end{equation}
\[
\delta_{a,b}(y):=c_{a,b}(1-y)^a(1+y)^b, \quad
c_{a,b}:=\frac{\Gamma(a+b+2)}{2^{a+b+1}\Gamma(a+1)\Gamma(b+1)},
\] 
and  $(J_k)_{k\geq 0}$ is the sequence of Jacobi polynomials on $\II := [-1,1]$ 
normalized with respect to the Jacobi probability measure
$
\int_{\II} |J_k(y)|^2 \delta_{a,b}(y) \rd y =1.
$
One has the Rodrigues' formula
\[
J_k(y ) 
\ = \
\frac{c_k^{a,b}}{k! 2^k}(1-y)^{-a}(1+y)^{-b} \frac{\rd^k}{\rd y^k} 
\left((y^2-1)^k(1-y)^a(1+y)^b\right),
\]
where $c_0^{a,b}:= 1$ and
\[
c_k^{a,b}
:= \
\sqrt{\frac{(2k+a+b+1)k! \Gamma(k+a+b+1) \Gamma(a+1) \Gamma(b+1)}
	{\Gamma(k+a+1)\Gamma(k+b+1)\Gamma(a+b+2)}}, \ k \in \NN.
\]
Examples corresponding to the values $a=b=0$  is the family of the Legendre polynomials, and to the values $a=b=-1/2$  the family of the Chebyshev polynomials.

We introduce the space
$W^r:=\{v\in V\; : \; \Delta v\in H^{r-2}(D)\}$
for $r\ge 2$ with the convention $W^1:= V$.
This space is equipped with the norm
$\|v\|_{W^r}:=\|\Delta v\|_{H^{r-2}(D)}$,
and coincides with the Sobolev space $V\cap H^{r-2}(D)$ with equivalent norms if the domain $D$ has $C^{r-1,1}$ smoothness,
see \cite[Theorem 2.5.1.1]{Gr85}. 
The following lemma has been proven in \cite{BCM17} for $r=1$ and in \cite{BCDS17} for $r > 1$.
\begin{lemma} \label{theorem[summability]J}	
	For a given $r \in \NN$,
	assume that $\bar a\in L^\infty(D)$  is such that ${\rm ess} \inf \bar a>0$,
	and that there exists a sequence $\brho_r=(\rho_{r;j}) _{j \in \NN}$ of positive numbers such that 
	\begin{equation} \nonumber
	\left \| \frac{\sum _{j \in \NN} \rho_{1;j}|\psi_j|}{\bar a} \right \|_{L^\infty(D)} 
	< 1.
	\end{equation} 	
	Assume that the right side $f$ in \eqref{ellip}
	belongs to $H^{r-2}(D)$, that the domain $D$ has $C^{r-2,1}$ smoothness,
	that $\bar a$ and all functions $\psi_j$ belong to $W^{r-1,\infty}(D)$ and that
	\begin{equation} \nonumber
	\sup_{|\alpha|\leq r-1} 
	\left\| \sum _{j \in \NN} \rho_{r;j} |D^\alpha \psi_j| \right\|_{L^\infty(D)} 
	<\infty.
	\end{equation}	
	Then 
		\begin{equation} \label{beta_r,s}
		\sum_{\bs\in\FF} (\sigma_{r;\bs} \|u_\bs\|_{W^r})^2 <\infty, \quad 
		\sigma_{r;\bs} := \brho_r^\bs\prod _{j \in \NN} c_{s_j}^{a,b}.
		\end{equation}	
\end{lemma}

\begin{lemma} \label{bcdmJ}
	Let $0 < q <\infty$, 	$\brho=(\rho_j) _{j \in \NN}$ of  numbers larger than 1 such the sequence $(\rho_j^{-1}) _{j \in \NN}$ belongs to $\ell_q(\NN)$, $(p_\bs(\theta,\lambda))_{\bs \in \FF}$ is a sequence of the form  \eqref{[p_s]} with arbitrary nonnegative $\theta, \lambda$. Then for every $\nu \in \NN_0$, we have
	\begin{equation} \nonumber
	\sum_{\bs \in \FF_\nu} p_\bs(\theta,\lambda) ( \rho^{-\bs})^{q/\nu} 
 \ < \ \infty.
	\end{equation}
\end{lemma}

\begin{proof}
	We have
	\begin{equation} \nonumber
	\sum_{\bs \in \FF_\nu} p_\bs(\theta,\lambda) ( \rho^{-\bs})^{q/\nu} =
	\prod_{j \in \NN}\ \sum_{s_j \in \NN_{0,\nu}} \ \rho_j^{-s_j q/\nu} (1 + \lambda s_j)^\theta =:  \prod_{j \in \NN} A_j.
	\end{equation}
	Since $\brho=(\rho_j) _{j \in \NN}$ of  numbers larger than one, and  such the sequence $(\rho_j^{-1}) _{j \in \NN}$ belongs to $\ell_q(\NN)$, we have $\min_{j \in \NN} \rho_j > 1$. Hence, there exists a constant $C$ independent of $j$ such that
	\begin{equation} \nonumber
	A_j =
	1 + \sum_{k = \nu}^\infty \ \rho_j^{-kq/\nu} (1 + \lambda k)^\theta \le 1 + C  \rho_j^{-q},
	\end{equation}
	and consequently,	
	\begin{equation} \nonumber
	\sum_{\bs \in \FF_\nu} p_\bs(\theta,\lambda) ( \rho^{-\bs})^{q/\nu}    \le
	\prod_{j \in \NN} (1 + C  \rho_j^{-q}) \le
	\exp\Big(\sum_{j \in \NN} C  \rho_j^{-q}\Big) < \infty.
	\end{equation}
	\hfill
\end{proof}

We assume that there holds  the following approximation property for $V$ and $W^r$ with $r > 1$.  

\noindent
{\bf Assumption III}  \,  
There are a sequence $(V_n)_{n \in \NN_0}$ of subspaces $V_n \subset V $ of dimension $\le n$, and a sequence $(P_n)_{n \in \NN_0}$ of linear operators from $V$ into $V_n$, 
and a number $\alpha>0$ such that
\begin{equation} \label{spatialappnJ}
\|P_n(v)\|_V \leq 
C\|v\|_V , \quad
\|v-P_n(v)\|_V \leq 
Cn^{-\alpha} \|v\|_{W^r}, \quad \forall n \in \NN_0, \quad \forall v \in W^r.
\end{equation}

In this section,  we make use the abbreviations: $\Ll_p(V) := L_p(\IIi,V,\nu_{a,b})$ and $\Ll_p(W^r) := L_p(\IIi,W^r,\nu_{a,b})$ and assume that $r > 1$. From Lemmata \ref{theorem[summability]J} and \ref{bcdmJ} we can prove the following results on non-adaptive fully and non-fully  discrete Jacobi GPC expansion and polynomial interpolation approximations and integration for the affine case.

\begin{theorem}\label{thm[L_2-approx]pdeJ}
	Let $0 < p \le 2$. Let Assumption III hold. Let the assumptions of Lemma~\ref{theorem[summability]J} hold  for the spaces $W^1=V$ and $W^r$
	with some $0<q_1\leq q_r <\infty$. For $\xi>1$, let $G(\xi)$ be the set defined by \eqref{G(xi)G} for $\sigma_{1;\bs}:= \beta_{1;\bs}$ and $\sigma_{2;\bs}:= \beta_{r;\bs}$ as in \eqref{beta_r,s}. Then there exists a constant $C$ such that
	for each $n \in \NN$, there exists a number $\xi_n$  such that   $\dim(\Vv(G(\xi_n)) \le n$ and
	\begin{equation} 	\label{L_2-rate-pdeJ}
	\|u-\Ss_{G(\xi_n)}u\|_{\Ll_p(V)} \leq C \times
	\begin{cases}
		n^{-\alpha} &\text{if } \alpha \leq  1/q_2,
		\\
		n^{-\beta}  &\text{if } \alpha >  1/q_2.
	\end{cases} 
	\end{equation}
	The rate $\alpha$ corresponds to the spatial approximation of a single function in $W^r$ as given by \eqref{spatialappnJ}, and
	the rate $\beta$ is given by \eqref{[beta]2}.
\end{theorem}

The rate $\min(\alpha,\beta)$ in \eqref{L_2-rate-pdeJ} is the same rate of fully discrete best adaptive  $n$-term approximation in $\Ll_2(V)$ based on $\ell_{p_1}$-summability of $(\|u_\bs\|_V)_{\bs \in \FF}$ and $\ell_{p_r}$-summability of $(\|u_\bs\|_{W^r})_{\bs \in \FF}$ proven in \cite{BCDS17}, where  $1/p_1 = 1/q_1 + 1/2$ and $1/p_r = 1/q_r + 1/2$. This rate can be achieved by linear fully discrete non-adaptive approximation when $(\|u_\bs\|_V)_{\bs \in \FF}$ and  $(\|u_\bs\|_{W^r})_{\bs \in \FF}$ have $\ell_{p_1}$-summable and  $\ell_{p_r}$-summable majorant sequences, respectively \cite{ZDS18}.

\begin{theorem}\label{thm[V_p-approx]pdeJ}
	Let $1 \le p \le \infty$. 
	Let Assumption III hold. Let the assumptions of Lemma~\ref{theorem[summability]J} hold  for the spaces $W^1=V$ and $W^r$
	with some $0<q_1\leq q_r <\infty$ with $q_1 <2$. For $\xi>1$, let $G(\xi)$ be the set defined by  in \eqref{G_(xi)} for $\sigma_{1;\bs}:= \beta_{1;\bs}$ and $\sigma_{2;\bs}:= \beta_{r;\bs}$ as in \eqref{beta_r,s}. 
	Then there exists a constant $C$ such that for each $n \in \NN$, there exists a number $\xi_n$  such that  $\dim(\Vv(G(\xi_n)) \le n$ and
	\begin{equation} \label{u-S_Gu-pdeJ}
	\|u-\Ss_{G(\xi_n)}u\|_{\Ll_p(V)} \leq C \times
	\begin{cases}
		n^{-\alpha} &\text{if } \alpha \leq  1/q_2 - 1/2,
		\\
		n^{-\beta} &\text{if } \alpha >  1/q_2 - 1/2.
	\end{cases} 
	\end{equation}
	The rate $\alpha$ corresponds to the spatial approximation of a single function in $W^r$ as given by \eqref{spatialappnJ}.
	The rate $\beta$ is given by \eqref{[beta]1}
\end{theorem}

For polynomial interpolation approximation and integration, we keep all definitions and notations in Section~\ref{interpolation} with a proper modification for the affine case. For example,  for univariate interpolation and integration we take  a sequence of points
$Y_m = (y_{m;k})_{k=0}^m$  in $\II$ such that 
\begin{equation} \nonumber
- \infty < y_{m;0} < \cdots < y_{m;m-1} < y_{m;m} < + \infty; \quad y_{0;0} = 0.
\end{equation}
Sequences of points  
$Y_m = (y_{m;k})_{k=0}^m$ satisfying  the inequality \eqref{ineq[lambda_m]}, are  the symmetric sequences of the Chebyshev points, the symmetric sequences of the Gauss-Lobatto (Clenshaw-Curtis) points and the nested sequence of the $\Re$-Leja points, see  \cite{CD15} for details. 

\begin{theorem}\label{thm[coll-approx]pdeJ}
		Let $1 \le p \le \infty$. 
	Let  Assumption III hold.  Let the assumptions of Lemma~\ref{theorem[summability]J} hold  for the spaces $W^1=V$ and $W^r$
	with some $0<q_1\leq q_r <\infty$ with $q_1 <2$. Assume that
	$(Y_m)_{m \in \NN_0}$ is a sequence satisfying the condition \eqref{ineq[lambda_m]} for some positive numbers $\tau$ and $C$. For $\xi>1$, let $G(\xi)$ be the set defined by \eqref{G_(xi)} for $\sigma_{1;\bs}:= \beta_{1;\bs}$ and $\sigma_{2;\bs}:= \beta_{r;\bs}$ as in \eqref{beta_r,s}. 
	 	Then there exists a constant $C$ such that for each $n \in \NN$, there exists a number $\xi_n$  such that  
for the operator $\Ii_{G(\xi_n)}: \Ll_2(W^r) \to \Vv(G(\xi_n)$, we have that $\dim \Vv(G(\xi_n)) \le  n$ and 
	\begin{equation} \label{u-I_Gu, p le 2-pdeJ}
	\|u -\Ii_{G(\xi_n)}u\|_{\Ll_p(V)} \leq C \times
	\begin{cases}
		n^{-\alpha} &\text{if } \alpha \leq  1/q_2 - 1/2,
		\\
		n^{-\beta} \log n  &\text{if } \alpha >  1/q_2 - 1/2.
	\end{cases} 
	\end{equation}
	The rate $\alpha$ corresponds to the spatial approximation of a single function in $W^r$ as given by \eqref{spatialappnJ}.
	The rate $\beta$ is given by \eqref{[beta]1}.
\end{theorem}

The rates in \eqref{L_2-rate-pdeJ}--\eqref{u-I_Gu, p le 2-pdeJ} for some non-adaptive approximations have been proven in the case when  $(\|u_\bs\|_V)_{\bs \in \FF}$ and $(\|u_\bs\|_{W^r})_{\bs \in \FF}$ have $\ell_{p_1}$-summable  and  $\ell_{p_r}$-summable majorant sequences, respectively, which are derived from the analyticity of the solution $u$, where  $1/p_1 = 1/q_1 + 1/2$ and $1/p_r = 1/q_r + 1/2$, see \cite{ZDS18}.

\begin{theorem}\label{thm[quadrature]pdeJ}
	Let  Assumption III hold.  Let $a=b$ for the Jacobi probability measure $\nu_{a,b}(\by)$, and the assumptions of Lemma~\ref{theorem[summability]J} hold  for the spaces $W^1=V$ and $W^r$
	with some $0<q_1\leq q_r <\infty$ with $q_1 < 4$. 
	Assume that
	$(Y_m)_{m \in \NN_0}$ is a sequence satisfying the condition \eqref{ineq[lambda_m]} for some positive numbers $\tau$ and $C$, and such that $Y_m$ is symmetric for every $m \in \NN_0$. For $\xi>1$, let $G_\rev^*(\xi)$ be the set defined  by \eqref{G_rev(xi)2} for $q_2=q_r$ and for $\sigma_{1;\bs}:= \beta_{1;\bs}$ and $\sigma_{2;\bs}:= \beta_{r;\bs}$ as in \eqref{beta_r,s}. 	Then for  the quadrature operator $\Qq_{G_\rev^*(\xi)}$ generated by the interpolation operator $\Ii^*_{G_\rev^*(\xi)}: \Ll_2(W^r) \to \Vv(G_\rev^*(\xi))$, we have the following.
	\begin{itemize}
		\item[{\rm (i)}]
		There exists a constant $C$ such that for each $n \in \NN$, there exists a number $\xi_n$  such that  $\dim\Vv(G_\rev^*(\xi_n))\le n$ and 
\begin{equation} \label{u-Q_Gu-quadratureJ}
\left\|\int_{\IIi}u(\by)\, \rd \nu_{a,b}(\by) - \Qq_{G_\rev^*(\xi_n)}u\right\|_V \leq C \times
	\begin{cases}
		n^{-\alpha} &\text{if } \alpha \leq  2/q_r - 1/2,
		\\
		n^{-\beta} \log n  &\text{if } \alpha >  2/q_r - 1/2.
	\end{cases} 
\end{equation}			
		\item[{\rm (ii)}] 
Let $\phi \in V'$  be a bounded linear functional on $V$. 	Then there exists a constant $C$ such that for  each $n \in \NN$, there exists a number $\xi_n$  such that  $\dim\Vv(G_\rev^*(\xi_n))\le n$ and 
\begin{equation} \label{u-Q_Gu_phiJ}
\left|\int_{\IIi} \langle \phi,  u(\by) \rangle\, \rd \nu_{a,b}(\by) - \Qq_{G_\rev^*(\xi_n)} \langle \phi,  u \rangle\right| \leq C \|\phi\|_{V'} \times
	\begin{cases}
		n^{-\alpha} &\text{if } \alpha \leq  2/q_r - 1/2,
		\\
		n^{-\beta} \log n &\text{if } \alpha >  2/q_r - 1/2.
	\end{cases} 
\end{equation}
	\end{itemize}	
	The rate $\alpha$ corresponds to the spatial approximation of a single function in $W^r$ as given by
	\eqref{spatialappnJ}. The rate $\beta$   is  given by 
	\begin{equation} 	\nonumber
	\beta := \left(\frac 2 {q_1} - \frac{1}{2}\right)\frac{\alpha}{\alpha + \delta}, \quad 
	\delta := \frac 2 {q_1} - \frac 2 {q_r}.
	\end{equation}	 
\end{theorem}

The rate in \eqref{u-Q_Gu-quadratureJ}--\eqref{u-Q_Gu_phiJ} can be improved as 
$\min(\alpha,  \frac 2 {q_1} \frac{\alpha}{\alpha + \delta})$ if   $(\|u_\bs\|_V)_{\bs \in \FF}$ and $(\|u_\bs\|_{W^r})_{\bs \in \FF}$ have $\ell_{p_1}$- and  $\ell_{p_r}$-summable majorant sequences, respectively, where  $1/p_1 = 1/q_1 + 1/2$ and $1/p_r = 1/q_r + 1/2$, see \cite{ZDS18}.

\begin{corollary}\label{corollary[quadrature]pdeJ}
Let $a=b$ for the Jacobi probability measure $\nu_{a,b}(\by)$, and the assumptions of Lemma~\ref{theorem[summability]J}  hold for the spaces $W^1=V$ with some $0 < q_1 = q  < 4$.  Assume that
$(Y_m)_{m \in \NN_0}$ is a sequence satisfying the condition \eqref{ineq[lambda_m]} for some positive numbers $\tau$ and $C$, and such that $Y_m$ is symmetric for every $m \in \NN_0$.
 For $\xi>1$, let $\Lambda_\rev^*(\xi)$ be the set defined by \eqref{Lambda_rev(xi)2} for $\sigma_\bs:= \beta_{1;\bs}$ as in \eqref{beta_r,s}. Then we have the following.
			\begin{itemize}
		\item[{\rm (i)}]	
	There exists a constant $C$ such that for each $n \in \NN$, there exists a number $\xi_n$  such that  $|\Gamma(\Lambda_\rev^*(\xi_n))| \le n$ and 
\begin{equation} \label{u-Q_Lambdau_pdeJ}
\left\|\int_{\IIi}u(\by)\, \rd \nu_{a,b}(\by) - Q_{\Lambda_\rev^*(\xi_n)}u\right\|_V
\ \le \
Cn^{-(2/q - 1/2)}.
\end{equation}
	\item[{\rm (ii)}] 
		Let  $\phi \in V'$ be a bounded linear functional on $V$. 
		Then there exists a constant $C$ such that for each  $n \in \NN$, there exists a number $\xi_n$  such that  $|\Gamma(\Lambda_\rev^*(\xi_n))| \le n$ and 
	\begin{equation} \label{u-Q_Lambdau_phi-pdeJ}
	\left|\int_{\IIi} \langle \phi, u (\by) \rangle\, \, \rd \nu_{a,b}(\by) - Q_{\Lambda_\rev^*(\xi_n)} \langle \phi, u \rangle \right|
	\ \le \
	C\|\phi\|_{V'} n^{-(2/q - 1/2)}.
	\end{equation}	
	\end{itemize}	
\end{corollary}

 The rate $2/q - 1/2$ in \eqref{u-Q_Lambdau_pdeJ} in Corollary~\ref{corollary[quadrature]pdeJ} improves the rate $2/q- 1/2- \varepsilon$ with arbitrary $\varepsilon >0$, which has been obtained in \cite[Corollary 3.13]{ZS17}. 
 
 We can also prove counterparts of Theorem \ref{thm[quadrature]pde2} and Corollary \ref{corollary[quadrature]pde2} for the parameterized diffusion elliptic equation \eqref{p-ellip} with the affine inputs \eqref{affine}.

\medskip
\noindent
{\bf Acknowledgments.}  This work is funded by Vietnam National Foundation for Science and Technology Development (NAFOSTED) under  Grant No. 102.01-2020.03.  It was partially supported by a grant from the Simon Foundation and  EPSRC Grant Number EP/R014604/1. The author would like to thank the Isaac Newton Institute for Mathematical Sciences for partial support and hospitality during the program Approximation, sampling and compression in data science when work on this paper was partially undertaken.   He expresses special thanks to Christoph Schwab, Nguyen Van Kien and Jacob Zech for valuable remarks, comments and suggestions.

\end{document}